\documentclass[11pt,a4paper,oneside]{amsart}
\usepackage{amssymb,amsmath,mathrsfs,amsthm,amsrefs,graphicx,color}
\usepackage{leftidx}

\theoremstyle{plain}
\newtheorem{theorem}{Theorem}[section]
\newtheorem{definition}[theorem]{Definition}
\newtheorem{lemma}[theorem]{Lemma}

\newtheorem{proposition}[theorem]{Proposition}

\theoremstyle{remark}
\newtheorem{remark}[theorem]{Remark}

\allowdisplaybreaks[4]

\usepackage{hyperref}

\numberwithin{equation}{section}

\makeatletter
\@namedef{subjclassname@2020}{\textup{2020} Mathematics Subject Classification}
\makeatother

\newcommand{\C}{\mathbb{C}}
\newcommand{\R}{\mathbb{R}}
\newcommand{\Z}{\mathbb{Z}}

\renewcommand{\Im}{\operatorname{Im}}
\renewcommand{\Re}{\operatorname{Re}}
\newcommand{\I}{\infty}
\newcommand{\abs}[1]{\left\lvert #1\right\rvert}

\newcommand{\norm}[1]{\left\lVert #1\right\rVert}

\newcommand{\Jbr}[1]{\left\langle #1 \right\rangle}

\newcommand{\IN}{\quad\text{in }}

\def\({\left(}
\def\){\right)}
\def\<{\left\langle}
\def\>{\right\rangle}
\def\le{\leqslant}
\def\ge{\geqslant}

\def \l{\lambda}

\def \s{\sigma}
\def \a{\alpha}

\newcommand{\eps}{\varepsilon}

\DeclareMathOperator{\sign}{sign}
\DeclareMathOperator{\sn}{sn}
\DeclareMathOperator{\cn}{cn}
\DeclareMathOperator{\dn}{dn}
\DeclareMathOperator{\cd}{cd}
\DeclareMathOperator{\sd}{sd}
\DeclareMathOperator{\nd}{nd}
\DeclareMathOperator{\am}{am}
\DeclareMathOperator{\sech}{sech}

\DeclareMathOperator{\tr}{tr}

\newcommand{\todayd}{\the\year/\the\month/\the\day}

\theoremstyle{definition}

\newcommand{\ol}{\overline}

\begin{document}
\title[Large-time behavior of cubic NLS systems]{Partial classification of the large-time behavior of solutions to 
cubic nonlinear Schr\"odinger systems}

\author[S. Masaki]{Satoshi MASAKI}
\address{Department of mathematics, 
Hokkaido University, Sapporo Hokkaido, 060-0810, Japan}
\email{masaki@math.sci.hokudai.ac.jp}
%
%


\begin{abstract}
In this paper, we study the large-time behavior of small solutions to the standard form of the systems of 1D cubic nonlinear Schr\"odinger equations consisting of two components and possessing a coercive mass-like conserved quantity.
The cubic nonlinearity is known to be critical in one space dimension in view of the large-time behavior.
By employing the result by Katayama and Sakoda, one can obtain the large-time behavior of the solution if we can integrate the corresponding ODE system.
We introduce an integration scheme suited to the system. The key idea is to rewrite the ODE system, which is cubic, as a quadratic system of quadratic quantities of the original unknown.
By using this technique, we described the large-time behavior of solutions in terms of elementary functions and the Jacobi elliptic functions for several examples of standard systems.
\end{abstract}

\maketitle

\section{Introduction}

In this paper, we consider the large-time behavior of solutions to the following system of the cubic nonlinear Schr\"odinger equations
\begin{equation}\label{E:NLS}
	(i \partial_t  + \partial_x^2)u_j =F_j(u_1,u_2), \quad (t,x) \in \R^{1+1}, \quad j=1,2,
\end{equation}
where $(u_1,u_2)$ is a $\C^2$-valued unknown and the nonlinearities are given by
\begin{equation}\label{E:nondef}
\left\{
\begin{aligned}
	F_1(u_1,u_2):={}&
(3p_2 +p_3+2p_4) |u_1|^2 u_1 +(p_1 + p_5)(2|u_1|^2 u_2 + u_1^2\ol{u_2})\\
&+(p_2-p_3) (2 u_1 |u_2|^2 + \ol{u_1}u_2^2) - (p_1- p_5) |u_2|^2u_2
	\\& - 4p_1 \Re (\overline{u_1} u_2) u_1 +\mathcal{V}(u_1,u_2) u_1, \\
	F_2(u_1,u_2):={}&
(p_1+p_5) |u_1|^2 u_1+(p_2-p_3)(2|u_1|^2 u_2 + u_1^2\ol{u_2})\\&-(p_1-p_5) (2 u_1 |u_2|^2 + \ol{u_1}u_2^2) + (3p_2+p_3-2p_4) |u_2|^2u_2
	\\& + 4p_1 \Re (\overline{u_1} u_2) u_2 + \mathcal{V}(u_1,u_2)u_2,
\end{aligned}
\right.
\end{equation}
with
$p_1 \ge 0$, $p_2\in\R$, $p_3 \ge 0$, $p_4\in\R$, and $p_5 \ge0$.
We exclude the trivial case $p_1=p_2=p_3=p_4=p_5=0$.
$\mathcal{V}$ is a real-valued quadratic form given by
\[
	\mathcal{V} (u_1,u_2) = q_1 |u_1|^2 + 2 q_2 \Re (\overline{v_1} v_2) + q_3 |u_2|^2
	= \overline{\begin{bmatrix} u_1 & u_2\end{bmatrix}} \begin{bmatrix} q_1 & q_2 \\ q_2 & q_3 \end{bmatrix}
	\begin{bmatrix} u_1 \\ u_2\end{bmatrix}
\]
with $q_1,q_2,q_3 \in \R$.
It is a standard form of a system which has a coercive mass-like conserved quantity.
We consider \eqref{E:NLS} under the initial condition
\begin{equation}\label{E:CD}
	(u_1,u_2)(0) = (u_{0,1} , u_{0,2})\in H^{3,1} (\R) \times H^{3,1} (\R)
\end{equation}
and specify large time behavior of solution under the smallness condition.
For any choice of parameters, the above system admits the conserved mass given by
\begin{equation}\label{E:stdmass}
	M(u_1,u_2) = \int_{\R} \tfrac12 (|u_1|^2 + |u_2|^2) dx.
\end{equation}
This is because the following identity holds:
\begin{equation}\label{E:nullcond}
	\Im (\overline{z_1} F_1(z_1,z_2)+\overline{z_2}F_2(z_1,z_2))= 0 
\end{equation}
for all $(z_1,z_2) \in \C^2$. 
The classification of the two-component systems of cubic nonlinear Schrodinger equation is discussed in \cite{MSU2,M} (see also reference therein).
The classification shows that \eqref{E:NLS} is a standard form of a system which has a coercive mass-like conserved quantity, i.e., for which there exists
$(a,b,c)\in \R^3$ with $ac>b^2$ such that
\[
	\Im \(\begin{bmatrix}\overline{z_1} & \overline{z_2} \end{bmatrix}
	\begin{bmatrix} a & b \\ b & c \end{bmatrix}
	\begin{bmatrix} F_1(z_1,z_2) \\ F_2(z_1,z_2) \end{bmatrix}
	\)= 0
\]
for all $(z_1,z_2) \in \C^2$
(see Section \ref{S:classification}).
It contains a physical model such as Manakov equation.

\subsection{Reduction to the analysis of an ODE system}

It is well known that the cubic nonlinearity is critical in one dimension  when considering the long-time behavior of small solutions to NLS equations.
Ozawa \cite{Oz} showed that the asymptotic profile of a small solution
 includes a logarithmic phase correction due to the presence of the nonlinearity (see also \cite{GO,HN,KP,IT,Mu}).
The asymptotic profile is described with a solution to an ordinary differential equation which is systematically derived from the cubic NLS equation.
More precisely, a small solution to the cubic NLS equation
\[
	(i\partial_t + \partial_x^2) u = \lambda |u|^2 u, \quad (t,x) \in \R^{1+1}
\]
with $\lambda \in\R$ is approximated by 
\begin{equation}\label{E:singlebehavior}
	u_{\mathrm{app}}(t,x) := 
	(2it)^{-\frac{1}2} e^{i\frac{x^2}{4t}} \widehat{u_+} \(\frac{x}{2t}\) \exp \(-i\frac{\lambda }{2} \left|\widehat{u_+} \(\frac{x}{2t}\)\right|^2\log t \)
\end{equation}
with a suitable function $u_+$ as $t\to\infty$ (see \cite{Oz,GO,HN}). Note that
if we introduce a function
\[
	A(\tau,\xi) = 
	\widehat{u_+} (\xi) \exp \(-i \lambda \tau |\widehat{u_+}(\xi)|^2  \)
\]
then the profile is written as
\[
	u_{\mathrm{app}}(t,x) = (2it)^{-\frac{1}2} e^{i\frac{x^2}{4t}} A\( \frac{1}{2}\log t, \frac{x}{2t} \)
\]
One sees that $A(\tau)=A(\tau,\xi)$ is a one-parameter family of solutions to the ordinary differential equation
\[
	iA' = \lambda |A|^2 A, \quad A(0,\xi) = \widehat{u_+}(\xi).
\]
This ODE is obtained by removing $\partial_x^2$ from the cubic NLS system.
In \cite{KS}, Katayama and Sakoda investigate
a broad class of systems possessing a coercive mass-like conserved quantity
and prove that the asymptotic profile for a small solution is obtained  in this way
 (also refer to \cite{HNS,LS,LS2,KN,NST}).

This result is also applicable to our model \eqref{E:NLS}
since it possesses \eqref{E:stdmass} as a conserved quantity.
To state the application rigorously,
we make notations. $H^{s,k}$ stands for the  weighted Sobolev space:
\[
	H^{s,k}= H^{s,k}(\R) := \{ f \in \mathcal{S}'(\R) \ |\ \norm{f}_{H^{s,k}} <\I \}	, \quad
	\norm{f}_{H^{s,k}} := \sum_{\sigma =0}^k \norm{\langle \cdot \rangle^{\sigma} f }_{H^{s-\sigma}},
\]
where $\langle x \rangle^{\sigma}=(1+|x|^2)^{1/2}$ and $H^s$ is the standard Sobolev space.
\begin{theorem}[\cite{KS}]\label{T:KS}
There exists $\delta_0\in (0,1/4)$ such that for any $\delta \in (0,\delta_0)$ there exists $\eps_0>0$
such that if $\eps:=\|u_{0,1}\|_{H^{3,1}}+\|u_{0,2}\|_{H^{3,1}}$
satisfies $\eps<\eps_0$
then there exist $\C^2$-valued functions $(\alpha_1^\pm,\alpha_2^\pm) \in L^\I \cap H^{0,1}$ satisfying
\begin{equation}\label{E:alphadecay}
	|\alpha_1^\pm (\xi)| + |\alpha_2^\pm (\xi)| \lesssim \Jbr{\xi}^{-2}
\end{equation}
 such that \eqref{E:NLS} with \eqref{E:CD} 
 admits a unique global solution $(u_1,u_2) \in C (\R, H^{3,1})^2$
satisfying
\begin{equation}\label{E:main_asymptotics}
	\partial_x^\beta u_j (t,x) = \(\frac{i x}{2t}\)^\beta u_{\mathrm{app},j}^{\pm}(t,x) + O(\eps t^{-\frac{3}4 + \delta}) \IN L^\I_{x}(\R)
\end{equation}
as $t\to\pm \I$ for $\beta \le 2$ with the asymptotic profile
\begin{equation}\label{E:main_uapp}
	u_{\mathrm{app},j}^{\pm} (t,x)=
	(2it)^{-\frac{1}2} e^{i\frac{x^2}{4t}} A_j^\pm\(\frac{t}{2|t|}\log| t|, \frac{x}{2t}\),
\end{equation}
where,
 for each fixed $\xi \in \R$, $(A_1^\pm (\cdot,\xi),A_2^\pm (\cdot,\xi)) \in C^\I (\R,\C^2)$ is a solution to the ODE system 
\begin{equation}\label{E:ODE}
	iA_1' = F_1(A_1,A_2), \quad iA_2' = F_2(A_1,A_2)
\end{equation}
subject to a data $(A_1^\pm (0,\xi),A_2^\pm (0,\xi))=(\alpha_1^\pm (\xi),\alpha_2^\pm (\xi))$. Here, $F_1$ and $F_2$ are given as in \eqref{E:nondef}.
\end{theorem}
\begin{remark}
The global existence of a solution to \eqref{E:NLS} in this framework follows from \cite{LS}.
\end{remark}

We emphasize that the theorem 
provides the asymptotic profile of the solution to \eqref{E:NLS} in an implicit manner, and to obtain an explicit form of the profile as in \eqref{E:singlebehavior}, we must derive explicit formulas for the solutions of the ODE system \eqref{E:ODE}.
In \cite{KS}, several applications are demonstrated.
However, even if we restrict ourselves to the system \eqref{E:NLS}, the explicit integrability of the corresponding ODE system \eqref{E:ODE} was left open for many cases.
In this paper, we present several examples of novel asymptotic behaviors by solving the ODE system \eqref{E:ODE} (almost) explicitly.
It will turn out that the system \eqref{E:NLS} exhibits various behaviors based on different parameter combinations.

\begin{remark}
As previously mentioned, for any choice of parameters $p_j$ and $q_j$, the system \eqref{E:NLS} possesses \eqref{E:stdmass} as a conserved quantity. 
Recently, the large-time behavior of solutions is extensively studied also for dissipative systems, that is, systems for which \eqref{E:stdmass} decrease in time (\cite{LNSS1,LNSS2}). 
Further details can be found in \cite{LNSS5} and references therein.
\end{remark}

\subsection{An integration scheme for \eqref{E:ODE}.
}

To find an explicit representation of solutions to the ODE system \eqref{E:ODE}, we
introduce an integration scheme well adapted to  the system. 
The crucial point of the scheme is introducing the intermediate step to obtain an explicit formula 
 for quadratic quantities of the unknowns.

For a solution $(A_1,A_2)$ to \eqref{E:ODE},
let us introduce
\begin{equation}\label{D:rhoDRI}
	\rho := |A_1|^2 + |A_2|^2 ,\quad \mathcal{D} := |A_1|^2 - |A_2|^2, \quad
	\mathcal{R} := 2 \Re (\ol{A_1}A_2), \quad
	\mathcal{I} : = 2\Im (\ol{A_1} A_2).
\end{equation}
The following identity is useful:
\begin{equation}\label{E:qqs}
	\rho^2 = \mathcal{D}^2 + \mathcal{R}^2 + \mathcal{I}^2.
\end{equation}
Thanks to the identity \eqref{E:nullcond},
$\rho$ becomes a conserved quantity of \eqref{E:ODE} for any choice of parameters.
The conservation of $\rho$ implies that the all solution to \eqref{E:ODE} exists globally and
that $(\mathcal{D},\mathcal{R},\mathcal{I})$ takes value on the sphere $$S^2_\rho := \{ (x,y,z) \in \R^3 \ |\ x^2+y^2+z^2 = \rho^2 \}$$ for $\rho>0$.

Note that the map
\[
	\C^2 \ni (A_1,A_2) \mapsto (\rho,\mathcal{D}, \mathcal{R}, \mathcal{I}) \in \R^4
\]
is not invertible. This can be seen by the fact that the change $(A_1,A_2)\mapsto (e^{i\theta}A_1,e^{i\theta}A_2)$ leaves $ (\rho,\mathcal{D}, \mathcal{R}, \mathcal{I})$ unchanged.
This suggests that the reconstruction of $(A_1,A_2)$ from the quadratic quantities
is not trivial.
Nevertheless, we can reproduce the original solution from the quadratic quantities with the help of the ODE system \eqref{E:ODE}.

\begin{theorem}[reconstruction of solution from its quadratic quantities]\label{T:main}
Let $(A_1(\tau),A_2(\tau))$ be a nontrivial global solution to \eqref{E:ODE}.
Suppose that explicit formulas for $\mathcal{D}(\tau)$, $\mathcal{R}(\tau)$, and $\mathcal{I}(\tau)$
are obtained.
If $A_1(0) \neq0$ then
one has
\begin{equation}\label{E:solformula1}
\left\{
\begin{aligned}
	A_1 (\tau) ={}& (-1)^{k_1(\tau)} \sqrt{\tfrac{\rho + \mathcal{D(\tau)}}{2}} \tfrac{A_1 (0)}{|A_1 (0)|} \exp \( i\int_0^\tau (N_1 (\tilde\tau)  - \mathcal{V}(A_1(\tilde\tau),A_2(\tilde\tau)))d\tilde\tau\) , \\
	A_2 (\tau) ={}& (-1)^{k_1(\tau)}  \tfrac{\mathcal{R}(\tau)+i \mathcal{I}(\tau)}{\sqrt{2(\rho + \mathcal{D}(\tau))}} \tfrac{A_1 (0)}{|A_1 (0)|} \exp \( i\int_0^\tau (N_1 (\tilde\tau)  - \mathcal{V}(A_1(\tau),A_2(\tau)))d\tilde\tau \) 
\end{aligned}
\right.
\end{equation}
for all $\tau \in \R$,
where
\[
	k_1 (\tau) := \begin{cases}
	\# (\{ s \in \R \ |\ \rho + \mathcal{D}(s) =0 \} \cap [0,\tau]) & \tau\ge 0, \\
	\# (\{ s \in \R \ |\ \rho + \mathcal{D}(s) =0 \} \cap [\tau,0]) & \tau<0 
	\end{cases}
\]
and 
\[
	N_1: = \tfrac{\rho \mathcal{R}}{\rho+\mathcal{D}}p_1 +  (-3 \rho + \tfrac{\mathcal{I}^2}{\rho + \mathcal{D}})p_2
	+( -  \mathcal{D} + \tfrac{  \mathcal{R}^2}{\rho + \mathcal{D}})p_3 - (\rho + \mathcal{D})p_4 
	+ (-  \mathcal{R} - \tfrac{\rho \mathcal{R}}{\rho+ \mathcal{D}})p_5.
\]
If $A_2(0)\neq0$ then one has
\begin{equation}\label{E:solformula2}
\left\{
\begin{aligned}
	A_1 (\tau) ={}& (-1)^{k_2(\tau)}\tfrac{A_2 (0)}{|A_2 (0)|} \tfrac{\mathcal{R}(\tau)-i \mathcal{I}(\tau)}{\sqrt{2(\rho - \mathcal{D}(\tau))}}   \exp \( i\int_0^\tau (N_2 (\tilde\tau)  - \mathcal{V}(A_1(\tilde\tau),A_2(\tilde\tau)))d\tilde\tau\) ,\\
	A_2 (\tau) ={}& (-1)^{k_2(\tau)}\tfrac{A_2 (0)}{|A_2 (0)|} \sqrt{\tfrac{\rho - \mathcal{D(\tau)}}{2}}  \exp \( i\int_0^\tau (N_2 (\tilde\tau)  - \mathcal{V}(A_1(\tilde\tau),A_2(\tilde\tau)))d\tilde\tau\) 
\end{aligned}
\right.
\end{equation}
for all $\tau \in \R$,
where
\[
	k_2 (\tau) := \begin{cases}
	\# (\{ s \in \R \ |\ \rho - \mathcal{D}(s) =0 \} \cap [0,\tau]) & \tau\ge 0, \\
	\# (\{ s \in \R \ |\ \rho - \mathcal{D}(s) =0 \} \cap [\tau,0]) & \tau<0
	\end{cases}
\]
and
\[
	N_2:=   - \tfrac{\rho \mathcal{R}}{\rho- \mathcal{D}}  p_1
	+ (-3\rho  + \tfrac{\mathcal{I}^2}{\rho - \mathcal{D}} )p_2 
	+( \mathcal{D} + \tfrac{\mathcal{R}^2}{\rho - \mathcal{D}})p_3
	+ (\rho - \mathcal{D})p_4
	+ (- \mathcal{R} - \tfrac{\rho \mathcal{R}}{\rho- \mathcal{D} }) p_5.
\]
\end{theorem}

\begin{remark}
When $p_1=p_5$, $A_1(0)\neq0$ implies  $A_1$ has no zero points.
In the other case, if $\{A_1 =0\}=\{ \rho + \mathcal{D}=0\}$ is nonempty then
the formula \eqref{E:solformula1} has singular points.
However, even in such a case, all points in the set are isolated and removable singularities.
As a consequence, $k_1(\tau)$ and $k_2(\tau)$ are finite for all $\tau \in \R$.
Similar applies to the formula \eqref{E:solformula2}.
We  also remark that $\mathcal{V}(A_1,A_2)=\frac{q_1+q_3}2\rho + \frac{q_1-q_3}2 \mathcal{D} + q_2 \mathcal{R}$
holds and that it is also given by $\rho, \mathcal{D}, \mathcal{R}$.
\end{remark}

We remark that the formulas
\eqref{E:solformula1} and \eqref{E:solformula2} 
is almost explicit.
Although the formulas still involve a definite integral (of an explicit function) in the phase part, which is not always explicitly integrable, these formulas give us precise information of the solution.
One verifies that in some cases the integral is also calculated explicitly.

\subsection{A quadratic ODE system for the quadratic quantities}
In view of Theorems \ref{T:KS} and \ref{T:main},
we set our goal as obtaining explicit formulas for the triplet $(\mathcal{D},\mathcal{R},\mathcal{I})$.
It turns out that the triplet solves the following 
quadratic ODE system:
\begin{equation}\label{E:qqq}
\left\{
\begin{aligned}
\begin{bmatrix}\mathcal{D}\\\mathcal{R}\end{bmatrix}' ={}&2\mathcal{I}
\begin{bmatrix}
p_1 & p_2-p_3 \\ - p_2 - p_3 & p_1 
\end{bmatrix}
\begin{bmatrix}\mathcal{D}\\\mathcal{R}\end{bmatrix} + 2 \rho \mathcal{I} \begin{bmatrix} p_5 \\ - p_4 \end{bmatrix},\\
\mathcal{I}' = {}& -2p_1 (\mathcal{D}^2+\mathcal{R}^2) + 4 p_3\mathcal{D} \mathcal{R} + 2 \rho (-p_5 \mathcal{D}  + p_4 \mathcal{R} )
\end{aligned}
\right.
\end{equation}
(see Proposition \ref{P:qqq}, below).
Hence, it suffices to obtain an explicit formula of the solutions to this system.

Since $(\mathcal{D},\mathcal{R},\mathcal{I})$ takes value on a sphere,
no chaotic behavior appears
(cf. Poincar\'e-Bendixson theorem).
Further, if $(\mathcal{D},\mathcal{R},\mathcal{I}) \in S^2_\rho$ is a solution to \eqref{E:qqq} with $\rho=\rho_1>0$ 
then for any $\rho_2>0$
\begin{equation}\label{E:ODEscale}
(\tfrac{\rho_2}{\rho_1} \mathcal{D}(\tfrac{\rho_2}{\rho_1} \cdot),\tfrac{\rho_2}{\rho_1} \mathcal{R}(\tfrac{\rho_2}{\rho_1} \cdot),\tfrac{\rho_2}{\rho_1} \mathcal{I}(\tfrac{\rho_2}{\rho_1} \cdot)) \in S_{\rho_2}^2
\end{equation}
 is also a solution to \eqref{E:qqq} with $\rho=\rho_2$.
Hence, the behavior of solutions to \eqref{E:qqq} is essentially independent of the radius $\rho>0$.
Our main aim here is to integrate the system \eqref{E:qqq} for a class of combination of parameters.
However, due to these properties, 
the phase-portrait analysis works well for a wider class.

Let us introduce one notion characterized by the flow given by \eqref{E:qqq}.
\begin{definition}\label{D:ns}
Fix $\rho>0$.
For given combination of parameters,
we say nonlinear synchronization occurs for positive time direction
if the flow on $S_\rho^2$ given by \eqref{E:qqq} admits finitely many fixed points and there exists an asymptotically stable fixed point, say $p_\infty$, such that 
the following properties hold:
\begin{enumerate}
\item (Strong attraction property)
$\omega (x_0) = \{p_\infty\}$ holds for all
non-equilibrium point $x_0 \in S_\rho^2$, where $\omega (x_0)\subset S_\rho^2$ is the set of all $\omega$-limit points of $x_0$;
\item (Lyapunov stability) For any
open neighborhood $U \subset S^2_\rho$ of $p_\infty$
there exists an open neighborhood $V \subset S^2_\rho$ of $p_\infty$ such that if
$(\mathcal{D}(0),\mathcal{R}(0),\mathcal{I}(0)) \in V$ then $(\mathcal{D}(\tau),\mathcal{R}(\tau),\mathcal{I}(\tau)) \in U$
holds for all $\tau\ge0$.
\end{enumerate}
\end{definition}
\begin{remark}
The validity of the above two properties is stronger than the asymptotic stability of a fixed point $p_\infty$.
The asymptotic stability is the validity of the second property
and a local version of the attraction property;
$\omega (x_0) = \{p_\infty\}$ holds in a neighborhood of $p_\infty$.
One easily verifies that,
as for \eqref{E:qqq},
one sufficient condition 
for the asymptotic stability of a fixed point $p_\infty \in S_\rho^2$ is
$\ltrans{{\bf v}} \mathrm{H}(p_\infty) {\bf v} <0$ for all $ {\bf v} \in T_{p_\infty} S_\rho^2$,
where \[
\mathrm{H}(\mathcal{D},\mathcal{R},\mathcal{I})
=\begin{bmatrix}
2p_1 \mathcal{I} &2(p_2-p_3) \mathcal{I} & 2p_1 \mathcal{D} + 2(p_2-p_3)\mathcal{R} + 2p_5\rho \\
-2(p_2+p_3) \mathcal{I}
& 2p_1 \mathcal{I} & -2(p_2+p_3) \mathcal{D}+ 2p_1 \mathcal{R} - 2p_4\rho  \\
-4p_1 \mathcal{D} + 4p_3 \mathcal{R} -2p_5 \rho& -4 p_1 \mathcal{R} + 4p_3 \mathcal{D} + 2p_4 \rho& 0
\end{bmatrix}
.
\]
\end{remark}
Notice that, due to the above scale property, if the definition of the nonlinear synchronization is fulfilled for some $\rho_0>0$ then we have the same conclusion for any $\rho>0$.

\begin{theorem}\label{P:NS}
If the nonlinear synchronization occurs then
there exists a pair $(\gamma_{1} ,\gamma_{2}) \in \C^2 \setminus \{ (0,0) \}$ such that
 the following-type asymptotics for solutions to \eqref{E:NLS} holds true in addition to the conclusion of Theorem  \ref{T:main}:
Let $p_\infty \in S^2$ be the asymptotically stable fixed point
and let $\mathfrak{P}=\{p_n\}_{n=1}^N\subset S^2$ be the set of other fixed points of \eqref{E:qqq} with $\rho=1$
 given by Definition \ref{D:ns}.
Then,  for any closed set 
$\mathfrak{E} \subset  S^2 \setminus \mathfrak{P}$
\[
	t^{\frac12} \| ( \gamma_1 u_1 + \gamma_2 u_2 )(t,2t \cdot) \|_{L^\infty (\Omega(\mathfrak{E}))} \to 0
\]
as $t\to\infty$, where
\[
\Omega(\mathfrak{E}) =  \{  \xi \in \R \ |\ (|\alpha_1^+(\xi)|^2-|\alpha_2^+(\xi)|^2, 2\Re \overline{\alpha_1^+(\xi)}\alpha_2^+(\xi),2\Im \overline{\alpha_1^+(\xi)}\alpha_2^+(\xi)) \in \{k \mathfrak{E}\in \R^3; k \ge 0\}   \}.
\]
\end{theorem}
%
%

\subsection{Overview of the results}

Our underlying interest is 
to comprehend the role of the components represented by the parameters $p_j$. To pursue this goal, we investigate the following fifteen cases. We provide a brief introduction to these cases here, while detailed statements for each case can be found in Section \ref{S:Gallery}.

The first four cases are fundamental:
\begin{itemize}
\item Case 1: Pure $p_1$, i.e., $p_j = \delta_{j1}$;
\item Case 2: Pure $p_2$, i.e., $p_j = \pm \delta_{j2}$;
\item Case 3: Pure $p_3$, i.e., $p_j = \delta_{j3}$;
\item Case 4: Pure $p_4$, i.e., $p_j = \delta_{j4}$. 
\end{itemize}
By examining these four cases, one discerns the effect of each individual component.
It will turn out that the $p_1$-component exhibits a synchronizing effect as defined in Definition \ref{D:ns},
while the other four components induce a type of rotational effect.
 Note that Case 4 contains the pure $p_5$ and the combination of $p_4$ and $p_5$, with a help of change of variable.

We next turn our attention to mixed cases, where explicit integration often appears to be challenging. Nevertheless, explicit integration is possible in the following six cases:
\begin{itemize}
\item Case 5: Combination of $p_1$ and $p_2$, i.e., $p_1 > 0$, $p_2 \neq 0$, and $p_3=p_4=p_5=0$;
\item Case 6: Combination of $p_1$ and $p_4$, i.e., $p_1 > 0$, $p_4 >0$, and $p_2=p_3=p_5=0$;
This contains the combination of $p_1$ and $p_5$;
\item Case 7: Combination of $p_2$ and $p_3$, i.e., $p_2 \neq 0$, $p_3 >0$, and $p_1=p_4=p_5=0$;
\item Case 8: Combination of $p_2$ and $p_4$, i.e., $p_2 \neq 0 $, $p_4 >0$, and $p_1=p_3=p_5=0$;
This contains the combination of $p_2$ and $p_5$;
\item Case 9: Combination of $p_3$ and $p_4$, i.e., $p_3 > 0$, $p_4 >0$, and $p_1=p_2=p_5=0$;
\item Case 10: Combination of $p_3$ and $p_5$, i.e., $p_3 > 0$, $p_5 >0$, and $p_1=p_2=p_4=0$.
\end{itemize}

For more intricate combinations, explicit integration seems to be generally not available. However, in several specific combinations, we can integrate the system. The final five cases show such scenarios:
\begin{itemize}
\item Case 11: Special combination of $p_1$ and $p_3$: $p_1 >0$, $p_3/p_1 \in \{\frac13,1,3\}$, and $p_2=p_4=p_5=0$.
\item Case 12: Special combination of $p_2$, $p_3$, and $p_4$: $p_2=p_3>0$, $p_4 \neq 0$, and $p_1=p_5=0$.
\item Case 13: Special combination of $p_2$, $p_3$, and $p_5$: $p_2=-p_3<0$, $p_5>0$, and $p_1=p_4=0$.
\item Case 14: Special combination of $p_1$, $p_2$, and $p_3$: $p_1^2 + p_2^2 = p_3^2$ and $p_4=p_5=0$.
\item Case 15: Special combination of all parameters: $p_1^2 + p_2^2 = p_3^2$ and $\frac{p_4}{p_5} = \frac{p_1}{p_2+p_3}$.
\end{itemize}

Our investigation reveals that nonlinear synchronization can be observed in Cases 1 and 5, the subcase $p_1> p_4$ of Case 6, and the subcase $p_1 > p_3$ of Case 11. One sees that the $p_1$-component of \eqref{E:NLS} exhibits a synchronizing effect. In Cases 3, 7, 8, 9, and 10, Jacobi elliptic functions play a crucial role in describing solutions to \eqref{E:qqq}.

\begin{remark}
In at least several subcases of Case 11, providing an explicit solution in terms of elementary functions and Jacobi elliptic functions appears difficult, as the integration procedure involves integrals of a polynomial of fifth order or higher. (see Remark \ref{R:Jacobi_criteria}, below, for detail).
\end{remark}

\subsubsection{A precise formula of the asymptotic profile 1}\label{sss:ex1}
At the end of the introduction,
let us describe the actual formula of the asymptotic profile given by our theory in two specific cases.
The first case is the pure $p_1$-component case (Case 1).
The system is takes the following form:
\begin{equation*}
\left\{
\begin{aligned}
	(i \partial_t  + \partial_x^2)u_1={}&p_1(2|u_1|^2 u_2 + u_1^2\ol{u_2})
 - p_1 |u_2|^2u_2
	- 4p_1 \Re (\overline{u_1} u_2) u_1 +\mathcal{V}(u_1,u_2) u_1, \\
	(i \partial_t  + \partial_x^2)u_2={}&
p_1 |u_1|^2 u_1-p_1 (2 u_1 |u_2|^2 + \ol{u_1}u_2^2) + 4p_1 \Re (\overline{u_1} u_2) u_2 + \mathcal{V}(u_1,u_2)u_2.
\end{aligned}
\right.
\end{equation*}
If $\alpha_1^+(\xi)\neq 0$ for all $\xi \in \R$ then $(u_{\mathrm{app},1}^+,u_{\mathrm{app,2}}^+)$ defined by
\eqref{E:main_uapp} becomes as follows:
\begin{align*}
	&u_{\mathrm{app},1}^+(t,x)\\
	&{}= (2it)^{-\frac{1}2} e^{i\frac{x^2}{4t}}\frac{\alpha_1^+(\frac{x}{2t})}{|\alpha_1^+(\frac{x}{2t})|}\(\frac{\rho(\frac{x}{2t})}2\)^{\frac12}
	\(1+ \frac{2\mathcal{D}_0(\frac{x}{2t})}{t^{p_1\rho(\frac{x}{2t})} (\rho(\frac{x}{2t})-\mathcal{I}_0(\frac{x}{2t}))+ t^{-p_1\rho(\frac{x}{2t})} (\rho(\frac{x}{2t})+\mathcal{I}_0(\frac{x}{2t})) }\)^{\frac12}\\
	&\quad \times \frac{
	(\mathcal{R}_0(\frac{x}{2t}) - i (\rho(\frac{x}{2t})-\mathcal{I}_0(\frac{x}{2t})+\mathcal{D}_0(\frac{x}{2t}))
	(\mathcal{R}_0(\frac{x}{2t}) + i (t^{p_1\rho(\frac{x}{2t})}(\rho(\frac{x}{2t})-\mathcal{I}_0(\frac{x}{2t}))+\mathcal{D}_0(\frac{x}{2t}))
	}{
	|(\mathcal{R}_0(\frac{x}{2t}) - i (\rho(\frac{x}{2t})-\mathcal{I}_0(\frac{x}{2t})+\mathcal{D}_0(\frac{x}{2t}))
	(\mathcal{R}_0(\frac{x}{2t}) + i (t^{p_1\rho(\frac{x}{2t})}(\rho(\frac{x}{2t})-\mathcal{I}_0(\frac{x}{2t}))+\mathcal{D}_0(\frac{x}{2t}))|
	} \\
	&\quad \times \exp \Bigg(-i\tfrac{q_1+q_3}4 \rho \(\tfrac{x}{2t}\) \log t  \\
	&\qquad\quad\quad - i \frac{ \frac{q_1-q_3}2 \mathcal{D}_0(\frac{x}{2t})+q_2\mathcal{R}_0(\frac{x}{2t}) }{2p_1 \rho (\frac{x}{2t}) \sqrt{\mathcal{D}_0(\frac{x}{2t})^2+\mathcal{R}_0(\frac{x}{2t})^2}} \log \(\frac{t^{p_1\rho(\frac{x}{2t})} (\rho(\frac{x}{2t})-\mathcal{I}_0(\frac{x}{2t}))+ t^{-p_1\rho(\frac{x}{2t})} (\rho(\frac{x}{2t})+\mathcal{I}_0(\frac{x}{2t})) }{2\rho(\frac{x}{2t})}\) \Bigg) 
\end{align*}
and
\begin{align*}
	&u_{\mathrm{app},2}^+(t,x)\\
	&{}= (2it)^{-\frac{1}2} e^{i\frac{x^2}{4t}}\frac{\alpha_1^+(\frac{x}{2t})}{|\alpha_1^+(\frac{x}{2t})|}\(\frac{\rho(\frac{x}{2t})}2\)^{\frac12}
	\(2 \mathcal{R}_0(\tfrac{x}{2t}) -i (t^{p_1\rho(\frac{x}{2t})} (\rho(\tfrac{x}{2t})-\mathcal{I}_0(\tfrac{x}{2t}))- t^{-p_1\rho(\frac{x}{2t})} (\rho(\tfrac{x}{2t})+\mathcal{I}_0(\tfrac{x}{2t})) )\)\\
	&\quad \times \( {t^{p_1\rho(\frac{x}{2t})} (\rho(\tfrac{x}{2t})-\mathcal{I}_0(\tfrac{x}{2t}))+2\mathcal{D}_0(\tfrac{x}{2t})+ t^{-p_1\rho(\frac{x}{2t})} (\rho(\tfrac{x}{2t})+\mathcal{I}_0(\tfrac{x}{2t})) }\)^{-\frac12} \\
	&\quad \times \( {t^{p_1\rho(\frac{x}{2t})} (\rho(\tfrac{x}{2t})-\mathcal{I}_0(\tfrac{x}{2t}))+ t^{-p_1\rho(\frac{x}{2t})} (\rho(\tfrac{x}{2t})+\mathcal{I}_0(\tfrac{x}{2t})) }\)^{-\frac12}\\
	&\quad \times \frac{
	(\mathcal{R}_0(\frac{x}{2t}) - i (\rho(\frac{x}{2t})-\mathcal{I}_0(\frac{x}{2t})+\mathcal{D}_0(\frac{x}{2t}))
	(\mathcal{R}_0(\frac{x}{2t}) + i (t^{p_1\rho(\frac{x}{2t})}(\rho(\frac{x}{2t})-\mathcal{I}_0(\frac{x}{2t}))+\mathcal{D}_0(\frac{x}{2t}))
	}{
	|(\mathcal{R}_0(\frac{x}{2t}) - i (\rho(\frac{x}{2t})-\mathcal{I}_0(\frac{x}{2t})+\mathcal{D}_0(\frac{x}{2t}))
	(\mathcal{R}_0(\frac{x}{2t}) + i (t^{p_1\rho(\frac{x}{2t})}(\rho(\frac{x}{2t})-\mathcal{I}_0(\frac{x}{2t}))+\mathcal{D}_0(\frac{x}{2t}))|
	} \\
	&\quad \times \exp \Bigg(-i\tfrac{q_1+q_3}4 \rho \(\tfrac{x}{2t}\) \log t  \\
	&\qquad\quad\quad - i \frac{ \frac{q_1-q_3}2 \mathcal{D}_0(\frac{x}{2t})+q_2\mathcal{R}_0(\frac{x}{2t}) }{2p_1 \rho (\frac{x}{2t}) \sqrt{\mathcal{D}_0(\frac{x}{2t})^2+\mathcal{R}_0(\frac{x}{2t})^2}} \log \(\frac{t^{p_1\rho(\frac{x}{2t})} (\rho(\frac{x}{2t})-\mathcal{I}_0(\frac{x}{2t}))+ t^{-p_1\rho(\frac{x}{2t})} (\rho(\frac{x}{2t})+\mathcal{I}_0(\frac{x}{2t})) }{2\rho(\frac{x}{2t})}\) \Bigg) ,
\end{align*}
where
\[
	\rho = |\alpha_1^+|^2 + |\alpha_2^+|^2, \quad \mathcal{D}_0 = |\alpha_1^+|^2 - |\alpha_2^+|^2, \quad \mathcal{R}_0 = 2\Re\overline{\alpha_1^+} \alpha_2^+,
	\quad \mathcal{I}_0 = 2\Im\overline{\alpha_1^+} \alpha_2^+.
\]
Note that the nonlinear synchronization occurs with the pair $(\gamma_1,\gamma_2)=(1,-i)$. 
It can be seen, for instance, from the fact that $\alpha_1^+(\xi) \neq -i \alpha_2^+(\xi) \Leftrightarrow \rho(\xi) \neq \mathcal{I}_0(\xi)$ implies that
\[
	t^{\frac12}|u_{\mathrm{app},1}^+(t,2t\xi) - i u_{\mathrm{app},2}^+(t,2t\xi )| \to 0
\]
as $t\to\infty$.
\subsubsection{A precise formula of the asymptotic profile 2}\label{sss:ex2}
The second case is the pure $p_3$-component case (Case 3).
The system takes the form
\[
\left\{
\begin{aligned}
	(i \partial_t  + \partial_x^2)u_1={}&
p_3 |u_1|^2 u_1 -p_3 (2 u_1 |u_2|^2 + \ol{u_1}u_2^2)  +\mathcal{V}(u_1,u_2) u_1, \\
	(i \partial_t  + \partial_x^2)u_2={}&
-p_3(2|u_1|^2 u_2 + u_1^2\ol{u_2}) + p_3 |u_2|^2u_2 + \mathcal{V}(u_1,u_2)u_2.
\end{aligned}
\right.
\]
In this case, the Jacobi elliptic functions appear in the profile.
If $ | \alpha_1^+ (\xi)- \alpha_2^+(\xi) | | \alpha_1^+(\xi) + \alpha_2^+(\xi) | >2 |\alpha_1^+(\xi)| |\alpha_2^+(\xi)|>0$ for all $\xi \in \R$ then
$(u_{\mathrm{app},1}^+,u_{\mathrm{app,2}}^+)$ defined by
\eqref{E:main_uapp} becomes as follows:
\begin{align*}
	&u_{\mathrm{app},1}^+(t,x)\\
	&{}= (2it)^{-\frac{1}2} e^{i\frac{x^2}{4t}}\frac{\alpha_1^+(\frac{x}{2t})}{|\alpha_1^+(\frac{x}{2t})|}\(\frac{\rho(\frac{x}{2t})}2\)^{\frac12}
	\(1+ \omega_1(\tfrac{x}{2t})  \dn \( \sqrt2 p_3 \rho(\tfrac{x}{2t})\omega_1(\tfrac{x}{2t}) \log t  + t_0 , m(\tfrac{x}{2t}) \)\)^{\frac12}\\
	&\quad \times 
	\exp \Bigg( i \frac{\omega_2(\tfrac{x}{2t})^2}{\sqrt8 \omega_1(\tfrac{x}{2t})} \int_0^{\sqrt2 p_3 \rho(\tfrac{x}{2t})\omega_1(\tfrac{x}{2t}) \log t } \frac{\cn^2(\sigma+t_0,m(\frac{x}{2t}))}{1+ \omega_1 \dn (\sigma + t_0,m(\frac{x}{2t}))} d\sigma 
-i\tfrac{q_1+q_3}4 \rho \(\tfrac{x}{2t}\) \log t  \\
	&\qquad\qquad- \tfrac{i}{\sqrt8} 
	(1+\tfrac{q_1-q_3}{2p_3} )\(\am \(\sqrt2 p_3 \rho(\tfrac{x}{2t})\omega_1(\tfrac{x}{2t}) \log t  + t_0 , m(\tfrac{x}{2t})\) - \am \(   t_0 , m(\tfrac{x}{2t})\)\)
	  \\
	 &\qquad\qquad - \tfrac{iq_2}{\sqrt8 p_3} 
	\arcsin \( \frac{ \omega_2(\tfrac{x}{2t}) }{ \omega_1(\tfrac{x}{2t}) }\sn \(\sqrt2 p_3 \rho(\tfrac{x}{2t})\omega_1(\tfrac{x}{2t}) \log t  + t_0 , m(\tfrac{x}{2t})\)\) \\
	 &\qquad\qquad +\tfrac{iq_2}{\sqrt8 p_3} 
	\arcsin \( \frac{ \mathcal{D}_0(\tfrac{x}{2t}) \mathcal{I}_0(\tfrac{x}{2t}) }{ |\mathcal{D}_0(\tfrac{x}{2t})| ( \mathcal{I}_0(\tfrac{x}{2t})+ 2\mathcal{D}_0(\tfrac{x}{2t})^2 )^{1/2}}\) \Bigg)
\end{align*}
and
\begin{align*}
	&u_{\mathrm{app},2}^+(t,x)\\
	&{}= (2it)^{-\frac{1}2} e^{i\frac{x^2}{4t}}\frac{\alpha_1^+(\frac{x}{2t})}{|\alpha_1^+(\frac{x}{2t})|}\(\frac{\rho(\frac{x}{2t})}2\)^{\frac12} \omega_2(\tfrac{x}{2t})
	\(1+ \omega_1(\tfrac{x}{2t})  \dn \( \sqrt2 p_3 \rho(\tfrac{x}{2t})\omega_1(\tfrac{x}{2t}) \log t  + t_0 , m(\tfrac{x}{2t}) \)\)^{-\frac12}\\
	&\quad \times  \( \cn \( \sqrt2 p_3 \rho(\tfrac{x}{2t})\omega_1(\tfrac{x}{2t}) \log t  + t_0 , m(\tfrac{x}{2t}) \) + \sqrt2 i \sn \( \sqrt2 p_3 \rho(\tfrac{x}{2t})\omega_1(\tfrac{x}{2t}) \log t  + t_0 , m(\tfrac{x}{2t}) \)\)\\
	&\quad \times 
	\exp \Bigg( i \frac{\omega_2(\tfrac{x}{2t})^2}{\sqrt8 \omega_1(\tfrac{x}{2t})} \int_0^{\sqrt2 p_3 \rho(\tfrac{x}{2t})\omega_1(\tfrac{x}{2t}) \log t } \frac{\cn^2(\sigma+t_0,m(\frac{x}{2t}))}{1+ \omega_1 \dn (\sigma + t_0,m(\frac{x}{2t}))} d\sigma 
-i\tfrac{q_1+q_3}4 \rho \(\tfrac{x}{2t}\) \log t  \\
	&\qquad\qquad- \tfrac{i}{\sqrt8} 
	(1+\tfrac{q_1-q_3}{2p_3} )\(\am \(\sqrt2 p_3 \rho(\tfrac{x}{2t})\omega_1(\tfrac{x}{2t}) \log t  + t_0 , m(\tfrac{x}{2t})\) - \am \(   t_0 , m(\tfrac{x}{2t})\)\)
	  \\
	 &\qquad\qquad - \tfrac{iq_2}{\sqrt8 p_3} 
	\arcsin \( \frac{ \omega_2(\tfrac{x}{2t}) }{ \omega_1(\tfrac{x}{2t}) }\sn \(\sqrt2 p_3 \rho(\tfrac{x}{2t})\omega_1(\tfrac{x}{2t}) \log t  + t_0 , m(\tfrac{x}{2t})\)\) \\
	 &\qquad\qquad +\tfrac{iq_2}{\sqrt8 p_3} 
	\arcsin \( \frac{ \mathcal{D}_0(\tfrac{x}{2t}) \mathcal{I}_0(\tfrac{x}{2t}) }{ |\mathcal{D}_0(\tfrac{x}{2t})| ( \mathcal{I}_0(\tfrac{x}{2t})+ 2\mathcal{D}_0(\tfrac{x}{2t})^2 )^{1/2}}\) \Bigg),
\end{align*}
where
\[
	\rho = |\alpha_1^+|^2 + |\alpha_2^+|^2, \quad \mathcal{D}_0 = |\alpha_1^+|^2 - |\alpha_2^+|^2, \quad \mathcal{R}_0 = 2\Re\overline{\alpha_1^+} \alpha_2^+,
	\quad \mathcal{I}_0 = 2\Im\overline{\alpha_1^+} \alpha_2^+,
\]
\[
		\omega_1 =  \tfrac{ \mathcal{D}_0  }{|\mathcal{D}_0 | } \( \tfrac{\mathcal{I}_0^2 + 2 \mathcal{D}_0^2}{2\rho^2} \)^\frac12 , \qquad
\omega_2 =   \( \tfrac{\mathcal{I}_0^2 + 2 \mathcal{R}_0^2}{2\rho^2} \)^\frac12, \qquad
m = \tfrac{\omega_2^2}{\omega_1^2},
\]
and $t_0=t_0(\xi)$ is given by 
\[
(\sn( t_0, m), \cn (t_0, m)) = (\tfrac{\mathcal{I}_0}{(\mathcal{I}_0^2+2\mathcal{R}_0^2)},\tfrac{\sqrt2 \mathcal{R}_0}{(\mathcal{I}_0^2+2\mathcal{R}_0^2)}).
\]
Here, $\sn$, $\cn$, and $\dn$ are the Jacobi elliptic functions and $\am$ is the amplitude function
(see
Appendix \ref{S:elliptic}).
Note that the above condition on $(\alpha_1^+,\alpha_2^+)$ is equivalent to $0< \omega_2 < |\omega_1|<1$.
The validity of these inequalities is assumed for simplicity, i.e., to eliminate the need for case divisions in the representation of the asymptotic profiles.
We also note that the formula involves a definite integral of $\frac{\cn^2(t,m)}{1+ \omega \dn (t.m)}$ ($|\omega|<1)$.
A computer-aided calculation suggests that a primitive of this function is also explicitly expressed in terms of the Jacobi elliptic functions and the elliptic integrals.
We do not pursue it here.

\bigskip

The rest of the paper is organized as follows.
In Section \ref{S:Gallery}, we collect the explicit formulas for the solution to \eqref{E:qqq} in the above fifteen cases. Let us recall again that these formula together with Theorems \ref{T:KS} and \ref{T:main} give the large
time asymptotics of the solutions to the NLS system \eqref{E:NLS}.
In Section \ref{S:classification},
we briefly recall the classification argument in \cites{MSU2,M} and prove that \eqref{E:NLS} is a standard form of system which has a coercive mass-like conserved quantity.
Then, we turn to the proof of the main results.
Section \ref{S:pfmain} is devoted to the proof of Theorem \ref{T:main}.
We discuss the integration of \eqref{E:qqq} in Section \ref{S:qqq}.
Finally, we prove Theorem \ref{P:NS} in Section \ref{S:NS}.

\section{Gallery} \label{S:Gallery}
In this section, we collect explicit solutions to the quadratic ODE system \eqref{E:qqq} in Cases 1 to 15.
In these cases, one can obtain the explicit representation of a solution for arbitrary data, in terms of elementary functions and  Jacobi elliptic functions.
The notation and basic facts on the Jacobi elliptic functions is summarized in Appendix \ref{S:elliptic}.
Since the trivial data gives a trivial solution to \eqref{E:qqq}, we consider nontrivial solutions, that is,
we take data from $S^2_\rho$ for some $\rho>0$ unless otherwise stated.
\subsection{Room 1 -- pure cases}
In the first part, we study the four pure cases.
They all have different characters.
\subsubsection{Case 1}
The NLS system takes the form
\begin{equation}\label{E:NLS1}
\left\{
\begin{aligned}
	(i \partial_t  + \partial_x^2)u_1={}&p_1(2|u_1|^2 u_2 + u_1^2\ol{u_2})
 - p_1 |u_2|^2u_2
	- 4p_1 \Re (\overline{u_1} u_2) u_1 +\mathcal{V}(u_1,u_2) u_1, \\
	(i \partial_t  + \partial_x^2)u_2={}&
p_1 |u_1|^2 u_1-p_1 (2 u_1 |u_2|^2 + \ol{u_1}u_2^2) + 4p_1 \Re (\overline{u_1} u_2) u_2 + \mathcal{V}(u_1,u_2)u_2
\end{aligned}
\right.
\end{equation}
with $p_1 = 1$.
The quadratic system \eqref{E:qqq} becomes
\begin{equation}\label{E:qqq_1}
	\mathcal{D}' = 2p_1 \mathcal{I} \mathcal{D}, \quad
	\mathcal{R}' = 2p_1 \mathcal{I} \mathcal{R}, \quad
	\mathcal{I}' = -2p_1 ( \mathcal{D}^2 + \mathcal{R}^2 ) .
\end{equation}
We have the following

\begin{proposition}
The two points $(\mathcal{D},\mathcal{R}, \mathcal{I})=\pm(0,0,\rho)$ are fixed points of \eqref{E:qqq_1}. 
If $(\mathcal{D}(0),\mathcal{R}(0),\mathcal{I}(0))\neq \pm (0,0,\rho)$ then the solution to \eqref{E:qqq_1} is given by
\[
	\mathcal{D}(\tau) = \tfrac{\rho \mathcal{D}(0) }{\sqrt{\mathcal{D}(0)^2+\mathcal{R}(0)^2}} (\cosh (2 p_1 \rho \tau -\tanh^{-1} \tfrac{\mathcal{I}(0)}\rho))^{-1}, 
\]
\[
	\mathcal{R}(\tau) = \tfrac{\rho \mathcal{R}(0) }{\sqrt{\mathcal{D}(0)^2+\mathcal{R}(0)^2}} (\cosh (2 p_1 \rho \tau -\tanh^{-1} \tfrac{\mathcal{I}(0)}\rho))^{-1}, 
\]
and
\[
	\mathcal{I}(\tau) = - \rho \tanh \(2 p_1 \rho \tau - \tanh^{-1} \tfrac{\mathcal{I}(0)}\rho \).
\]
\end{proposition}
Combining the proposition with Theorems \ref{T:KS} and \ref{T:main}, we obtain the asymptotic profile of  solutions to \eqref{E:NLS1}, as seen in Section \ref{sss:ex1}.
One sees that the nonlinear synchronization occurs in the sense of Definition \ref{D:ns}; all solution, other than $(0,0,\rho)$, converges to the same fixed point $(0,0,-\rho)$ as $\tau \to \infty$.

\begin{remark}
Katayama-Matoba-Sunagawa \cite{KMS} studies a system of semilinear  nonlinear wave equations for which
an energy-transfer type phenomena occurs.
In there result, one key of this kind of one-take-all-type behavior was the ODE system
\[
	X' = XY, \quad Y' = Y^2 -C.
\]
If we set
 $X=\mathcal{D}^2+\mathcal{R}^2-(\mathcal{D}(0)^2+\mathcal{R}(0)^2)$ and $Y= \mathcal{I}$ for a solution to \eqref{E:qqq_1}, we obtain essentially the same ODE system.
\end{remark}

\subsubsection{Case 2}
We move to the Case 2. 
The NLS system takes the form
\begin{equation}\label{E:NLS2}
\left\{
\begin{aligned}
	(i \partial_t  + \partial_x^2)u_1={}&
3p_2  |u_1|^2 u_1 +p_2 (2 u_1 |u_2|^2 + \ol{u_1}u_2^2) 
  +\mathcal{V}(u_1,u_2) u_1, \\
	(i \partial_t  + \partial_x^2)u_2={}&
p_2(2|u_1|^2 u_2 + u_1^2\ol{u_2}) + 3p_2 |u_2|^2u_2
 + \mathcal{V}(u_1,u_2)u_2.
\end{aligned}
\right.
\end{equation}
The corresponding quadratic system \eqref{E:qqq} becomes
\begin{equation}\label{E:qqq_2}
	\mathcal{D}' = 2p_2 \mathcal{I} \mathcal{R}, \quad
	\mathcal{R}' = -2p_2 \mathcal{I} \mathcal{D}, \quad
	\mathcal{I}' = 0.
\end{equation}
The system is studied in \cites{MS,S} in the context of nonlinear Klein-Gordon system.
We have the following result for the solution to \eqref{E:qqq_2}.
\begin{proposition}
$\{ (\rho\cos \theta, \rho \sin \theta,0)\ |\ \theta\in \R/2\pi\Z \}\cup\{(0,0,\pm\rho)\}$ is the set of all fixed points of \eqref{E:qqq_2}.
The solution to \eqref{E:qqq_2} is given by
\[
	\mathcal{D}(\tau) = \mathcal{D}(0) \cos (2p_2 \tau \mathcal{I}(0) ) + \mathcal{R}(0) \sin (2p_2 \tau \mathcal{I}(0) ), 
\]
\[
	\mathcal{R}(\tau) = -\mathcal{D}(0) \sin (2p_2 \tau \mathcal{I}(0) ) + \mathcal{R}(0) \cos (2p_2 \tau \mathcal{I}(0) ), 
\]
and
\[
	\mathcal{I}(\tau) = \mathcal{I}(0).
\]
\end{proposition}
Combining the proposition with Theorems \ref{T:KS} and \ref{T:main}, we obtain the asymptotic profile of  solutions to \eqref{E:NLS2}.

\subsubsection{Case 3}
In the pure $p_3$ case,  the NLS system is
\begin{equation}\label{E:NLS3}
\left\{
\begin{aligned}
	(i \partial_t  + \partial_x^2)u_1={}&
p_3 |u_1|^2 u_1 -p_3 (2 u_1 |u_2|^2 + \ol{u_1}u_2^2)  +\mathcal{V}(u_1,u_2) u_1, \\
	(i \partial_t  + \partial_x^2)u_2={}&
-p_3(2|u_1|^2 u_2 + u_1^2\ol{u_2}) + p_3 |u_2|^2u_2 + \mathcal{V}(u_1,u_2)u_2
\end{aligned}
\right.
\end{equation}
and the quadratic system \eqref{E:qqq} takes the form
\begin{equation}\label{E:qqq_3}
	\mathcal{D}' = -2p_3 \mathcal{I} \mathcal{R}, \quad
	\mathcal{R}' = -2p_3 \mathcal{I} \mathcal{D}, \quad
	\mathcal{I}' = 4p_3 \mathcal{D} \mathcal{R}
\end{equation}
with $p_3=1$.
This is the typical ODE system which the Jacobi elliptic functions solve:
\begin{lemma}\label{L:ellipticODE1}
A solution to the quadratic ODE system 
\begin{equation}\label{E:ellipticODE1}
	f' = gh, \quad g' = - fh , \quad h' = -fg
\end{equation}
with a data
\[
	(f,g,h)(0) = (f_0,g_0,h_0) \in \R^3
\]
is given as follows: Let
$R_{fg} =\sqrt{f_0^2 + g_0^2}$ and $R_{fh} =\sqrt{ f_0^2 + h_0^2} $.
Suppose that $R_{fh} \ge R_{fg}$.
\begin{itemize}
\item If $R_{fg}=0$ then
$f(t)= 0$, $g(t) = 0$, and $h(t) = h_0$.
\item
If $R_{fh} > R_{fg}>0$ then $h_0\neq0$ follows and one has
\begin{align*}
	 f(t) ={}& R_{fg} \sn \( (\sign h_0) R_{fh} t  + t_0,  \tfrac{R_{fg}^2}{R_{fh}^2} \), \\
	  g(t) ={}& R_{fg} \cn \( (\sign h_0) R_{fh} t + t_0, \tfrac{R_{fg}^2}{R_{fh}^2} \), \\
		h(t) ={}& (\sign h_0) R_{fh}  \dn \( (\sign h_0) R_{fh} t + t_0,  \tfrac{R_{fg}^2}{R_{fh}^2} \),
\end{align*}
where $t_0$ is given by $(\sn( t_0, R_{fg}^2/R_{fh}^2), \cn (t_0, R_{fg}^2/R_{fh}^2)) = (f_0/R_{fg},g_0/R_{fg})$.
\item If $R_{fh} = R_{fg}>0$ and $g_0=h_0=0$ then $f(t)=f_0$, $g(t)=0$, and $h(t)=0$. 
\item If $R_{fh} = R_{fg}>0$ and $g_0\neq 0 $ then $|h_0|= |g_0|>0$ follows and one has
\begin{align*}
	 f(t) ={}& R_{fh} \tanh ( (\sign (g_0h_0))R_{fh} t + t_0), \\ g(t) ={}& (\sign g_0) R_{fh} \sech ( (\sign (g_0h_0))R_{fh} t + t_0), \\
		h(t) ={}& (\sign h_0)R_{fh} \sech ( (\sign (g_0h_0))R_{fh} t + t_0), 
\end{align*}
where $t_0 = \tanh^{-1} (f_0/R_{fh})$.
\end{itemize}
In all cases, $f(t)^2 + g(t)^2$ and $f(t)^2 + h(t)^2$ are conserved and equal to $R_{fg}^2$ and $R_{fh}^2$, respectively.
The explicit formula of a solution in the case $R_{fh}<R_{fg}$ is obtained by swapping $g$ and $h$.
\end{lemma}

\begin{proposition}
The six points $(\mathcal{D},\mathcal{R}, \mathcal{I})=\pm(\rho,0,0),\pm(0,\rho,0),\pm(0,0,\rho)$ are fixed points
of \eqref{E:qqq_3}.
Further,
the triplet  $(2 p_3\mathcal{I}, \sqrt{8}p_3 \mathcal{R}, \sqrt{8}p_3\mathcal{D} )$ solves \eqref{E:ellipticODE1}
and hence the solution $(\mathcal{D},\mathcal{R},\mathcal{I})$ to \eqref{E:qqq_3} is written explicitly in terms of 
the elementary functions and
the Jacobi elliptic functions as in Lemma \ref{L:ellipticODE1}.
\end{proposition}
Combining the proposition with Theorems \ref{T:KS} and \ref{T:main}, we obtain the asymptotic profile of  solutions to \eqref{E:NLS3}, as seen in Section \ref{sss:ex2}.

\begin{remark}
Although the explicit formula of a solution $(\mathcal{D}, \mathcal{R}, \mathcal{I})$ is somewhat complicated, the orbit of the solution is easily understood:
Since $2\mathcal{D}^2 + \mathcal{I}^2$ 
is a conserved quantity,
the orbit is a subset of the intersection of the 
$S_\rho^2$
and the boundary of the elliptical cylinder $\{2 x^2 + z^2 =
2\mathcal{D}(0)^2 + \mathcal{I}(0)^2 \}$. Note that there are two more conserved quantities;
$2\mathcal{R}^2 + \mathcal{I}^2$ and $\mathcal{D}^2 - \mathcal{R}^2$.
One obtains similar characterizations of the orbit with these quantities.
\end{remark}

\subsubsection{Case 4}
This case is notably simple since the system is essentially decoupled. (If $\mathcal{V}\equiv0$ in addition then it is completely decoupled.)  Indeed, one has
\begin{equation}\label{E:NLS4}
\left\{
\begin{aligned}
	(i \partial_t  + \partial_x^2)u_1={}&
2p_4 |u_1|^2 u_1  +\mathcal{V}(u_1,u_2) u_1, \\
	(i \partial_t  + \partial_x^2)u_2={}&
-2p_4 |u_2|^2u_2 + \mathcal{V}(u_1,u_2)u_2.
\end{aligned}
\right.
\end{equation}
Hence, the analysis for the single equation applies. 
Here, for completeness, let us record the explicit solution to the quadratic system \eqref{E:qqq}, which is now of the form
\begin{equation}\label{E:qqq_4}
	\mathcal{D}' = 0, \quad
	\mathcal{R}' = -2p_4 \rho \mathcal{I} , \quad
	\mathcal{I}' = 2p_4 \rho \mathcal{R}.
\end{equation}
\begin{proposition}
The two points $(\mathcal{D},\mathcal{R}, \mathcal{I})=\pm(\rho, 0,0)$ are fixed points of \eqref{E:qqq_4}.
Moreover, the solution to \eqref{E:qqq_4} is given by
\[
	\mathcal{D}(\tau) = \mathcal{D}(0), 
\]
\[
	\mathcal{R}(\tau) = \mathcal{R}(0) \cos (2p_4\rho \tau  ) - \mathcal{I}(0) \sin (2p_4 \rho \tau ), 
\]
and
\[
	\mathcal{I}(\tau) = \mathcal{R}(0) \sin (2p_4\rho \tau  ) + \mathcal{I}(0) \cos (2p_4 \rho \tau ).
\]
\end{proposition}
Combining the proposition with Theorems \ref{T:KS} and \ref{T:main}, we obtain the asymptotic profile of  solutions to \eqref{E:NLS4}.
Note that, in this case, one has $N_1 = p_4(\rho + \mathcal{D}(0)) = 2p_4|A_1(0)|^2$
and $N_2=p_4(\rho -\mathcal{D}(0))= 2 p_4 |A_2(0)|^2$.
These yield the standard asymptotic profile.
\begin{remark}
The pure $p_5$-component case is studied in \cite{KS}*{Example 6.2}.
As mentioned above, the case is reduced to this pure $p_4$-component case by a change of variable.
\end{remark}

\subsection{Room 2 -- mixed cases}

We turn to the mixed cases.
One will see that the characteristic properties seen in the pure cases are sometimes simply superposed 
and sometimes compete with each other.

\subsubsection{Case 5}
Let us begin with the mixture of  the
$p_1$-component  and the $p_2$-component. 
We see that the resulting behavior possesses the both properties appear in Cases 1 and 2.
The system is
\begin{equation}\label{E:NLS5}
\left\{
\begin{aligned}
	(i \partial_t  + \partial_x^2)u_1={}&
3p_2  |u_1|^2 u_1 +p_1 (2|u_1|^2 u_2 + u_1^2\ol{u_2})
+p_2 (2 u_1 |u_2|^2 + \ol{u_1}u_2^2) - p_1 |u_2|^2u_2
	\\& - 4p_1 \Re (\overline{u_1} u_2) u_1 +\mathcal{V}(u_1,u_2) u_1, \\
	(i \partial_t  + \partial_x^2)u_2={}&
p_1 |u_1|^2 u_1+p_2(2|u_1|^2 u_2 + u_1^2\ol{u_2})-p_1 (2 u_1 |u_2|^2 + \ol{u_1}u_2^2) + 3p_2 |u_2|^2u_2
	\\& + 4p_1 \Re (\overline{u_1} u_2) u_2 + \mathcal{V}(u_1,u_2)u_2.
\end{aligned}
\right.
\end{equation}
The quadratic system \eqref{E:qqq} is
\begin{equation}\label{E:qqq_5}
\begin{bmatrix}\mathcal{D}\\\mathcal{R}\end{bmatrix}' =2\mathcal{I}
\begin{bmatrix}
p_1 & p_2 \\ - p_2 & p_1 
\end{bmatrix}
\begin{bmatrix}\mathcal{D}\\\mathcal{R}\end{bmatrix} ,\quad
\mathcal{I}' =  -2p_1 (\mathcal{D}^2+\mathcal{R}^2) 
\end{equation}
in this case.
\begin{proposition}
The two points $(\mathcal{D},\mathcal{R}, \mathcal{I})=\pm(0,0,\rho)$ are fixed points of \eqref{E:qqq_5}.
If $(\mathcal{D}(0),\mathcal{R}(0),\mathcal{I}(0))\neq \pm (0,0,\rho)$ then the solution to \eqref{E:qqq_5} is given by
\[
	\mathcal{D}(\tau) = \rho \frac{\cos (\tau_0 + \frac{p_2}{p_1} \log ((\rho-\mathcal{I}(0))e^{2p_1\rho \tau}+(\rho+\mathcal{I}(0))e^{-2p_1\rho \tau}) )}{\cosh (2 p_1 \rho \tau - \tanh^{-1} \tfrac{\mathcal{I}(0)}\rho)},
\]
\[
	\mathcal{R}(\tau) = \rho \frac{\sin (\tau_0 + \frac{p_2}{p_1} \log ((\rho-\mathcal{I}(0))e^{2p_1\rho \tau}+(\rho+\mathcal{I}(0))e^{-2p_1\rho \tau}) )}{\cosh (2 p_1 \rho \tau - \tanh^{-1} \tfrac{\mathcal{I}(0)}\rho)},
\]
and
\[
	\mathcal{I}(\tau) = - \rho \tanh \(2 p_1 \rho \tau - \tanh^{-1} \tfrac{\mathcal{I}(0)}\rho\),
\]
where $\tau_0$ is given by the relation
\[
	(\cos (\tau_0 + \log (2\rho)^{p_2/p_1}) , \sin (\tau_0 + \log (2\rho)^{p_2/p_1}) )
	= (\mathcal{D}(0)^2+\mathcal{R}(0)^2)^{-\frac12} (\mathcal{D}(0),\mathcal{R}(0)).
\]
\end{proposition}
Combining the proposition with Theorems \ref{T:KS} and \ref{T:main}, we obtain the asymptotic profile of  solutions to \eqref{E:NLS5}.
One sees that the nonlinear synchronization occurs as in Case 1, all solution, other than $(0,0,\rho)$, converges to the same fixed point $(0,0,-\rho)$ as $\tau \to \infty$.

\subsubsection{Case 6}

We next consider the case where $p_1$-component and $p_4$-component are present.
It will turn out that the two characteristic behavior seen in the pure cases compete with each other.
The system is
\begin{equation}\label{E:NLS6}
\left\{
\begin{aligned}
	(i \partial_t  + \partial_x^2)u_1={}&
2p_4 |u_1|^2 u_1 + p_1 (2|u_1|^2 u_2 + u_1^2\ol{u_2}) - p_1 |u_2|^2u_2
	 - 4p_1 \Re (\overline{u_1} u_2) u_1 +\mathcal{V}(u_1,u_2) u_1, \\
	(i \partial_t  + \partial_x^2)u_2={}&
p_1 |u_1|^2 u_1-p_1 (2 u_1 |u_2|^2 + \ol{u_1}u_2^2) -2p_4 |u_2|^2u_2
	+ 4p_1 \Re (\overline{u_1} u_2) u_2 + \mathcal{V}(u_1,u_2)u_2.
\end{aligned}
\right.
\end{equation}
The quadratic system is
\begin{equation}\label{E:qqq_6}
	\mathcal{D}' = 2p_1 \mathcal{I} \mathcal{D}, \quad
	\mathcal{R}' = 2p_1 \mathcal{I} \mathcal{R}-2p_4 \rho \mathcal{I}, \quad
	\mathcal{I}' = -2p_1 ( \mathcal{D}^2 + \mathcal{R}^2 )+2p_4 \rho \mathcal{R} .
\end{equation}

\begin{proposition}
\begin{itemize}
\item If $p_1>p_4$ then
the two points $(\mathcal{D},\mathcal{R}, \mathcal{I})=(0,\tfrac{p_4}{p_1} \rho, \pm \rho \sqrt{1-(\tfrac{p_4}{p_1})^2})$ are fixed points of \eqref{E:qqq_6}.
If $(\mathcal{D}(0),\mathcal{R}(0),\mathcal{I}(0))$ is not equal to the fixed point then the solution to \eqref{E:qqq_6} is given by
\[
	\mathcal{D}(\tau) = \tfrac{(1-(\frac{p_4}{p_1})^2) \rho \mathcal{D}(0) }{ \sqrt{ (\frac{p_4}{p_1} \rho-\mathcal{R}(0))^2 + (1-(\frac{p_4}{p_1})^2 )\mathcal{D}(0)^2 } \cosh \(2p_1 \rho \sqrt{1-(\frac{p_4}{p_1})^2} \tau - \tau_0\) - \frac{p_4}{p_1}(\frac{p_4}{p_1} \rho-\mathcal{R}(0)) },
\]
\[
	\mathcal{R}(\tau) =\tfrac{p_4}{p_1} \rho + \tfrac{(1-(\frac{p_4}{p_1})^2) \rho (\mathcal{R}(0) - \frac{p_4}{p_1} \rho)}{ \sqrt{ (\frac{p_4}{p_1} \rho-\mathcal{R}(0))^2 + (1-(\frac{p_4}{p_1})^2 )\mathcal{D}(0)^2 } \cosh \(2p_1 \rho \sqrt{1-(\frac{p_4}{p_1})^2} \tau - \tau_0\) - \frac{p_4}{p_1}(\frac{p_4}{p_1} \rho-\mathcal{R}(0)) } 
	 ,
\]
and
\[
	\mathcal{I}(\tau) = -\tfrac{\rho\sqrt{1-(\frac{p_4}{p_1})^2}  \sqrt{ (\frac{p_4}{p_1} \rho-\mathcal{R}(0))^2 + (1-(\frac{p_4}{p_1})^2 )\mathcal{D}(0)^2 } \sinh \(2p_1 \rho \sqrt{1-(\frac{p_4}{p_1})^2} \tau - \tau_0\) }{ \sqrt{ (\frac{p_4}{p_1} \rho-\mathcal{R}(0))^2 + (1-(\frac{p_4}{p_1})^2 )\mathcal{D}(0)^2 } \cosh \(2p_1 \rho \sqrt{1-(\frac{p_4}{p_1})^2} \tau - \tau_0\) - \frac{p_4}{p_1}(\frac{p_4}{p_1} \rho-\mathcal{R}(0)) },
\]
where $\tau_0=\tau_0(\xi) \in \R$ is such that $\tau_0 \mathcal{I}(0) \ge 0$ and
\[
	\tfrac{(1-(\frac{p_4}{p_1})^2) \rho }{ \sqrt{ (\frac{p_4}{p_1} \rho-\mathcal{R}(0))^2 + (1-(\frac{p_4}{p_1})^2 )\mathcal{D}(0)^2 } \cosh ( \tau_0) - \frac{p_4}{p_1}(\frac{p_4}{p_1} \rho-\mathcal{R}(0)) } =1.
\]
Any non-equilibrium solution satisfies
\[
	(\mathcal{D}(\tau),\mathcal{R}(\tau),\mathcal{I}(\tau)) \to (0,\tfrac{p_4}{p_1} \rho, \mp \rho \sqrt{1-(\tfrac{p_4}{p_1})^2})
\]
as $\tau \to \pm \infty$. 
\item If $p_1=p_4$ then
the $(\mathcal{D},\mathcal{R}, \mathcal{I})=(0,\rho, 0)$ is the unique fixed point of \eqref{E:qqq_6}.
If $(\mathcal{D}(0),\mathcal{R}(0),\mathcal{I}(0))$ is not equal to the fixed point then the solution to \eqref{E:qqq_6} is given by
\[
	\mathcal{D}(\tau) = \tfrac{ 2 \rho (\rho-\mathcal{R}(0)) \mathcal{D}(0)}{ (2p_1 \rho (\rho-\mathcal{R}(0)) \tau - \mathcal{I}(0))^2 +  2 \rho (\rho-\mathcal{R}(0)) - \mathcal{I}(0)^2 },
\]
\[
	\mathcal{R}(\tau) =  \rho + \tfrac{ -2 \rho (\rho-\mathcal{R}(0))^2}{ (2p_1 \rho (\rho-\mathcal{R}(0)) \tau - \mathcal{I}(0))^2 +  2 \rho (\rho-\mathcal{R}(0)) - \mathcal{I}(0)^2 } 
	,
\]
and
\[
	\mathcal{I}(\tau) = - \tfrac{ 2 \rho (\rho-\mathcal{R}(0))(2p_1 \rho (\rho-\mathcal{R}(0)) \tau - \mathcal{I}(0))}{ (2p_1 \rho (\rho-\mathcal{R}(0)) \tau- \mathcal{I}(0))^2 +  2 \rho (\rho-\mathcal{R}(0)) - \mathcal{I}(0)^2 }.
\]
Any solution satisfies
\[
	(\mathcal{D}(\tau),\mathcal{R}(\tau),\mathcal{I}(\tau)) \to (0,\rho,0)
\]
as $\tau \to \pm \infty$. 
\item If $p_1<p_4$ then
the two points $(\mathcal{D},\mathcal{R}, \mathcal{I})=( \pm \rho \sqrt{1-(\tfrac{p_1}{p_4})^2},\tfrac{p_1}{p_4} \rho,0)$ are fixed points of \eqref{E:qqq_6}.
If $(\mathcal{D}(0),\mathcal{R}(0),\mathcal{I}(0))$ is not equal to the fixed point then the solution to \eqref{E:qqq_6} is given by
\[
	\mathcal{D}(\tau) = \tfrac{((\frac{p_4}{p_1})^2-1) \rho \mathcal{D}(0) }{\frac{p_4}{p_1}(\frac{p_4}{p_1} \rho-\mathcal{R}(0)) -  \sqrt{ (\frac{p_4}{p_1} \rho-\mathcal{R}(0))^2 - ((\frac{p_4}{p_1})^2-1 )\mathcal{D}(0)^2 } \cos \(2p_1 \rho \sqrt{(\frac{p_4}{p_1})^2-1} \tau - \tau_0\) },
\]
\[
	\mathcal{R}(\tau) =\tfrac{p_4}{p_1} \rho + \tfrac{((\frac{p_4}{p_1})^2-1) \rho (\mathcal{R}(0)-\frac{p_4}{p_1} \rho) }{\frac{p_4}{p_1}(\frac{p_4}{p_1} \rho-\mathcal{R}(0)) -  \sqrt{ (\frac{p_4}{p_1} \rho-\mathcal{R}(0))^2 - ((\frac{p_4}{p_1})^2-1 )\mathcal{D}(0)^2 } \cos \(2p_1 \rho \sqrt{(\frac{p_4}{p_1})^2-1} \tau - \tau_0\) } 
	 ,
\]
and
\[
	\mathcal{I}(\tau) = - \tfrac{\rho\sqrt{(\frac{p_4}{p_1}-1)^2}  \sqrt{ (\frac{p_4}{p_1} \rho-\mathcal{R}(0))^2 - ((\frac{p_4}{p_1})^2-1 )\mathcal{D}(0)^2 } \sin \(2p_1 \rho \sqrt{1-(\frac{p_4}{p_1})^2} \tau - \tau_0\) }{\frac{p_4}{p_1}(\frac{p_4}{p_1} \rho-\mathcal{R}(0)) -  \sqrt{ (\frac{p_4}{p_1} \rho-\mathcal{R}(0))^2 - ((\frac{p_4}{p_1})^2-1 )\mathcal{D}(0)^2 } \cos \(2p_1 \rho \sqrt{(\frac{p_4}{p_1})^2-1} \tau - \tau_0\) } ,
\]
where $\tau_0=\tau_0(\xi) \in (-\pi,\pi]$ is given by $\tau_0\mathcal{I}(0) \ge 0$ and
\[
\tfrac{((\frac{p_4}{p_1})^2-1) \rho }{\frac{p_4}{p_1}(\frac{p_4}{p_1} \rho-\mathcal{R}(0)) -  \sqrt{ (\frac{p_4}{p_1} \rho-\mathcal{R}(0))^2 - ((\frac{p_4}{p_1})^2-1 )\mathcal{D}(0)^2 } \cos  \tau_0 }=1.
\]
Any non-equilibrium solution is periodic in $\tau$.
\end{itemize}
\end{proposition}
Combining the proposition with Theorems \ref{T:KS} and \ref{T:main}, we obtain the asymptotic profile of  solutions to \eqref{E:NLS6}.
One sees that the nonlinear synchronization occurs if $p_1>p_4$.

\subsubsection{Case 7}

Let us move to Case 7, the combination of $p_2$ and $p_3$. The system is
\begin{equation}\label{E:NLS7}
	\left\{
\begin{aligned}
	(i \partial_t  + \partial_x^2)u_1={}&
(3p_2 +p_3) |u_1|^2 u_1 +(p_2-p_3) (2 u_1 |u_2|^2 + \ol{u_1}u_2^2)  +\mathcal{V}(u_1,u_2) u_1, \\
	(i \partial_t  + \partial_x^2)u_2={}&
(p_2-p_3)(2|u_1|^2 u_2 + u_1^2\ol{u_2}) + (3p_2+p_3) |u_2|^2u_2
 + \mathcal{V}(u_1,u_2)u_2.
\end{aligned}
\right.
\end{equation}
We make an additional condition $|p_2|\neq p_3$ with the following reason.
The case $p_2 =  p_3$ is a degenerate case. The system is almost decoupled.
The case $p_2 = -p_3$ is also a degenerate case. The system \eqref{E:NLS} becomes essentially decoupled
by introducing new unknowns $v_1=u_1+u_2$ and $v_2 = u_1 - u_2$.
Hence, we consider the other case, i.e., $|p_2| \neq p_3$.

The quadratic system is
\begin{equation}\label{E:qqq_7}
	\mathcal{D}' = 2(p_2-p_3) \mathcal{R}\mathcal{I}, \quad
	\mathcal{R}' = -2(p_2+p_3) \mathcal{D}\mathcal{I}, \quad
	\mathcal{I}' = 4p_3 \mathcal{D}\mathcal{R}, \quad
\end{equation}
This is again the typical system for Jacobi elliptic functions.
\begin{proposition}
The six points $(\mathcal{D},\mathcal{R}, \mathcal{I})=\pm(\rho,0,0),\pm(0,\rho,0),\pm(0,0,\rho)$ are fixed points.
In other case,
the triplet  
\[
	(f,g,h) =
	\begin{cases}
	(- \sqrt{8p_3|p_2+p_3|} \mathcal{D}, \sqrt{8p_3(p_3-p_2)} \mathcal{R}, 2 \sqrt{p_2^2-p_3^2} \mathcal{I}) & p_2<-p_3,\\
	(-2 \sqrt{p_3^2-p_2^2} \mathcal{I},  \sqrt{8p_3(p_3-p_2)} \mathcal{R}, \sqrt{8p_3(p_2+p_3)} \mathcal{D}) & -p_3<p_2<p_3, \\
	(-\sqrt{8p_3(p_2-p_3)} \mathcal{R}, \sqrt{8p_3(p_2+p_3)} \mathcal{D},  2 \sqrt{p_2^2-p_3^2} \mathcal{I}) & p_3<p_2
	\end{cases}
\]
 solves \eqref{E:ellipticODE1}
and hence the solution $(\mathcal{D},\mathcal{R},\mathcal{I})$ to \eqref{E:ODE} is written explicitly in terms of the Jacobi elliptic functions as in Lemma \ref{L:ellipticODE1}.
\end{proposition}
Combining the proposition with Theorems \ref{T:KS} and \ref{T:main}, we obtain the asymptotic profile of  solutions to \eqref{E:NLS7}.

\begin{remark}
The large-time behavior of special solutions of the system
\begin{equation}\label{E:ellipticU}
	\left\{
	\begin{aligned}
	&(i \partial_t  + \partial_x^2) v_1 =  v_2^2 \ol{v_1},  \\
	&(i \partial_t  + \partial_x^2) v_2 = v_1^2 \ol{v_2}
	\end{aligned}
	\right.
\end{equation}
is studied in Uriya \cite{U}. 
Note that if we apply a change of variable  $u_1=2^{-\frac12} (v_1-v_2)$,
$u_2=2^{-\frac12} (v_1+v_2)$ then \eqref{E:ellipticU} turns into \eqref{E:NLS}
with $p_2=3/4$, $p_3=1/4$, $p_1=p_4=p_5=0$, and $(q_1,q_2,q_3)=(-2,0,-2)$, that is, into \eqref{E:NLS7}.
In this sense, \eqref{E:ellipticU} is classified in this case.
It is shown in \cite{U} that there exists a one-take-all-type solution to 
 \eqref{E:ellipticU}  such that
 $\lim_{t\to\I} \norm{v_2(t)}_{L^2}=0$.
In view of the above change of variable, 
the solution $(v_1,v_2)$ corresponds to the solution $(u_1,u_2)$ of \eqref{E:NLS7} for which
\[
	\mathcal{D}(\tau;\xi) \to 0,
	\quad
	\mathcal{R}(\tau;\xi) \to \rho,
	\quad
	\mathcal{I}(\tau;\xi) \to 0
\]
hold for all $\xi \in \R$ as $\tau\to\I$. 
In view of Theorems \ref{T:KS} and \ref{T:main},
this solution $(u_1,u_2)$ is so special that the final data
$\alpha_1^+$ and $\alpha_2^+$ satisfies
$ \mathcal{D}(0) \equiv - \mathcal{I}(0)$ everywhere.
A modified-scattering-type solution $(v_1,v_2)$ to \eqref{E:ellipticU} is also constructed,
The solution corresponds to that of \eqref{E:NLS7} such that
$\alpha_1^+$ and $\alpha_2^+$ satisfies
$\mathcal{R}(0) \equiv \mathcal{I}(0) \equiv 0$.
These special solutions are given by solving the final value problem.
It seems difficult to characterize these solutions to \eqref{E:NLS7} in the language of
the initial data since we do not have a control on the functions
$\alpha_1^+$ and $\alpha_2^+$ good enough to assure $ \mathcal{D}(0) \equiv - \mathcal{I}(0)$ or $\mathcal{R}(0) \equiv \mathcal{I}(0) \equiv 0$ holds for all $\xi$.
\end{remark}

\subsubsection{Case 8}

Let us next consider the combination of $p_2$ and $p_4$.
\begin{equation}\label{E:NLS8}
\left\{
\begin{aligned}
	(i \partial_t  + \partial_x^2)u_1={}&
(3p_2 +2p_4) |u_1|^2 u_1 +p_2 (2 u_1 |u_2|^2 + \ol{u_1}u_2^2)  +\mathcal{V}(u_1,u_2) u_1, \\
	(i \partial_t  + \partial_x^2)u_2={}&
p_2(2|u_1|^2 u_2 + u_1^2\ol{u_2}) + (3p_2-2p_4) |u_2|^2u_2 + \mathcal{V}(u_1,u_2)u_2.
\end{aligned}
\right.
\end{equation}
In this case, the Jacobi elliptic function appears.
\begin{equation}\label{E:qqq_8}
	\mathcal{D}' = 2p_2 \mathcal{I} \mathcal{R}, \quad
	\mathcal{R}' = -2p_2 \mathcal{I} \mathcal{D}-2p_4 \rho \mathcal{I}, \quad
	\mathcal{I}' = 2p_4 \rho \mathcal{R} .
\end{equation}
We remark that, as seen in Cases 2 and 4, the $p_2$-component and the $p_4$-component do not give behavior described by the Jacobi elliptic functions.

\begin{lemma}\label{L:ellipticODE2}
A solution to the quadratic ODE system 
\begin{equation}\label{E:ellipticODE2}
	f' = gh, \quad g' = -fh, \quad h'=-f
\end{equation}
 with a data
\[
	(f,g,h)(0) = (f_0,g_0,h_0) \in \R^3
\]
is given as follows:
If $(f_0,h_0)=(0,0)$ then $(f(t),g(t),h(t))=(0,g_0,0)$.
If $(f_0,g_0)=(0,0)$ then $(f(t),g(t),h(t))=(0,0,h_0)$.
In the other case, $f(t)$ and $g(t)$ are given in terms of $h(t)$ as
\[
		f(t) = -  h'(t),
 \qquad g(t)= \tfrac{1}{2} (h(t)^2 - h_0^2) + g_0,
\]
respectively.
Further, $h(t)$ is given as follows:
Let $P$, $R_{fg}$ be positive constants such that
\[
	P^2 = \tfrac{(f_0^2+g_0^2)^{\frac12} - g_0}2 + \tfrac{h_0^2}4, \quad
	R_{fg}^2 =  (f_0^2 + g_0^2)^{\frac12}.
\]
\begin{itemize}
\item
If $h_0^2 < 2 (\sqrt{f_0^2+g_0^2} + g_0 )$ then
\[
	h(t) =  2 P  \cn \(R_{fg} t + t_0, \tfrac{P^2}{R_{fg}^2}  \)
\]
where $t_0$ is given by
$\cn (t_0, P^2/R_{fg}^2) = h_0/2P$ and  $\sign (\sn (t_0, P^2/R_{fg}^2)) = - \sign f_0$.
\item
If $h_0^2 = 2 (\sqrt{f_0^2+g_0^2} + g_0 )$ then
$h_0\neq0$ holds and $h(t)$ is given by
\[
	h(t) =  (\sign h_0) 2 P \sech (P t+t_0),
\]
where
$t_0$ is given by $ t_0 = \sign (f_0 h_0) \cosh^{-1} (2P/|h_0|)$.
\item
If $h_0^2 > 2 (\sqrt{f_0^2+g_0^2} + g_0 )$ then
\[
	h(t) = (\sign h_0 ) 2 P  \dn \( (\sign h_0 ) P t + t_0 , \tfrac{R_{fg}^2}{P^2} \),
\]
where $t_0$ is given by 
$ \dn (t_0, R_{fg}^2/P^2) =|h_0|/2P$ and $ \sign \dn' (t_0) = - \sign f_0$.
\end{itemize}
\end{lemma}

\begin{proposition}
The two points $(\mathcal{D},\mathcal{R}, \mathcal{I})=\pm(\rho,0,0)$ are fixed points of \eqref{E:qqq_8}.
Further, if $|p_2| \ge p_4$ then there exist more fixed points $(\mathcal{D},\mathcal{R}, \mathcal{I})=(-\frac{p_4}{p_2}\rho,0,\pm \sqrt{1-(\frac{p_4}{p_2})^2}\rho)$.
In other case,
the triplet 
\[
	(f,g,h) :=
	(4p_2p_4 \rho \mathcal{R}, 4p_4 \rho (p_2 \mathcal{D} + p_4\rho), -2p_2 \mathcal{I} )
\]
 solves \eqref{E:ellipticODE2}
and hence the solution $(\mathcal{D},\mathcal{R},\mathcal{I})$ to \eqref{E:qqq_8} is written explicitly in terms of the Jacobi elliptic functions as in Lemma \ref{L:ellipticODE2}.
\end{proposition}
Combining the proposition with Theorems \ref{T:KS} and \ref{T:main}, we obtain the asymptotic profile of  solutions to \eqref{E:NLS8}.

\begin{remark}
The orbit of the solution  $(\mathcal{D}, \mathcal{R}, \mathcal{I})$ is easily understood:
Since $(\mathcal{D}+\frac{p_4}{p_2} \rho)^2 + \mathcal{R}^2$ and $p_2\mathcal{I}^2 - 2p_4 \rho \mathcal{D}$
are conserved quantities,
the orbit is a subset of the intersection of $S_\rho^2$ and the boundary of the cylinder $\{ (x+\frac{p_4}{p_2}\rho)^2 + y^2 =
(\mathcal{D}(0)+\frac{p_4}{p_2} \rho)^2 + \mathcal{R}(0)^2 \}$ or of the
parabolic cylinder $\{ p_2 z^2 - 2 p_4 \rho x  =
p_2\mathcal{I}(0)^2 - 2p_4 \rho \mathcal{D}(0) \}$
. 
\end{remark}

\subsubsection{Case 9}

We turn to the study of the combination of $p_3$ and $p_4$. The NLS system is
\begin{equation}\label{E:NLS9}
	\left\{
\begin{aligned}
	(i \partial_t  + \partial_x^2)u_1={}&
(p_3+2p_4) |u_1|^2 u_1 
-p_3 (2 u_1 |u_2|^2 + \ol{u_1}u_2^2)  +\mathcal{V}(u_1,u_2) u_1, \\
	(i \partial_t  + \partial_x^2)u_2={}&
-p_3(2|u_1|^2 u_2 + u_1^2\ol{u_2}) + (p_3-2p_4) |u_2|^2u_2
 + \mathcal{V}(u_1,u_2)u_2.
\end{aligned}
\right.
\end{equation}
The corresponding quadratic system \eqref{E:qqq} takes the form
\begin{equation}\label{E:qqq_9}
	\mathcal{D}' = -2p_3 \mathcal{I} \mathcal{R}, \quad
	\mathcal{R}' = -2p_3 \mathcal{I} \mathcal{D} - 2p_4 \rho \mathcal{I}, \quad
	\mathcal{I}' = 4p_3 \mathcal{D} \mathcal{R} + 2p_4 \rho \mathcal{R}.
\end{equation}
This is a variant of the ODE systems for the Jacobi elliptic functions.
It is reduced to 
\begin{equation}\label{E:ellipticODE3}
	f' = -gh, \quad g'= fh, \quad h' = - (f + \eta )g,
\end{equation}
where $\eta>0$ is a constant. 
Since the complete description of the solution to this system is lengthy, we state
the result for \eqref{E:qqq_9}  beforehand.
\begin{proposition}
The two points $(\mathcal{D},\mathcal{R}, \mathcal{I})=\pm(\rho,0,0)$ are fixed points of \eqref{E:qqq_9}.
Moreover, if $p_4 \le 2p_3$ then 
\[
	(\mathcal{D},\mathcal{R}, \mathcal{I})=(-\tfrac{p_4}{2p_3}\rho,\pm \sqrt{1-(\tfrac{p_4}{2p_3})^2}\rho,0)
\]
are fixed points.
Furthermore, 
if $p_4 \le p_3$ then 
\[
	(\mathcal{D},\mathcal{R}, \mathcal{I})= ( -\tfrac{p_4}{p_3}\rho , 0 , \pm \sqrt{1- (\tfrac{ p_4 }{  p_3 })^2  }\rho)
\]
are also fixed points.
The triplet  
\[
	(f,g,h) =
	( 2\sqrt2 (p_3 \mathcal{D} + \tfrac{p_4 \rho }2 ) ,  2p_3 \mathcal{I}  ,  2\sqrt{2} p_3\mathcal{R})
\]
 solves \eqref{E:ellipticODE3} with $\eta =\sqrt2 p_4 \rho$
and hence the solution $(\mathcal{D},\mathcal{R},\mathcal{I})$ to \eqref{E:ODE} is given explicitly by means of Lemma \ref{L:ellipticODE3} below.
\end{proposition}
The ODE system \eqref{E:ellipticODE3} is integrable. Hence,
combining the proposition with Theorems \ref{T:KS} and \ref{T:main}, we obtain the asymptotic profile of  solutions to \eqref{E:NLS9}.

Let us give the formula for the solution to \eqref{E:ellipticODE3}.
\begin{lemma}\label{L:ellipticODE3}
Let $\eta>0$.
Let
$\ell_1 := \{(x,0,0) \ | \ x \in \R \}$, 
$\ell_2 := \{(0,0,z) \ | \ z \in \R \}$,
and $\ell_3 := \{(-\eta,y,0) \ | \ y \in \R \}$
and let $\mathcal{P}:= \ell_1 \cup \ell_2 \cup \ell_3$.
For $R_0 \ge 0$ and $K_0 \in \R$, let
$\Upsilon=\Upsilon(R_0):=\{(x,y,z) \in \R^3 \ |\ x^2 + y^2 = R_0^2 \}$ and $\Sigma=\Sigma(K_0):=\{(x,y,z) \in \R^3 \ |\ z^2 - (x+\eta)^2 = K_0 \}$.
Then the following holds:
\begin{enumerate}
\item
The set of stationary points of \eqref{E:ellipticODE3}
is $\mathcal{P}$.
\item For any solution $(f,g,h)(t)$ to \eqref{E:ellipticODE3},
$R=R(f,g,h):= \sqrt{f(t)^2+g(t)^2}$ and $K=K(f,g,h):= h(t)^2 - (f(t) + \eta)^2$ are conserved.
\item
Given $R_0$ and $K_0$, the orbit of a solution $(f,g,h)$ such that $R(f,g,h)=R_0$ and $K(f,g,h)=K_0$
is a subset of $\Upsilon \cap \Sigma$.
\item If $(f,g,h) (t)$ is a solution then $(f,-g,-h)(t)$ is also a solution.
Further, $R(f,g,h)=R(f,-g,-h)$ and $K(f,g,h)=K(f,-g,-h)$.
\item A solution $(f,g,h)(t)$ such that $R(f,g,h)=R_0>0$, $K(f,g,h)=K_0$, $(f,g,h)(0) \not\in \mathcal{P}$, and $h(0) \ge 0$
is described as follows:
\begin{enumerate}
\item If $K_0>0$ then
\[
	f = R_0 \tfrac{ -\xi +  \cn \(\theta t + t_0 , m_0 \) }{ 1 - \xi \cn \(\theta t + t_0, m_0  \)}, \quad
	g  
	= R_0 \tfrac{  \sqrt{1-\xi^2} \sn \(\theta t + t_0 , m_0 \)}{ 1 - \xi  \cn \(\theta t + t_0 , m_0 \) },
\]
and
\[
	h =   \theta \tfrac{  \sqrt{1-\xi^2} \dn \(\theta t + t_0 , m_0 \)}{ 1 - \xi   \cn \(\theta t + t_0 , m_0 \) }
\]
for some $t_0 \in \R$, where
\[
	\theta =  ((R_0+\eta)^2 + K_0)^{\frac14} ((R_0-\eta)^2 + K_0)^{\frac14},
\]
and
\[
	\xi = \tfrac{ 2 \eta R_0 } { K_0 + \eta^2+ R_0^2 + \theta^2 } \in (0,1), \quad
	m_0 =  \tfrac{ \theta^2+ R_0^2  - K_0 - \eta^2  }{ 2\theta^2 }.
\]
\item If $K_0=0$ then we have three subcases:
\begin{enumerate}
\item If $R_0< \eta$ then
\begin{equation*}
	f(t)+ \eta = h(t) =  \tfrac{\eta^2 - R_0^2 }{ \eta - R_0 \cos \( \sqrt{\eta^2-R_0^2} t+ t_0 \) }
,\qquad
	g(t)= \tfrac{R_0 \sqrt{\eta^2 - R_0^2 } \sin \( \sqrt{\eta^2-R_0^2} t+ t_0 \) }{ \eta - R_0 \cos \( \sqrt{\eta^2-R_0^2} t+ t_0 \) }
\end{equation*}	
for some $t_0 \in \R$;
\item If $R_0= \eta$ then
\begin{equation*}
	f(t) + \eta = h(t) = \tfrac{ 2\eta }{ 1+ (\eta t + t_0)^2},\qquad
	g(t) = \tfrac{2(\eta t + t_0) }{ 1+ (\eta t + t_0)^2}
\end{equation*}	
for some $t_0 \in \R$;
\item If $R_0> \eta$ then
\begin{equation*}
	f(t) + \eta = h(t) = \tfrac{R_0^2-\eta^2}{ R_0 \cosh (t \sqrt{R_0^2-\eta^2} + t_0) - \eta }
,\qquad
	g(t) = \tfrac{ R_0 \sqrt{R_0^2 -\eta^2} \sinh (t \sqrt{R_0^2-\eta^2} + t_0) }{ R_0 \cosh (t \sqrt{R_0^2-\eta^2} + t_0) - \eta }
\end{equation*}	
for some $t_0 \in \R$;
\end{enumerate}
\item If $K_0<0$ then we have six subcases.
In this case, $\Sigma= \Sigma_+ \cup \Sigma_-$, where
$\Sigma_\pm := \{ (x,y,z) \in \Sigma \ |\ \pm x >0 \}$.
\begin{enumerate}
\item If $R_0 > \eta + \sqrt{-K_0}$ and $(f,g,h) \in \Sigma_-$ then
\begin{equation*}
	f(t)=R_0 \tfrac{ -(R_0 + \eta + \sqrt{-K_0}) +(R_0 - \eta - \sqrt{-K_0}) \sn^2\( \theta t+t_0 , m_0 \) }{(R_0 + \eta + \sqrt{-K_0}) +(R_0 - \eta - \sqrt{-K_0}) \sn^2\( \theta t+t_0 , m_0 \)},
\end{equation*}
\begin{equation*}
	g (t)=  
	\tfrac{ 2R_0 \sqrt{R_0^2 - (\eta + \sqrt{-K_0})^2} \sn\( \theta t+t_0 , m_0 \) }{{(R_0 + \eta + \sqrt{-K_0}) +(R_0 - \eta - \sqrt{-K_0}) \sn^2\( \theta t+t_0 , m_0 \)}},
\end{equation*}
and
\begin{equation*}
	h(t)=
	\tfrac{  (R_0 + \eta + \sqrt{-K_0})\sqrt{(R_0-\eta)^2 + K_0}   \cn\( \theta t+t_0 , m_0 \) \dn\( \theta t+t_0 , m_0 \)}
	{{(R_0 + \eta + \sqrt{-K_0}) +(R_0 - \eta - \sqrt{-K_0}) \sn^2\( \theta t+t_0 , m_0 \)}} 
\end{equation*}
for some $t_0 \in \R$,
where
\[
	\theta = \tfrac{ \sqrt{ (R_0+\sqrt{-K_0})^2 - \eta^2 }}{2}, \quad 
	m_0 = {\tfrac{(R_0-\sqrt{-K_0})^2 - \eta^2}{(R_0+\sqrt{-K_0})^2 - \eta^2}} ;
\]
\item If $R_0 > \eta + \sqrt{-K_0}$ and $(f,g,h) \in \Sigma_+$ then
\begin{equation*}
	f(t)=-\eta  + \tfrac{ \sqrt{-K_0}( (R_0 + \eta + \sqrt{-K_0})+(R_0 + \eta -\sqrt{-K_0}) \sn^2\( \theta t+t_0 , m_0 \) }{(R_0 + \eta + \sqrt{-K_0}) -(R_0 + \eta -\sqrt{-K_0}) \sn^2\( \theta t+t_0 , m_0 \)},
\end{equation*}
\begin{equation*}
	g(t) = - \tfrac{ R_0 (R_0 + \eta + \sqrt{-K_0}) \sqrt{R_0^2 -(\eta - \sqrt{-K_0})^2} \cn\( \theta t+t_0 , m_0 \)\dn\( \theta t+t_0 ; k_0 \) }{(R_0 + \eta + \sqrt{-K_0}) -(R_0 + \eta -\sqrt{-K_0}) \sn^2\( \theta t+t_0 , m_0 \)},
\end{equation*}
and
\begin{equation*}
	h(t)= \tfrac{ 2 \sqrt{-K_0} \sqrt{(R_0 + \eta)^2 + K_0} \sn\( \theta t+t_0 , m_0 \)}
	{(R_0 + \eta + \sqrt{-K_0}) -(R_0 + \eta -\sqrt{-K_0}) \sn^2\( \theta t+t_0 , m_0 \)}
\end{equation*}
for some $t_0 \in \R$, where
\[
	\theta = \tfrac{ \sqrt{ (R_0+\sqrt{-K_0})^2 - \eta^2 }}{2}, \quad
	m_0 = {\tfrac{(R_0-\sqrt{-K_0})^2 - \eta^2}{(R_0+\sqrt{-K_0})^2 - \eta^2}};
\]
\item If $R_0 = \eta + \sqrt{-K_0}$ then
\begin{equation*}
	f(t)= R_0 -  \tfrac{   2 R_0\eta \sin^2  ( t \sqrt{ R_0 (R_0-\eta)}+t_0) }{ R_0 -\eta  \cos^2 ( t \sqrt{ R_0 (R_0-\eta)}+t_0)  } ,
\quad
	g(t)= \tfrac{ 2R_0 \sqrt{\eta (R_0 - \eta )}  \sin  ( t \sqrt{ R_0 (R_0-\eta)}+t_0) }{ R_0 - \eta   \cos^2 ( t \sqrt{ R_0 (R_0-\eta)}+t_0)  },
\end{equation*}
and
\begin{equation*}
	h(t) =   \tfrac{ 2(R_0 - \eta )\sqrt{R_0\eta}  \cos  ( t \sqrt{ R_0 (R_0-\eta)}+t_0) }{ R_0  - \eta  \cos^2 (t  \sqrt{ R_0 (R_0-\eta)}+t_0) }
\end{equation*}
for some $t_0 \in \R$;
\item If $\eta - \sqrt{-K_0} <R_0 < \eta + \sqrt{-K_0}$ then
\begin{equation*}
	f (t)=-R_0  + \tfrac{ 2 R_0 (R_0 - \eta + \sqrt{-K_0}) }{2R_0 -(R_0 + \eta - \sqrt{-K_0}) \sn^2\( \theta t+t_0 , m_0 \)},
\quad
	g(t)= \tfrac{ 2 R_0 \sqrt{R_0^2 - (\eta - \sqrt{-K_0})^2} \cn\( \theta t+t_0 , m_0 \)}{2R_0 -(R_0 + \eta - \sqrt{-K_0}) \sn^2\( \theta t+t_0 , m_0 \)},
\end{equation*}
and
\begin{equation*}
	h(t) = -\tfrac{ 2 R_0^{1/2} (-K_0)^{1/4}\sqrt{R_0^2 - (\eta - \sqrt{-K_0})^2} \sn\( \theta t+t_0 , m_0 \)\dn\( \theta t+t_0 , m_0 \)  }{2R_0 -(R_0 + \eta - \sqrt{-K_0}) \sn^2\( \theta t+t_0 , m_0 \)} 
\end{equation*}
for some $t_0 \in \R$,
where
\[
	\theta =R_0^{1/2} (-K_0)^{1/4} , \quad
	m_0 = \tfrac{{\eta^2 - (R_0-\sqrt{-K_0})^2}}{4 R_0 \sqrt{-K_0}} ;
\]
\item If $R_0 = \eta - \sqrt{-K_0}$ then
\begin{equation*}
	f(t)= -R_0 + \tfrac{ 2 R_0 (\eta-R_0) }{\eta \cosh^2 ( t \sqrt{R_0(\eta-R_0)}  + t_0) -R_0},
\quad
	g(t)=  \tfrac{2 R_0  \sqrt{\eta (\eta-R_0)} \sinh (t \sqrt{R_0 (\eta-R_0)}  + t_0) }{\eta \cosh^2 (t \sqrt{R_0 (\eta-R_0)}  + t_0) - R_0},
\end{equation*}
and
\begin{equation*}
	h(t) =  \tfrac{ R_0\sqrt{2\eta (\eta-R_0)} \cosh (t \sqrt{R_0 (\eta-R_0)}  + t_0) }{\eta \cosh^2 (t \sqrt{R_0 (\eta-R_0)}  + t_0) - R_0}
\end{equation*}
for some $t_0 \in \R$;
\item If $R_0 < \eta - \sqrt{-K_0}$ then
\begin{equation*}
	f(t) =-R_0 + \tfrac{  2R_0(-R_0+\eta- \sqrt{-K_0}) \sn^2\( \theta t+t_0 , m_0 \)}{(R_0+\eta - \sqrt{-K_0}) -2R_0 \sn^2\( \theta t+t_0 , m_0 \)},
\quad
	g(t)= \tfrac{ - 2R_0 \sqrt{(\eta- \sqrt{-K_0})^2 -R_0^2} \sn\( \theta t+t_0 , m_0 \) \cn\( \theta t+t_0 , m_0 \) }{(R_0+\eta - \sqrt{-K_0}) -2R_0 \sn^2\( \theta t+t_0 ,m_0 \)},
\end{equation*}
and
\begin{equation*}
	h(t) = \tfrac{  (R_0 + \eta- \sqrt{-K_0})\sqrt{(R_0-\eta)^2 + K_0} \dn\( \theta t+t_0 , m_0 \)}{(R_0+\eta - \sqrt{-K_0}) -2R_0 \sn^2\( \theta t+t_0 , m_0 \)}
\end{equation*}
for some $t_0 \in \R$, where
\[
	\theta = \tfrac{ \sqrt{ \eta^2 - (R_0 - \sqrt{-K_0})^2 }}{2} , \quad
	m_0 = {\tfrac{4 R_0 \sqrt{-K_0}}{\eta^2 - (R_0 - \sqrt{-K_0})^2}} .
\]
\end{enumerate}
\end{enumerate}
\end{enumerate}
\end{lemma}
Note that the explicit formula of
a solution $(f,g,h)(t)$ to \eqref{E:ellipticODE3}
such that $R(f,g,h)=R_0>0$, $K(f,g,h)=K_0$, $(f,g,h)(0) \not\in \mathcal{P}$, and $h(0) \le 0$
is obtained by combining properties (4) and (5).

\subsubsection{Case 10}

The next case is the combination of $p_3$ and $p_5$.
The NLS system takes the form
\begin{equation}\label{E:NLS10}
\left\{
\begin{aligned}
	(i \partial_t  + \partial_x^2)u_1={}&
p_3 |u_1|^2 u_1 + p_5(2|u_1|^2 u_2 + u_1^2\ol{u_2})-p_3 (2 u_1 |u_2|^2 + \ol{u_1}u_2^2) + p_5 |u_2|^2u_2
+\mathcal{V}(u_1,u_2) u_1, \\
	(i \partial_t  + \partial_x^2)u_2={}&
p_5 |u_1|^2 u_1-p_3(2|u_1|^2 u_2 + u_1^2\ol{u_2})+p_5 (2 u_1 |u_2|^2 + \ol{u_1}u_2^2) + p_3 |u_2|^2u_2
 + \mathcal{V}(u_1,u_2)u_2
\end{aligned}
\right.
\end{equation}
and the quadratic system \eqref{E:qqq} takes the form
\begin{equation}\label{E:qqq_10}
	\mathcal{D}' = -2p_3 \mathcal{I} \mathcal{R} +  2p_5 \rho \mathcal{I}, \quad
	\mathcal{R}' = -2p_3 \mathcal{I} \mathcal{D} , \quad
	\mathcal{I}' = 4p_3 \mathcal{D} \mathcal{R} - 2p_5 \rho \mathcal{D}.
\end{equation}
This ODE system is also reduced to \eqref{E:ellipticODE3}.
\begin{proposition}
The two points $(\mathcal{D},\mathcal{R}, \mathcal{I})=\pm(0,\rho,0)$ are fixed points of \eqref{E:qqq_10}.
Moreover, if $p_5 \le 2p_3$ then 
\[
	(\mathcal{D},\mathcal{R}, \mathcal{I})=(\pm \sqrt{1-(\tfrac{p_5}{2p_3})^2}\rho,\tfrac{p_5}{2p_3}\rho,0)
\]
are fixed points.
Furthermore, if $p_5 \le p_3$ then 
\[
	(\mathcal{D},\mathcal{R}, \mathcal{I})= ( 0,\tfrac{p_5}{p_3}\rho  , \pm \sqrt{1- (\tfrac{ p_5 }{ p_3})^2  }\rho)
\]
are also fixed points.
The triplet  
\[
	(f,g,h) =
	( -2\sqrt2 (p_3 \mathcal{R} - \tfrac{p_5}2 \rho  ) ,  -2p_3 \mathcal{I}  ,  2\sqrt{2} p_3\mathcal{D})
\]
 solves \eqref{E:ellipticODE3} with $\eta = \sqrt2 p_5 \rho$
and hence the solution $(\mathcal{D},\mathcal{R},\mathcal{I})$ to \eqref{E:qqq_10} is given explicitly by means of Lemma \ref{L:ellipticODE3}.
\end{proposition}
Combining the proposition with Theorems \ref{T:KS} and \ref{T:main}, we obtain the asymptotic profile of  solutions to \eqref{E:NLS10}.

\subsection{Room 3 -- specific integrable combinations}

Let us finally collect the cases where we can integrate the ODE system under a specific relation between the parameters.

\subsubsection{Case 11}

Let us consider the combination of $p_1$ and $p_3$.
The NLS system is
\begin{equation}\label{E:NLS11}
\left\{
\begin{aligned}
	(i \partial_t  + \partial_x^2)u_1={}&
p_3 |u_1|^2 u_1 +p_1 (2|u_1|^2 u_2 + u_1^2\ol{u_2})
-p_3 (2 u_1 |u_2|^2 + \ol{u_1}u_2^2) - p_1 |u_2|^2u_2
	\\& - 4p_1 \Re (\overline{u_1} u_2) u_1 +\mathcal{V}(u_1,u_2) u_1, \\
	(i \partial_t  + \partial_x^2)u_2={}&
p_1 |u_1|^2 u_1-p_3(2|u_1|^2 u_2 + u_1^2\ol{u_2})-p_1 (2 u_1 |u_2|^2 + \ol{u_1}u_2^2) + p_3 |u_2|^2u_2
	\\& + 4p_1 \Re (\overline{u_1} u_2) u_2 + \mathcal{V}(u_1,u_2)u_2.
\end{aligned}
\right.
\end{equation}
We have the ODE system
\begin{equation}\label{E:qqq_11}
	\mathcal{D}' = 2 \mathcal{I}(p_1 \mathcal{D}-p_3  \mathcal{R}), \quad
	\mathcal{R}' = 2 \mathcal{I} (p_1\mathcal{R}-p_3  \mathcal{D}) , \quad
	\mathcal{I}' = -2p_1 ( \mathcal{D}^2 + \mathcal{R}^2 ) +4p_3 \mathcal{D} \mathcal{R}.
\end{equation}
In this case, the effect by $p_1$ and $p_3$ compete with each other.
One can integrate the system at least when $p_1/p_3$ takes specific values.

The first case is the balanced case $p_1=p_3$, which turns out to be the threshold case.
\begin{proposition}
Suppose $p_1=p_3$.
Then $\{(2^{-\frac12}\rho \cos \theta, 2^{-\frac12}\rho \cos \theta , \rho \sin \theta)| \theta \in \R/2\pi \Z \}$ is the set of fixed points of \eqref{E:qqq_11}.
If $\mathcal{D}(0)\neq \mathcal{R}(0)$ then the solution of \eqref{E:qqq_11} is given a follows:
\begin{align*}
	\mathcal{D}(t) = \tfrac{\mathcal{D}(0)+\mathcal{R}(0)}{2}+ \sqrt{\tfrac{(\mathcal{D}(0)-\mathcal{R}(0))^2}{(\mathcal{D}(0)-\mathcal{R}(0))^2+2 \mathcal{I}(0)^2} } \tfrac{\mathcal{D}(0)-\mathcal{R}(0)}{2}
	\cosh \( 4p_1 t \sqrt{\tfrac{2 \mathcal{I}(0)^2+ (\mathcal{D}(0)-\mathcal{R}(0))^2}{2} } + \tau_0 \),
\end{align*}
\begin{align*}
	\mathcal{R}(t) = \tfrac{\mathcal{D}(0)+\mathcal{R}(0)}{2} - \sqrt{\tfrac{(\mathcal{D}(0)-\mathcal{R}(0))^2}{(\mathcal{D}(0)-\mathcal{R}(0))^2+2 \mathcal{I}(0)^2} } \tfrac{\mathcal{D}(0)-\mathcal{R}(0)}{2}
	\cosh \( 4p_1 t \sqrt{\tfrac{2 \mathcal{I}(0)^2+ (\mathcal{D}(0)-\mathcal{R}(0))^2}{2} } + \tau_0 \),
\end{align*}
and
\[
	\mathcal{I}(t)= - \sqrt{ \tfrac{2 \mathcal{I}(0)^2+(\mathcal{D}(0)-\mathcal{R}(0))^2}{2} }
	\tanh \( 4p_1 t \sqrt{\tfrac{2 \mathcal{I}(0)^2+ (\mathcal{D}(0)-\mathcal{R}(0))^2}{2} } + \tau_0 \),
\]
where $\tau_0 = -(\sign \mathcal{I}(0)) \cosh^{-1} \sqrt{1+\tfrac{2 \mathcal{I}(0)^2}{(\mathcal{D}(0)-\mathcal{R}(0))^2} }$.
\end{proposition}
The next case is in the $p_1$-dominant region, i.e., $p_1>p_3$.
\begin{proposition}
Suppose $p_1=3p_3$.
There are two fixed points $\pm(0,0,\rho)$ of \eqref{E:qqq_11}.
If the initial point is not the fixed point then the solution of \eqref{E:qqq_11} is given a follows:
\begin{align*}
	\mathcal{D}(t) ={}& \tfrac{\mathcal{D}(0)+\mathcal{R}(0)}{2}\tfrac{ 2\rho }{( {\sqrt{ 8\rho^2(\mathcal{D}(0)-\mathcal{R}(0))^2+ (\mathcal{D}(0)+\mathcal{R}(0))^4 }}\cosh (8p_3 \rho \tau + \tau_0) + {(\mathcal{D}(0)+\mathcal{R}(0))^2} )^{1/2} }\\
	&+  \tfrac{\mathcal{D}(0)-\mathcal{R}(0)}{2}
	\tfrac{ 4\rho^2 }{ {\sqrt{ 8\rho^2(\mathcal{D}(0)-\mathcal{R}(0))^2+ (\mathcal{D}(0)+\mathcal{R}(0))^4 }}\cosh (8p_3 \rho \tau + \tau_0) + {(\mathcal{D}(0)+\mathcal{R}(0))^2}  } ,
\end{align*}
\begin{align*}
	\mathcal{R}(t) ={}& \tfrac{\mathcal{D}(0)+\mathcal{R}(0)}{2}\tfrac{ 2\rho }{( {\sqrt{ 8\rho^2(\mathcal{D}(0)-\mathcal{R}(0))^2+ (\mathcal{D}(0)+\mathcal{R}(0))^4 }}\cosh (8p_3 \rho \tau + \tau_0) + {(\mathcal{D}(0)+\mathcal{R}(0))^2} )^{1/2} }\\
	&-  \tfrac{\mathcal{D}(0)-\mathcal{R}(0)}{2}
	\tfrac{ 4\rho^2 }{ {\sqrt{ 8\rho^2(\mathcal{D}(0)-\mathcal{R}(0))^2+ (\mathcal{D}(0)+\mathcal{R}(0))^4 }}\cosh (8p_3 \rho \tau + \tau_0) + {(\mathcal{D}(0)+\mathcal{R}(0))^2}  } ,
\end{align*}
and
\[
	\mathcal{I}(\tau)= - \tfrac{ \rho {\sqrt{ 8\rho^2(\mathcal{D}(0)-\mathcal{R}(0))^2+ (\mathcal{D}(0)+\mathcal{R}(0))^4 }} \sinh (8p_3 \rho \tau + \tau_0)}{ {\sqrt{ 8\rho^2(\mathcal{D}(0)-\mathcal{R}(0))^2+ (\mathcal{D}(0)+\mathcal{R}(0))^4 }}\cosh (8p_3 \rho \tau + \tau_0) + {(\mathcal{D}(0)+\mathcal{R}(0))^2}  }  ,
\]
where $\tau_0 = -(\sign \mathcal{I}(0)) \cosh^{-1} \frac{4\rho^2 - (\mathcal{D}(0)+\mathcal{R}(0))^2}{\sqrt{ 8\rho^2(\mathcal{D}(0)-\mathcal{R}(0))^2+ (\mathcal{D}(0)+\mathcal{R}(0))^4 }} =-\tanh^{-1} \frac{\rho \mathcal{I}(0)}{4\rho^2 -(\mathcal{D}(0)+\mathcal{R}(0))^2}$.
\end{proposition}
The last case is in the $p_3$-dominant region, i.e., $p_1<p_3$.
\begin{proposition}
Suppose $p_1=\frac13 p_3$.
Let $c_\pm := \frac12 (\mathcal{D}(0)\mp \mathcal{R}(0))^2$.
\begin{enumerate}
\item If $c_+=0= c_-$ then the solution of \eqref{E:qqq_11} is
$(\mathcal{D}(\tau),\mathcal{R}(\tau),\mathcal{I}(\tau))=(0,0,\pm \rho)$.
\item If $c_+=0< c_-$ then the solution of \eqref{E:qqq_11} is given as
\[
	\mathcal{D}(\tau) = \mathcal{R}(\tau) = (\sign \mathcal{D}(0) )\tfrac{\rho}{ \sqrt2 \cosh (\frac43 p_3 \rho \tau + \tau_0)},
\quad
	\mathcal{I}(\tau) = \rho \tanh (\tfrac43 p_3 \rho \tau + \tau_0),
\]
where $\tau_0 = \tanh^{-1} \frac{\mathcal{I}(0)}\rho=( \sign \mathcal{I}(0)) \cosh^{-1} \frac{\rho}{ \sqrt{ 2|\mathcal{D}(0)|}}$.
\item If $c_+> 0= c_-$ then the solution of \eqref{E:qqq_11} is given as
\[
	\mathcal{D}(\tau) =- \mathcal{R}(\tau) = (\sign \mathcal{D}(0) )\tfrac{\rho}{ \sqrt2 \cosh (\frac83 p_3 \rho \tau - \tau_0)},
\quad
	\mathcal{I}(\tau) = - \rho \tanh (\tfrac83 p_3 \rho \tau - \tau_0),
\]
where $\tau_0 = \tanh^{-1} \frac{\mathcal{I}(0)}\rho=( \sign \mathcal{I}(0)) \cosh^{-1} \frac{\rho}{ \sqrt{2 |\mathcal{D}(0)|}}$.
\item If $c_+ , c_->0$ and $\mathcal{I}(0)=0$ then we have three subcases:
\begin{enumerate}
\item If $2c_+=c_-$ then $(\mathcal{D}(\tau),\mathcal{R}(\tau),\mathcal{I}(\tau))$ is constant solution of \eqref{E:qqq_11}. The constant is either one of
 $
\pm (\tfrac{\sqrt2-1}{\sqrt6}\rho,\tfrac{\sqrt2+1}{\sqrt6}\rho,0)$, $\pm (\tfrac{\sqrt2+1}{\sqrt6}\rho,\tfrac{\sqrt2-1}{\sqrt6}\rho,0).
$
\item If $2c_+>c_-$ then the solution of \eqref{E:qqq_11} is given as
\begin{align*}
	\mathcal{D} (\tau) ={}& \tfrac{\mathcal{D}(0)+\mathcal{R}(0)}{2} \( \tfrac{ \beta }{ 1- (1-\beta) \cd^2(\frac43 p_3 \sqrt{c_+(\beta-\alpha)} t, \frac{ -\alpha(1-\beta)}{ \beta-\alpha }  ) }
 \)^{-\frac12} \\
	&+ \tfrac{\mathcal{D}(0)-\mathcal{R}(0)}{2}
	\(\tfrac{ \beta }{ 1- (1-\beta) \cd^2(\frac43 p_3 \sqrt{c_+(\beta-\alpha)} t, \frac{ -\alpha(1-\beta)}{ \beta-\alpha }  ) }
\), \\
	\mathcal{R} (\tau) ={}& \tfrac{\mathcal{D}(0)+\mathcal{R}(0)}{2} \( \tfrac{ \beta }{ 1- (1-\beta) \cd^2(\frac43 p_3 \sqrt{c_+(\beta-\alpha)} t, \frac{ -\alpha(1-\beta)}{ \beta-\alpha }  ) }
 \)^{-\frac12} \\
	&- \tfrac{\mathcal{D}(0)-\mathcal{R}(0)}{2}
	\( \tfrac{ \beta }{ 1- (1-\beta) \cd^2(\frac43 p_3 \sqrt{c_+(\beta-\alpha)} t, \frac{ -\alpha(1-\beta)}{ \beta-\alpha }  ) }
 \),
\end{align*}
and
\begin{align*}
	\mathcal{I}(\tau) ={}& - \sqrt{\tfrac{c_+}{\beta-\alpha}} \beta(1-\alpha)(1-\beta) \\
	&\times \tfrac{ \cd(\frac43 p_3 \sqrt{c_+(\beta-\alpha)} t, \frac{ -\alpha(1-\beta)}{ \beta-\alpha }  )\sd(\frac43 p_3 \sqrt{c_+(\beta-\alpha)} t, \frac{ -\alpha(1-\beta)}{ \beta-\alpha }  )\nd(\frac43 p_3 \sqrt{c_+(\beta-\alpha)} t, \frac{ -\alpha(1-\beta)}{ \beta-\alpha }  ) }{ 1- (1-\beta) \cd^2(\frac43 p_3 \sqrt{c_+(\beta-\alpha)} t, \frac{ -\alpha(1-\beta)}{ \beta-\alpha }  ) },
\end{align*}
where $\alpha = -\frac12 - \frac{ \sqrt{c_+(c_+ + 4c_-)}}{2c_+} \in (-2,-1)$ and $\beta= -\frac12 + \frac{ \sqrt{c_+(c_+ + 4c_-)}}{2c_+} \in (0,1)$ are two roots of the quadratic equation $c_+ w^2 + c_+ w - c_-=0$.

\item If $2c_+<c_-$ then the solution of \eqref{E:qqq_11} is given as 
\begin{align*}
	\mathcal{D} (\tau) ={}& \tfrac{\mathcal{D}(0)+\mathcal{R}(0)}{2} \(\tfrac{ \gamma }{ \gamma -  (\gamma-1) \sn^2(\frac43 p_3 \sqrt{c_+\gamma(1-\alpha)} \tau, \frac{ (\gamma-1)(-\alpha) }{ \gamma(1-\alpha) } ) }\)^{-\frac12} \\
	&+ \tfrac{\mathcal{D}(0)-\mathcal{R}(0)}{2}
	\( \tfrac{ \gamma }{ \gamma -  (\gamma-1) \sn^2(\frac43 p_3 \sqrt{c_+\gamma(1-\alpha)} \tau, \frac{ (\gamma-1)(-\alpha) }{ \gamma(1-\alpha) } ) } \), \\
	\mathcal{R} (\tau) ={}& \tfrac{\mathcal{D}(0)+\mathcal{R}(0)}{2} \( \tfrac{ \gamma }{ \gamma -  (\gamma-1) \sn^2(\frac43 p_3 \sqrt{c_+\gamma(1-\alpha)} \tau, \frac{ (\gamma-1)(-\alpha) }{ \gamma(1-\alpha) } ) } \)^{-\frac12} \\
	&- \tfrac{\mathcal{D}(0)-\mathcal{R}(0)}{2}
	\( \tfrac{ \gamma }{ \gamma -  (\gamma-1) \sn^2(\frac43 p_3 \sqrt{c_+\gamma(1-\alpha)} \tau, \frac{ (\gamma-1)(-\alpha) }{ \gamma(1-\alpha) } ) } \),
\end{align*}
and
\begin{align*}
	\mathcal{I}(\tau) ={}& \sqrt{c_+\gamma(1-\alpha)} (\gamma-1) \\
	&\times \tfrac{ \sn(\frac43 p_3 \sqrt{c_+\gamma(1-\alpha)} \tau, \frac{ (\gamma-1)(-\alpha) }{ \gamma(1-\alpha) } ) \cn(\frac43 p_3 \sqrt{c_+\gamma(1-\alpha)} \tau, \frac{ (\gamma-1)(-\alpha) }{ \gamma(1-\alpha) } ) \dn(\frac43 p_3 \sqrt{c_+\gamma(1-\alpha)} \tau, \frac{ (\gamma-1)(-\alpha) }{ \gamma(1-\alpha) } ) }{ \gamma -  (\gamma-1) \sn^2(\frac43 p_3 \sqrt{c_+\gamma(1-\alpha)} \tau, \frac{ (\gamma-1)(-\alpha) }{ \gamma(1-\alpha) } ) },
\end{align*}
where $\alpha = -\frac12 - \frac{ \sqrt{c_+(c_+ + 4c_-)}}{2c_+} <-2)$ and $\gamma= -\frac12 + \frac{ \sqrt{c_+(c_+ + 4c_-)}}{2c_+} >1$ are two roots of the quadratic equation $c_+ w^2 + c_+ w - c_-=0$.
\end{enumerate}
\item If $c_+,c_->0$ and $\mathcal{I}(0)\neq0$ then the solution of \eqref{E:qqq_11} is given as
\begin{align*}
	\mathcal{D} (\tau) ={}& \tfrac{\mathcal{D}(0)+\mathcal{R}(0)}{2} \(\tfrac{ \beta \gamma }{ \gamma -  (\gamma-\beta) \sn^2(\frac43 p_3 \sqrt{c_+\gamma(\beta-\alpha)} \tau + \tau_0, \frac{ (\gamma-\beta)(-\alpha) }{ \gamma(\beta-\alpha) } ) }\)^{-\frac12} \\
	&+ \tfrac{\mathcal{D}(0)-\mathcal{R}(0)}{2}
	\( \tfrac{ \beta \gamma }{ \gamma -  (\gamma-\beta) \sn^2(\frac43 p_3 \sqrt{c_+\gamma(\beta-\alpha)} \tau + \tau_0, \frac{ (\gamma-\beta)(-\alpha) }{ \gamma(\beta-\alpha) } ) } \), \\
	\mathcal{R} (\tau) ={}& \tfrac{\mathcal{D}(0)+\mathcal{R}(0)}{2} \( \tfrac{ \beta \gamma }{ \gamma -  (\gamma-\beta) \sn^2(\frac43 p_3 \sqrt{c_+\gamma(\beta-\alpha)} \tau + \tau_0, \frac{ (\gamma-\beta)(-\alpha) }{ \gamma(\beta-\alpha) } ) } \)^{-\frac12} \\
	&- \tfrac{\mathcal{D}(0)-\mathcal{R}(0)}{2}
	\( \tfrac{ \beta \gamma }{ \gamma -  (\gamma-\beta) \sn^2(\frac43 p_3 \sqrt{c_+\gamma(\beta-\alpha)} \tau + \tau_0, \frac{ (\gamma-\beta)(-\alpha) }{ \gamma(\beta-\alpha) } ) } \),
\end{align*}
and
\begin{align*}
	\mathcal{I}(\tau) ={}& \sqrt{c_+\gamma(\beta-\alpha)} (\gamma-\beta) \\
	&\times \tfrac{ \sn(\frac43 p_3 \sqrt{c_+\gamma(\beta-\alpha)} \tau + \tau_0, \frac{ (\gamma-\beta)(-\alpha) }{ \gamma(\beta-\alpha) } )
	\cn(\frac43 p_3 \sqrt{c_+\gamma(\beta-\alpha)} \tau + \tau_0, \frac{ (\gamma-\beta)(-\alpha) }{ \gamma(\beta-\alpha) } )
	\dn(\frac43 p_3 \sqrt{c_+\gamma(\beta-\alpha)} \tau + \tau_0, \frac{ (\gamma-\beta)(-\alpha) }{ \gamma(\beta-\alpha) } )
	 }{ \gamma -  (\gamma-1) \sn^2(\frac43 p_3 \sqrt{c_+\gamma(\beta-\alpha)} \tau + \tau_0, \frac{ (\gamma-\beta)(-\alpha) }{ \gamma(\beta-\alpha) } ) },
\end{align*}
where $\alpha \in (-\infty,0)$, $\beta \in (0,1)$, and $\gamma>1$ are three roots of the cubic equation
$-c_+ w^3 + \rho^2  w - c_-=0$ and $\tau_0= - \sign(I(0)) \sn^{-1} (\sqrt{\frac{\gamma(1-\beta)}{ \gamma-\beta } })$.
\end{enumerate}
\end{proposition}
Combining the above propositions with Theorems \ref{T:KS} and \ref{T:main}, we obtain the asymptotic profile of  solutions to the corresponding cases of \eqref{E:NLS11}.
A phase portrait analysis shows that the nonlinear synchronization occurs if $p_1>p_3$. 

\begin{remark}
Theoretically, we can obtain an explicit formula of solutions to \eqref{E:qqq_11}
if $p_1/p_3 \in \{ \frac13, 1, \frac53,  2, \frac73, 3,4, 5, 7, 9, 11 \}$ (see Remark \ref{R:Jacobi_criteria}).
However, the formula is more complicated than the above cases, in general, and hence we do not pursue them here.
\end{remark}

\subsubsection{Case 12} The next case is
$p_2=p_3>0$, $p_4 \neq 0$, and $p_1=p_5=0$.
Although it involves three nonzero parameters, the system itself is considerably simple:
\begin{equation}\label{E:NLS12}
\left\{
\begin{aligned}
	(i \partial_t  + \partial_x^2)u_1={}&
(3p_2 +p_3+2p_4) |u_1|^2 u_1  +\mathcal{V}(u_1,u_2) u_1, \\
	(i \partial_t  + \partial_x^2)u_2={}&
(3p_2+p_3-2p_4) |u_2|^2u_2 + \mathcal{V}(u_1,u_2)u_2.
\end{aligned}
\right.
\end{equation}
It is almost decoupled as in Case 4. The present case is covered by \cite{KS}*{Example 6.4}.
For completeness, we record the result.
The quadratic system \eqref{E:qqq} takes the form
\begin{equation}\label{E:qqq_12}
	\mathcal{D}'=0, \quad \mathcal{R}'=- 4p_3 \mathcal{D} \mathcal{I} -2p_4 \rho \mathcal{I},
	\quad \mathcal{I}' = 4p_3 \mathcal{D} \mathcal{R} + 2p_4 \rho \mathcal{R}.
\end{equation}
\begin{proposition}
The two points $(\mathcal{D},\mathcal{R}, \mathcal{I})=\pm(\rho, 0,0)$ are fixed points of \eqref{E:qqq_12}.
If $2p_3>|p_4|$ then all points in the set $\{ (-\frac{p_4}{2p_3} \rho, \sqrt{1-(\frac{p_4}{2p_3} )^2} \rho \cos \eta,\sqrt{1-(\frac{p_4}{2p_3} )^2} \rho\sin \eta) \ |\ \eta \in \R/2\pi \Z \}$ are fixed points.
Any solution to \eqref{E:qqq_12} is given by
\[
	\mathcal{D}(\tau) = \mathcal{D}(0), 
\]
\[
	\mathcal{R}(\tau) = \mathcal{R}(0) \cos (2(2p_3\mathcal{D}(0) +p_4 \rho )\tau  ) - \mathcal{I}(0) \sin (2(2p_3\mathcal{D}(0) +p_4 \rho )\tau  ), 
\]
and
\[
	\mathcal{I}(\tau) = \mathcal{R}(0) \sin (2(2p_3\mathcal{D}(0) +p_4 \rho )\tau  ) + \mathcal{I}(0) \cos (2(2p_3\mathcal{D}(0) +p_4 \rho )\tau ).
\]
\end{proposition}
Combining the proposition with Theorems \ref{T:KS} and \ref{T:main}, we obtain the asymptotic profile of  solutions to \eqref{E:NLS12}.

\subsubsection{Case 13}
The next case is $p_2=-p_3<0$, $p_5>0$, and $p_1=p_4=0$.
It is, in a sense, paired with Case 12.  The NLS system takes the form
\begin{equation}\label{E:NLS13}
\left\{
\begin{aligned}
	(i \partial_t  + \partial_x^2)u_1={}&
-2 p_3 |u_1|^2 u_1 + p_5(2|u_1|^2 u_2 + u_1^2\ol{u_2})
-2p_3 (2 u_1 |u_2|^2 + \ol{u_1}u_2^2) + p_5 |u_2|^2u_2
	\\&  +\mathcal{V}(u_1,u_2) u_1, \\
	(i \partial_t  + \partial_x^2)u_2={}&
p_5 |u_1|^2 u_1-2p_3(2|u_1|^2 u_2 + u_1^2\ol{u_2})+p_5 (2 u_1 |u_2|^2 + \ol{u_1}u_2^2) -2p_3 |u_2|^2u_2
	\\&  + \mathcal{V}(u_1,u_2)u_2
\end{aligned}
\right.
\end{equation}
and the quadratic system \eqref{E:qqq} takes the form
\begin{equation}\label{E:qqq_13}
	\mathcal{D}'=- 4p_3 \mathcal{R} \mathcal{I} +2p_5 \rho \mathcal{I}, \quad \mathcal{R}'=0,
	\quad \mathcal{I}' = 4p_3 \mathcal{D} \mathcal{R} - 2p_5 \rho \mathcal{D}.
\end{equation}
\begin{proposition}
The two points $(\mathcal{D},\mathcal{R}, \mathcal{I})=\pm(0,\rho,0)$ are fixed points of \eqref{E:qqq_13}.
If $2p_3>p_5$ then all points in the set $\{ (\sqrt{1-(\frac{p_5}{2p_3} )^2} \rho \cos \eta, \frac{p_5}{2p_3} \rho, \sqrt{1-(\frac{p_5}{2p_3} )^2} \rho\sin \eta) \ |\ \eta \in \R/2\pi \Z \}$ are fixed points.
Any solution to \eqref{E:qqq_13} is given by
\[
	\mathcal{D}(\tau) = \mathcal{D}(0) \cos (2(2p_3\mathcal{R}(0) -p_5 \rho )\tau  ) - \mathcal{I}(0) \sin (2(2p_3\mathcal{R}(0) -p_5 \rho )\tau  ), 
\]
\[
	\mathcal{R}(\tau) = \mathcal{R}(0), 
\]
and
\[
	\mathcal{I}(\tau) = \mathcal{D}(0) \sin (2(2p_3\mathcal{R}(0) -p_5 \rho )\tau  ) + \mathcal{I}(0) \cos (2(2p_3\mathcal{R}(0) -p_5 \rho )\tau ).
\]
\end{proposition}
Combining the proposition with Theorems \ref{T:KS} and \ref{T:main}, we obtain the asymptotic profile of  solutions to \eqref{E:NLS13}.

\subsubsection{Case 14}
We turn to the case
$p_1^2 + p_2^2 = p_3^2$ and $p_4=p_5=0$.
We further assume $p_1>0$ and $p_2 \neq0$ otherwise this case is reduced to Cases 5 or 8.
Note that $p_1 < p_3$ and $|p_2|<p_3$ follow by assumption.
Let us introduce $\Theta = \tan^{-1} \frac{p_1}{p_2+p_3} \in (0,\pi/2)$.
The NLS system becomes
\begin{equation}\label{E:NLS14}
\left\{
\begin{aligned}
	(i \partial_t  + \partial_x^2)u_1={}&
(3p_2 +p_3) |u_1|^2 u_1 +p_1(2|u_1|^2 u_2 + u_1^2\ol{u_2})\\
&+(p_2-p_3) (2 u_1 |u_2|^2 + \ol{u_1}u_2^2) - p_1 |u_2|^2u_2
	\\& - 4p_1 \Re (\overline{u_1} u_2) u_1 +\mathcal{V}(u_1,u_2) u_1, \\
	(i \partial_t  + \partial_x^2)u_2={}&
p_1 |u_1|^2 u_1+(p_2-p_3)(2|u_1|^2 u_2 + u_1^2\ol{u_2})\\&-p_1 (2 u_1 |u_2|^2 + \ol{u_1}u_2^2) + (3p_2+p_3) |u_2|^2u_2
	\\& + 4p_1 \Re (\overline{u_1} u_2) u_2 + \mathcal{V}(u_1,u_2)u_2.
\end{aligned}
\right.
\end{equation}
The quadratic system is  
\begin{equation}\label{E:qqq_14}
\left\{
\begin{aligned}
\begin{bmatrix}\mathcal{D}\\\mathcal{R}\end{bmatrix}' ={}&2\mathcal{I}
\begin{bmatrix}
p_1 & p_2-p_3 \\ - p_2 - p_3 & p_1 
\end{bmatrix}
\begin{bmatrix}\mathcal{D}\\\mathcal{R}\end{bmatrix} ,\\
\mathcal{I}' = {}& -2p_1 (\mathcal{D}^2+\mathcal{R}^2) + 4 p_3\mathcal{D} \mathcal{R} .
\end{aligned}
\right.
\end{equation}
We note that if $p_1^2 + p_2^2 = p_3^2$ then the matrix in the first equation of \eqref{E:qqq_14} has zero as its eigenvalue.
As a result,
\[
	X(t):= \tfrac{1}{2 \sin \Theta} \mathcal{D}(t) + \tfrac{1}{2 \cos \Theta} \mathcal{R}(t)
\]
is a conserved quantity. This fact is useful to solve the quadratic system.
\begin{proposition}
Let
\[
	X = \tfrac{1}{2 \sin \Theta} \mathcal{D}(0) + \tfrac{1}{2 \cos \Theta} \mathcal{R}(0) , \quad
	Y = -\tfrac{1}{2 \sin \Theta} \mathcal{D}(0) + \tfrac{1}{2 \cos \Theta} \mathcal{R}(0) 
\]
and $r= \sqrt{ \left\lvert \rho^2-X^2\right\rvert}$.
\begin{enumerate}
\item If $|X| < \rho$ then the solution of \eqref{E:qqq_14} is given as
\begin{align*}
	\mathcal{D}(t) ={}& \sin \Theta (X   -  \tfrac{4r^2Y}{ ((r - \mathcal{I}(0))^2 +Y^2) e^{4p_1r t }  + ((r + \mathcal{I}(0))^2 +Y^2) e^{-4p_1r t} -2\mathcal{I}(0)^2 - 2Y^2 + 2r^2} ),
\end{align*}
\begin{align*}
	\mathcal{R}(t) ={}& \cos \Theta  ( X+ \tfrac{4r^2Y}{ ((r - \mathcal{I}(0))^2 +Y^2) e^{4p_1r t }  + ((r + \mathcal{I}(0))^2 +Y^2) e^{-4p_1r t} -2\mathcal{I}(0)^2 - 2Y^2 + 2r^2} ),
\end{align*}
and
\[
	\mathcal{I}(t) = -r\tfrac{((r - \mathcal{I}(0))^2 +Y^2) e^{4p_1r t }  - ((r + \mathcal{I}(0))^2 +Y^2) e^{-4p_1r t}  }{ ((r - \mathcal{I}(0))^2 +Y^2) e^{4p_1r t }  + ((r + \mathcal{I}(0))^2 +Y^2) e^{-4p_1r t} -2\mathcal{I}(0)^2 - 2Y^2 + 2r^2}.
\]
\item If $|X|=\rho$ then the solution of \eqref{E:qqq_14} is given as
\begin{align*}
	\mathcal{D}(t) ={}& \sin \Theta (X -\tfrac{(\mathcal{I}(0)^2 + Y^2)Y }{
	(2p_1 t (\mathcal{I}(0)^2 + Y^2)- \mathcal{I}(0))^2 + Y^2
	}),
\end{align*}
\begin{align*}
	\mathcal{R}(t) ={}& \cos \Theta (X +\tfrac{(\mathcal{I}(0)^2 + Y^2)Y }{
	(2p_1 t (\mathcal{I}(0)^2 + Y^2)- \mathcal{I}(0))^2 + Y^2} ),
\end{align*}
and
\[
	\mathcal{I}(t) = -\tfrac{(\mathcal{I}(0)^2 + Y^2)(2p_1 t (\mathcal{I}(0)^2 + Y^2)- \mathcal{I}(0)) }{
	(2p_1 t (\mathcal{I}(0)^2 + Y^2)- \mathcal{I}(0))^2 + Y^2
	}.
\]
\item If $|X|>\rho$ then the solution of \eqref{E:qqq_14} is given as
\begin{align*}
	\mathcal{D}(t) ={}& \sin \Theta (X -\tfrac{ 2r^2Y }{
	 (\mathcal{I}(0)^2 + Y^2+r^2)- 2r \mathcal{I}(0) \sin (4p_1 r t) - ( \mathcal{I}(0)^2 + Y^2-r^2) \cos (4p_1 r t)} ),
\end{align*}
\begin{align*}
	\mathcal{R}(t) ={}& \cos \Theta (X + \tfrac{ 2r^2Y }{
	 (\mathcal{I}(0)^2 + Y(0)^2+r^2)- 2r \mathcal{I}(0) \sin (4p_1 r t) - ( \mathcal{I}(0)^2 + Y^2-r^2) \cos (4p_1 r t)} ),
\end{align*}
and
\[
	\mathcal{I}(t) = r\tfrac{
	 2r \mathcal{I}(0) \cos (4p_1 r t) - ( \mathcal{I}(0)^2 + Y^2-r^2) \sin (4p_1 r t)
	  }{ 
	(\mathcal{I}(0)^2 + Y(0)^2+r^2)- 2r \mathcal{I}(0) \sin (4p_1 r t) - ( \mathcal{I}(0)^2 + Y^2-r^2) \cos (4p_1 r t)}.
\]
\end{enumerate}
\end{proposition}

Note that $|(\frac1{2\sin \Theta},\frac1{2\cos \Theta})|\le1$ if and only if $\Theta = \pi/4$, in which case $|X| \le \rho$ holds for any data.
Conversely, if $\Theta \neq \pi/4$ then there exists a data such that $\rho < |X|$.

Combining the proposition with Theorems \ref{T:KS} and \ref{T:main}, we obtain the asymptotic profile of  solutions to \eqref{E:NLS14}.

\subsubsection{Case 15}
The last case is a special combination of all five parameters.
It is an extension of the previous case: $p_1^2 + p_2^2 = p_3^2$ and $\frac{p_4}{p_5} = \frac{p_1}{p_2+p_3}$.
We further assume $p_1>0$ and $p_2\neq0$.
Let us keep the notation $\Theta = \tan^{-1} \frac{p_1}{p_2+p_3} \in (0,\pi/2)$.
One has $p_5 = p_4 \tan \Theta$.
The system is of the form \eqref{E:NLS}--\eqref{E:nondef}. The quadratic system is \eqref{E:qqq}.
Thanks to the special ratio between $p_4$ and $p_5$, the quantity
\[
	\tfrac{1}{2 \sin \Theta} \mathcal{D}(t) + \tfrac{1}{2 \cos \Theta} \mathcal{R}(t) 
\]
again becomes a conserved quantity.
%
We modify the definition of $X$ by adding a constant to make the description of the case division slightly simple.

\begin{proposition}
Let
\[
	X = \tfrac{1}{2 \sin \Theta} \mathcal{D}(0) + \tfrac{1}{2 \cos \Theta} \mathcal{R}(0) + 
	\tfrac{p_2p_4}{2p_1p_3 \cos \Theta} \rho
\]
and
\[
	Y = -\tfrac{1}{2 \sin \Theta} \mathcal{D}(0) + \tfrac{1}{2 \cos \Theta} \mathcal{R}(0) - 
	\tfrac{p_4}{2p_1 \cos \Theta} \rho.
\]
Let
\[
	r= \sqrt{ \left\lvert(1-\tfrac{p_4^2}{4p_3^2 \cos^2 \Theta})\rho^2-X^2\right\rvert}.
\]

\begin{enumerate}
\item If  $X^2 < (1-\tfrac{p_4^2}{4p_3^2 \cos^2 \Theta})\rho^2$ then the solution of \eqref{E:qqq} is given as
\begin{align*}
	\mathcal{D}(t) ={}& \sin \Theta (X -\tfrac{p_2p_4}{2p_1p_3\cos \Theta}\rho) \\&- \sin \Theta (\tfrac{4r^2Y}{ ((r - \mathcal{I}(0))^2 +Y^2) e^{4p_1r t }  + ((r + \mathcal{I}(0))^2 +Y^2) e^{-4p_1r t} -2\mathcal{I}(0)^2 - 2Y^2 + 2r^2}) - \tfrac{p_5}{2p_1}\rho,
\end{align*}
\begin{align*}
	\mathcal{R}(t) ={}& \cos \Theta (X -\tfrac{p_2p_4}{2p_1p_3 \cos \Theta}\rho) \\&+ \cos \Theta (\tfrac{4r^2Y}{ ((r - \mathcal{I}(0))^2 +Y^2) e^{4p_1r t }  + ((r + \mathcal{I}(0))^2 +Y^2) e^{-4p_1r t} -2\mathcal{I}(0)^2 - 2Y^2 + 2r^2} )+ \tfrac{p_4}{2p_1}\rho,
\end{align*}
and
\[
	\mathcal{I}(t) = -r\tfrac{((r - \mathcal{I}(0))^2 +Y^2) e^{4p_1r t }  - ((r + \mathcal{I}(0))^2 +Y^2) e^{-4p_1r t}  }{ ((r - \mathcal{I}(0))^2 +Y^2) e^{4p_1r t }  + ((r + \mathcal{I}(0))^2 +Y^2) e^{-4p_1r t} -2\mathcal{I}(0)^2 - 2Y^2 + 2r^2}.
\]
\item If  $X^2 = (1-\tfrac{p_4^2}{4p_3^2 \cos^2 \Theta})\rho^2$ then the solution of \eqref{E:qqq} is given as
\begin{align*}
	\mathcal{D}(t) ={}& \sin \Theta (X -\tfrac{p_2p_4}{2p_1p_3 \cos \Theta}\rho) - \sin \Theta (\tfrac{\mathcal{I}(0)^2 + Y^2 }{
	(2p_1 t (\mathcal{I}(0)^2 + Y^2)- \mathcal{I}(0))^2 + Y^2
	}) - \tfrac{p_5}{2p_1}\rho,
\end{align*}
\begin{align*}
	\mathcal{R}(t) ={}& \cos \Theta (X -\tfrac{p_2p_4}{2p_1p_3\cos \Theta}\rho) + \cos \Theta (\tfrac{\mathcal{I}(0)^2 + Y^2 }{
	(2p_1 t (\mathcal{I}(0)^2 + Y^2)- \mathcal{I}(0))^2 + Y^2})
	 + \tfrac{p_4}{2p_1}\rho,
\end{align*}
and
\[
	\mathcal{I}(t) = -\tfrac{(\mathcal{I}(0)^2 + Y^2)(2p_1 t (\mathcal{I}(0)^2 + Y^2)- \mathcal{I}(0)) }{
	(2p_1 t (\mathcal{I}(0)^2 + Y^2)- \mathcal{I}(0))^2 + Y^2
	}.
\]
\item If  $X^2 > (1-\tfrac{p_4^2}{4p_3^2 \cos^2 \Theta})\rho^2$ then the solution of \eqref{E:qqq} is given as
\begin{align*}
	\mathcal{D}(t) ={}& \sin \Theta (X -\tfrac{p_2p_4}{2p_1p_3\cos \Theta }\rho) \\&- \sin \Theta (\tfrac{ 2r^2 }{
	 (\mathcal{I}(0)^2 + Y^2+r^2)- 2r \mathcal{I}(0) \sin (4p_1 r t) - ( \mathcal{I}(0)^2 + Y^2-r^2) \cos (4p_1 r t)}) - \tfrac{p_5}{2p_1}\rho,
\end{align*}
\begin{align*}
	\mathcal{R}(t) ={}& \cos \Theta (X -\tfrac{p_2p_4}{2p_1p_3\cos \Theta }\rho) \\&+ \cos \Theta (\tfrac{ 2r^2 }{
	 (\mathcal{I}(0)^2 + Y^2+r^2)- 2r \mathcal{I}(0) \sin (4p_1 r t) - ( \mathcal{I}(0)^2 + Y^2-r^2) \cos (4p_1 r t)}) + \tfrac{p_4}{2p_1}\rho,
\end{align*}
and
\[
	\mathcal{I}(t) = r\tfrac{ 2r \mathcal{I}(0) \cos (4p_1 r t) - ( \mathcal{I}(0)^2 + Y^2-r^2) \sin (4p_1 r t)  }{  (\mathcal{I}(0)^2 + Y^2+r^2)- 2r \mathcal{I}(0) \sin (4p_1 r t) - ( \mathcal{I}(0)^2 + Y^2-r^2) \cos (4p_1 r t)}.
\]

\end{enumerate}
\end{proposition}

Combining the proposition with Theorems \ref{T:KS} and \ref{T:main}, we obtain the asymptotic profile of  solutions to \eqref{E:NLS} under the present specific combination of parameters

\section{The derivation of \eqref{E:NLS}}\label{S:classification}\label{S:deriveNLS}


In this section, we will see that a cubic NLS system that possesses a coercive mass-like conserved quantity
is transformed into \eqref{E:NLS}.
We remark that a suitable restriction of the parameters in \eqref{E:NLS} is in need to assure the uniqueness of the transformed system.
See \cites{MSU2,M} for details.
Here, we do not pursue the uniqueness issue in order to simplify
the argument.

As a starting point,
let us consider the following NLS system of a general form:
\begin{equation}\label{E:generalNLS}
	\left\{
	\begin{aligned}
	& (i  \partial_t + \partial_x^2 )u_1 = \l_1 |u_1|^2 u_1 + \l_2 |u_1|^2 u_2 + \l_3 u_1^2 \ol{u_2} +
	\l_4 |u_2|^2 u_1 + \l_5 u_2^2 \ol{u_1} + \l_6 |u_2|^2 u_2,  \\
	& (i  \partial_t + \partial_x^2 )u_2 = \l_7 |u_1|^2 u_1 + \l_8 |u_1|^2 u_2 + \l_9 u_1^2 \ol{u_2} +
	\l_{10} |u_2|^2 u_1 + \l_{11} u_2^2 \ol{u_1} + \l_{12} |u_2|^2 u_2,
	\end{aligned}
	\right.
\end{equation}
where $(\lambda_1, \dots, \lambda_{12}) \in \R^{12}$.

\subsection{Matrix-Vector representation of the system}\label{SS:CV}
The system \eqref{E:generalNLS} can be identified with a pair consisting of a matrix $C \in M_3 (\R) \simeq \R^9$ and a vector $V \in \R^3$ as follows:
Given $(\lambda_1, \dots, \lambda_{12}) \in \R^{12}$, we define
\begin{equation}\label{E:matrixC}
C:=
\begin{bmatrix}
\lambda_2-\lambda_3 & -\lambda_1+\lambda_8-\lambda_9 & -\lambda_7 \\
\lambda_5 & -\lambda_3+\lambda_{11} & -\lambda_9 \\
\lambda_6 & -\lambda_4+\lambda_5+\lambda_{12} & -\lambda_{10}+\lambda_{11}
\end{bmatrix}  
\end{equation}
and
\begin{equation}\label{E:vectorV}
	V:=
	\begin{bmatrix} \lambda_8-2\lambda_9 \\ \tfrac12 (-\lambda_2+2\lambda_3-\lambda_{10}+ 2\lambda_{11}) \\  \lambda_4-2\lambda_5 \end{bmatrix}. 
\end{equation}
Conversely, for a given pair $(C=(c_{ij})_{1\le i,j \le 3}, (v_k)_{1\le k \le 3} )\in M_3(\R) \times \R^3$, one defines a system
by
\begin{equation}\label{E:CVNLS}
\left\{
\begin{aligned}
	(i  \partial_t + \partial_x^2 )u_1={}&
-(c_{12}+c_{23}) |u_1|^2 u_1 +c_{11}(2|u_1|^2 u_2 + u_1^2\ol{u_2})+c_{21} (2 u_1 |u_2|^2 + \ol{u_1}u_2^2) \\&+ c_{31} |u_2|^2u_2
	 - (\tr{C}) \Re (\overline{u_1} u_2) u_1 +\mathcal{V}(u_1,u_2) u_1, \\
	(i  \partial_t + \partial_x^2 )u_2={}&
-c_{13} |u_1|^2 u_1-c_{23}(2|u_1|^2 u_2 + u_1^2\ol{u_2})-c_{33} (2 u_1 |u_2|^2 + \ol{u_1}u_2^2) \\&+ (c_{21}+c_{32}) |u_2|^2u_2
	 + (\tr{C}) \Re (\overline{u_1} u_2) u_2 + \mathcal{V}(u_1,u_2)u_2,
\end{aligned}
\right.
\end{equation}
where $\mathcal{V}(u_1,u_2)=q_1 |u_1|^2 + 2q_2 \Re (\overline{u_1}u_2)  + q_3 |u_2|^2$ is a real-valued quadratic potential.

This matrix-vector representation is introduced in \cite{MSU2} (see also \cite{MSU1,M}).
The validity of a mass-like conservation law is well described by the representation.

\begin{proposition}[\cite{MSU2}*{Proposition A.8}]\label{P:masscond}
Let $a,b,c \in \R$. The quantity
\begin{equation}\label{E:genmass}
	\int (a |u_1|^2 + 2 b\Re \overline{u_1} u_2 + c |u_2|^2) dx
\end{equation}
becomes a conserved quantity of \eqref{E:generalNLS}, i.e., it holds that
\[
	\Im \(\begin{bmatrix}\overline{u_1} & \overline{u_2} \end{bmatrix}
	\begin{bmatrix} a & b \\ b & c \end{bmatrix}
	\begin{bmatrix} F_1(u_1,u_2) \\ F_2(u_1,u_2) \end{bmatrix}
	\)\equiv 0
\]
if and only if $\ltrans{(a,b,c)} \in \ker C$.
\end{proposition}

\subsection{First reduction}

Suppose that \eqref{E:generalNLS} possesses a conserved quantity of the form \eqref{E:genmass}.
Then, by a simple quadratic completion, one may find $M \in GL_2(\R)$ such that
\[
	\int (a |u_1|^2 + 2 b\Re \overline{u_1} u_2 + c |u_2|^2) dx
	= \int ( |v_1|^2 + \sigma |v_2|^2) dx
\]
holds for
\[
	\begin{bmatrix} v_1 \\ v_2 \end{bmatrix} = M \begin{bmatrix} u_1 \\ u_2 \end{bmatrix},
\]
where $\sigma \in \{ 1,0,-1\}$ is determined by the sign of the quadratic form; $\sigma=1$ if $b^2 -ac < 0$,
$\sigma=0$ if $b^2-ac=0$, and $\sigma=-1$ if $b^2-ac>0$.
Now, we suppose that the original system has a coercive mass, i.e., $\sigma=1$.
Then, by means of Proposition \ref{P:masscond}, the matrix part of the system for $(v_1,v_2)$ is of the following form
\[
	\begin{bmatrix}
	c_{11} & c_{12} & - c_{11} \\
	c_{21} & c_{22} & - c_{21} \\
	c_{31} & c_{32} & - c_{31} 
	\end{bmatrix}.
\]

\subsection{Second reduction}

Let us introduce the following parametrix for the matrix part of the system obtained in the first reduction.
\begin{equation}\label{E:stdMatform}
	\begin{bmatrix}
	p_1+\tilde{p}_3+p_5 & -2p_2-2p_3-2p_4 & - p_1-\tilde{p}_3 - p_5 \\
	p_{2}-p_3 & 2p_1- 2\tilde{p}_3& - p_{2}+p_3 \\
	-p_1-\tilde{p}_3+p_5 & 2p_{2}+2p_3-2p_4 & p_1 + \tilde{p}_3 - p_5
	\end{bmatrix}.
\end{equation}
We remark that the 6-tuple $(p_1,p_2,p_3,\tilde{p}_3,p_4,p_5)$ is uniquely determined.
More specifically, we can define
\[
	p_1 = \tfrac14(c_{11}+c_{22} - c_{31}), \quad
	p_4 = - \tfrac14 (c_{12} + c_{32}), \quad
	p_5 = \tfrac12 (c_{11}+c_{31})
\]
and then 
\[
	p_2   = \tfrac18(c_{32}-c_{12}) + \tfrac12 c_{21}, \quad
	p_3   = \tfrac18(c_{32}-c_{12}) - \tfrac12 c_{21}, \quad
	\tilde{p}_3 = c_{11} - p_1 - p_5 .
\]
Note that \eqref{E:NLS} is obtained if we have $\tilde{p}_3=0$.

If $p_1<0$ then we apply a change of variable $(v_1,v_2) \mapsto (v_1,-v_2)$. This changes the sign of $p_1$.
Now, let us introduce a change of variable
\[
	\begin{bmatrix} w_1 \\ w_2 \end{bmatrix} = \begin{bmatrix} \cos \theta  & \sin \theta \\ - \sin \theta & \cos \theta \end{bmatrix}
	\begin{bmatrix} u_1 \\ u_2 \end{bmatrix}.
\]
Then, the new 6-tuple is given as
\[
	\begin{bmatrix} 1 & 0 & 0 & 0 & 0 & 0 \\ 0 & 1 & 0 & 0 & 0 & 0 \\
	0 & 0 &  \cos 4 \theta & -  \sin 4 \theta& 0 & 0\\ 	0 & 0 &  \sin 4\theta& \cos 4 \theta & 0 & 0\\
		0 & 0 & 0& 0&  \cos 2 \theta & -  \sin 2 \theta  \\
	0 & 0 & 0& 0&  \sin 2\theta & \cos 2 \theta 
	 \end{bmatrix} \begin{bmatrix} p_1\\ p_2 \\ p_3 \\ \tilde{p}_3\\ p_4\\ p_5 \end{bmatrix}
\]
(See \cite{M}*{Lemma 3.3}). Hence, by taking a suitable $\theta$, one may have $\tilde{p}_3=0$ and $p_3\ge0$ and $p_5\ge0$.
Similarly, if $p_3=\tilde{p}_3=0$ then one can choose $\theta$ so that $p_5=0$.

Thus, we obtain \eqref{E:NLS} by plugging \eqref{E:stdMatform} (with $\tilde{p}_3=0$) to \eqref{E:CVNLS}.

\section{Proof of Theorem \ref{T:main}}\label{S:pfmain}

In this section, we prove Theorem \ref{T:main}.
This theorem applies to a wider class of systems.
Hence, in this section,
let us consider the following form
\begin{equation}\label{E:originODE}
	\left\{
	\begin{aligned}
	& i  u_1' = \l_1 |u_1|^2 u_1 + \l_2 |u_1|^2 u_2 + \l_3 u_1^2 \ol{u_2} +
	\l_4 |u_2|^2 u_1 + \l_5 u_2^2 \ol{u_1} + \l_6 |u_2|^2 u_2,  \\
	& i u_2' = \l_7 |u_1|^2 u_1 + \l_8 |u_1|^2 u_2 + \l_9 u_1^2 \ol{u_2} +
	\l_{10} |u_2|^2 u_1 + \l_{11} u_2^2 \ol{u_1} + \l_{12} |u_2|^2 u_2,
	\end{aligned}
	\right.
\end{equation}
where $\lambda_j \in \R$.
We apply the matrix-vector representation discussed in Section \ref{SS:CV}.
By defining the matrix $C=(c_{ij})_{1\le i,j \le 3} \in M_3(\R)$ and the vector $V=(q_k)_{1\le k \le 3} \in \R^3$ as in \eqref{E:matrixC} and \eqref{E:vectorV}, respectively,
one obtains
\begin{equation}\label{E:CVODE}
\left\{
\begin{aligned}
	iu_1'={}&
-(c_{12}+c_{23}) |u_1|^2 u_1 +c_{11}(2|u_1|^2 u_2 + u_1^2\ol{u_2})+c_{21} (2 u_1 |u_2|^2 + \ol{u_1}u_2^2) + c_{31} |u_2|^2u_2
	\\& - (\tr{C}) \Re (\overline{u_1} u_2) u_1 +\mathcal{V}(u_1,u_2) u_1, \\
	iu_2'={}&
-c_{13} |u_1|^2 u_1-c_{23}(2|u_1|^2 u_2 + u_1^2\ol{u_2})-c_{33} (2 u_1 |u_2|^2 + \ol{u_1}u_2^2) + (c_{21}+c_{32}) |u_2|^2u_2
	\\& + (\tr{C}) \Re (\overline{u_1} u_2) u_2 + \mathcal{V}(u_1,u_2)u_2,
\end{aligned}
\right.
\end{equation}
where $\mathcal{V}(u_1,u_2)=q_1 |u_1|^2 + 2q_2 \Re (\overline{u_1}u_2)  + q_3 |u_2|^2$.

\subsection{Removal of the vector part by gauge transform}

One good point of the above representation with a pair $(C,V)$ is that the essence of the system lies only in the matrix part $C$.
This is because one may let $V=0$ by a gauge transform. Indeed, let us introduce a pair of new unknowns $(\alpha_1,\alpha_2)$ by
\[
	\alpha_j = u_j \exp({\textstyle  \int_0^t \mathcal{V}(u_1(\tau),u_2(\tau))}d\tau).
\]
Then, noticing that the nonlinearities are gauge-invariant, one sees that the ODE system \eqref{E:CVODE} 
turns into
\begin{equation}\label{E:gODE}
\left\{
\begin{aligned}
	i\a_1'={}&
-(c_{12}+c_{23}) |\a_1|^2 \a_1 +c_{11}(2|\a_1|^2 \a_2 + \a_1^2\ol{\a_2})+c_{21} (2 \a_1 |\a_2|^2 + \ol{\a_1}\a_2^2) + c_{31} |\a_2|^2\a_2
	\\& - (\tr{C}) \Re (\overline{\a_1} \a_2) \a_1 \\
	i\a_2'={}&
-c_{13} |\a_1|^2 \a_1-c_{23}(2|\a_1|^2 \a_2 + \a_1^2\ol{\a_2})-c_{33} (2 \a_1 |\a_2|^2 + \ol{\a_1}\a_2^2) + (c_{21}+c_{32}) |\a_2|^2\a_2
	\\& + (\tr{C}) \Re (\overline{\a_1} \a_2) \a_2
\end{aligned}
\right.
\end{equation}
with the same matrix $C=(c_{ij})_{1\le i,j \le 3}$. 

Thus, the problem boils down to solving the system \eqref{E:gODE}.
Once we obtain a representation of $(\alpha_1,\alpha_2)$, we obtain that of $(u_1,u_2)$ by 
\begin{equation}\label{E:ungauge}
	u_j = \alpha_j \exp(-{\textstyle  \int_0^t \mathcal{V}(\a_1(\tau),\a_2(\tau))}d\tau).
\end{equation}
Let us remark that $\mathcal{V}(\a_1,\a_2)=\mathcal{V}(u_1,u_2)$ holds since $\mathcal{V}$ is invariant under the gauge transform, which implies that \eqref{E:ungauge} is the inverse transform.

\subsection{Reduction to quadratic quantities}

Now, our purpose is to find an (almost) explicit representation of a solution to \eqref{E:gODE} in terms of the corresponding quadratic quantities.
The following, which is the generalization of Theorem \ref{T:main}, is the crucial step of the argument.
\begin{theorem}\label{T:maingen}
Let $\rho$, $\mathcal{D}$, $\mathcal{R}$, and $\mathcal{I}$ be quadratic quantities which correspond to a nontrivial solution $(\alpha_1,\alpha_2) \in C^\infty (I_{\max};\C^2)$ to \eqref{E:gODE} as in \eqref{D:rhoDRI},
where $I_{\max}$ is the lifespan of the solution.
Suppose that $0 \in I_{\max}$.
If $\alpha_1(0)\neq0$ then
\[
	\alpha_1 (t) = (-1)^{k_1(t)} \sqrt{\tfrac{\rho + \mathcal{D}(t)}{2}} \tfrac{\alpha_1(0)}{|\alpha_1(0)|} e^{ i \int_0^t N_1(\sigma) d\sigma }
\]
and
\[
	\alpha_2 (t) = (-1)^{k_1(t)} \tfrac{\mathcal{R}(t)+i\mathcal{I}(t)}{\sqrt{2(\rho + \mathcal{D}(t))}}
	\tfrac{\alpha_1(0)}{|\alpha_1(0)|}  e^{ i \int_0^t N_1(\sigma) d\sigma }
\]
hold for all $t\in I_{\max}$,
where
\[
	k_1 (t) = \begin{cases}
	\# (\{ s \in I_{\max} \ |\ \rho + \mathcal{D}(s) =0 \} \cap [0,t]) & t>0, \\
	\# (\{ s \in I_{\max} \ |\ \rho + \mathcal{D}(s) =0 \} \cap [t,0]) & t<0 
	\end{cases}
\]
is finite for all $t\in I_{\max}$ and
\[
	N_1:=  (c_{12}+c_{23}) \tfrac{\rho + \mathcal{D}}{2} - \tfrac32 c_{11} \mathcal{R} - c_{21} (\rho-\mathcal{D}+\tfrac{\mathcal{R}^2-\mathcal{I}^2}{2(\rho + \mathcal{D})})
	-c_{31} \tfrac{(\rho- \mathcal{D}) \mathcal{R}}{2(\rho+ \mathcal{D})} + \tfrac{\tr C}2 \mathcal{R}.
\]
If $\alpha_2(0)\neq0$ then
\[
	\alpha_1 (t) =   (-1)^{k_2(t)}
	\tfrac{\mathcal{R}(t)-i\mathcal{I}(t)}{\sqrt{2(\rho - \mathcal{D}(t))}}
	\tfrac{\alpha_2(0)}{|\alpha_2(0)|} e^{ i \int_0^t N_2(\sigma) d\sigma }
\]
and
\[
	\alpha_2 (t) =   (-1)^{k_2(t)}
	\sqrt{\tfrac{\rho - \mathcal{D}(t)}{2}}
	\tfrac{\alpha_2(0)}{|\alpha_2(0)|} e^{ i \int_0^t N_2(\sigma) d\sigma }
\]
hold for all $t\in I_{\max}$,
where
\[
	k_2 (t) = \begin{cases}
	\# (\{ s \in I_{\max} \ |\ \rho - \mathcal{D}(s) =0 \} \cap [0,t]) & t>0, \\
	\# (\{ s \in I_{\max} \ |\ \rho - \mathcal{D}(s) =0 \} \cap [t,0]) & t<0 
	\end{cases}
\]
is finite for all $t\in I_{\max}$ and 
\[
	N_2:= - (c_{21}+c_{32}) \tfrac{\rho - \mathcal{D}}{2} + \tfrac32 c_{33} \mathcal{R} + c_{23} (\rho+\mathcal{D}+\tfrac{\mathcal{R}^2-\mathcal{I}^2}{2(\rho - \mathcal{D})})
	+c_{13} \tfrac{(\rho+ \mathcal{D}) \mathcal{R}}{2(\rho- \mathcal{D})} - \tfrac{\tr C}2 \mathcal{R}.
\]
\end{theorem}
\begin{proof}
By the uniqueness property, one sees that $(\alpha_1(t),\alpha_2(t))\neq (0,0)$ for all $t\in I_{\max}$.
%
We only consider the case $\alpha_1(0)\neq0$.
The other case is handled similarly.

\noindent{\bf Step 1}.
Pick an open interval $I \subset I_{\max}$ such that $0 \in I$ and
 $\alpha_1 \neq0$ on $I$.
We introduce new variable $\theta=\theta(t)$ on $I$ by
\[
	\alpha_1 = |\alpha_1| e^{i \theta}= \sqrt{\tfrac{\rho + \mathcal{D}}{2}} e^{i \theta}.
\]
We remark that $\alpha_2$ is given by
\[
	\alpha_2 = \tfrac12 \tfrac{\alpha_1}{|\alpha_1|^2} (2\ol{\alpha_1}\alpha_2) = \tfrac1{2|\alpha_1|}  e^{i \theta} (\mathcal{R}+i\mathcal{I})
	= \tfrac{\mathcal{R}+i\mathcal{I}}{\sqrt{2(\rho + \mathcal{D})}}  e^{i \theta} 
\]
on $I$.
Hence, it suffices to find an explicit formula of $\theta$.
To this end,
let us derive an ODE for $\theta$.  One has
\[
	 |\alpha_1|^2 \theta' = \Im \ol{\alpha_1} \alpha_1'=-\Re \ol{\alpha_1} (i\alpha_1')  .
\]
Plugging the first equation of \eqref{E:gODE} to the right hand side,
one  obtains
\begin{equation}\label{E:integration-thetaeq}
	\theta' =N_1:=  (c_{12}+c_{23}) |\alpha_1|^2 - \tfrac32 c_{11} \mathcal{R} - c_{21} (2|\alpha_2|^2+\tfrac{\mathcal{R}^2-\mathcal{I}^2}{4|\alpha_1|^2})
	-c_{31} \tfrac{|\alpha_2|^2 \mathcal{R}}{2|\alpha_1|^2} + \tfrac{\tr C}2 \mathcal{R}.
\end{equation}
Note that $|\alpha_1|^2 = (\rho + \mathcal{D})/2$ and $|\alpha_2|^2 = (\rho-\mathcal{D})/2$ are given functions and hence $N_1$ is written in terms of known functions.
Recall that $\alpha_1\neq 0$ on $I$.
Hence, by integration, one has
\[
	\theta(s) = \theta(0) + \int_0^s N_1(\s)d\s
\]
on $I$. Note that $\theta(0)$ is given in terms of $(\alpha_1(0),\alpha_2(0))$.
This shows that we have an explicit representation of $\theta$ as desired.

\noindent{\bf Step 2}.
If $c_{31}=0$ then $\alpha_1 (t_0)=0$ at some $t_0 \in I_{\max}$ implies $\alpha_1(t)=0$ on $I_{\max}$.
This is due to the uniqueness of the solution to \eqref{E:gODE} and the fact that the pair $(0, \alpha_2(t_0) e^{-i(c_{21}+c_{32})|\alpha_2(t_0)|^2(t-t_0)})$ becomes a solution in this case. Hence, $\alpha_1(0)\neq0$ implies that $\alpha_1 \neq 0$ on $I_{\max}$
and hence the formula obtained in the previous step is valid on whole $I_{\max}$.

Let us consider the case $c_{31}\neq0$. We shall show that, for any solution to \eqref{E:gODE} satisfying $\alpha_1(0)\neq0$, the zero points $\{t \in I_{\max} \ |\ \alpha_1(t)=0\}$ are all isolated removable singular points in the above formula. 
We only consider positive zero points.
Let $t_1 \in (0, \sup I_{\max})$ be the smallest positive zero point.  Define $I_0 :=[0,t_1)$. 
Note that $\alpha_2(t_1) \neq0$ holds since otherwise the solution $(\alpha_1,\alpha_2)$ becomes a trivial one.
We have
\begin{equation}\label{E:mainpf1-1}
	\alpha_1(t) = \sqrt{\tfrac{\rho + \mathcal{D}(t)}{2}} e^{i \theta(0) + i \int_0^t N_1(\sigma) d\sigma}
\end{equation}
for $t\in I_0$.
By the first equation of \eqref{E:gODE}, one has $\alpha_1' (t_1) = -i c_{31} |\alpha_2(t_1)|^2 \alpha_2(t_1)\neq0$.
Hence, one sees that 
\begin{equation}\label{E:mainpf1-2}
\alpha_1(t) = -i c_{31} |\alpha_2(t_1)|^2 \alpha_2(t_1) (t-t_1) + O(|t-t_1|^2)
\end{equation}
 around $t=t_1$.
This shows that $t_1$ is an isolated zero point of $\alpha_1$. It also
gives us
\begin{equation}\label{E:mainpf1-3}
	\lim_{t\uparrow t_1} \tfrac{\alpha_1(t)}{|\alpha_1(t)|} = - i \tfrac{\alpha_2(t_1)}{|\alpha_2(t_1)|} = - \lim_{t\downarrow t_1} \tfrac{\alpha_1(t)}{|\alpha_1(t)|}.
\end{equation}
Further, combining \eqref{E:mainpf1-2} and $\alpha_2 (t) = \alpha_2(t_1) + O(|t-t_1|)$ around $t=t_1$, one obtains
\[
	\rho + \mathcal{D}(t) = 2|\alpha_1(t)|^2=  2c_{31}^2 |\alpha_2(t_1)|^6 (t-t_1)^2 + O(|t-t_1|^3),
\]
\[
	\mathcal{R}(t) = 2 \Re \overline{\alpha_1(t)} \alpha_2(t) =  R_0  (t-t_1)^2 + O(|t-t_1|^3),
\]
and
\[
	\mathcal{I}(t) = 2 \Im \overline{\alpha_1(t)} \alpha_2(t) = c_{31} |\alpha_2(t_1)|^4(t-t_1) + O(|t-t_1|^2)
\]
around $t=t_1$,
where $R_0:= 2 \Re (\overline{\alpha_1''(t_1)} \alpha_2(t_1) + \overline{\alpha_1'(t_1)} \alpha_2'(t_1)) $.
We remark that $R_0$ depend only on the parameters of the ODE system and $\alpha_2(t_1)$. 
Recalling that $\alpha_2(t_1)\neq0$,
these asymptotics show that $N_1(t)$ given in \eqref{E:integration-thetaeq} is continuous at $t=t_1$.
Hence, the integral
$
	 \int_0^t N_1(\sigma) d\sigma
$
makes sense beyond $t_1$.
Further, combining \eqref{E:mainpf1-1} and \eqref{E:mainpf1-3}, one obtains
\begin{equation}\label{E:mainpf1-4}
	e^{i\theta(0)+ \int_0^{t_1} N_1 (\sigma) d\sigma } = - i \tfrac{\alpha_2(t_1)}{|\alpha_2(t_1)|}.
\end{equation}
Let $I_1$ be an interval of the form $ (t_1, t_2)$ such that $\alpha_1(t) \neq0$ holds on $I_1$.
Then, arguing as in Step 1, one has
\[
	\alpha_1(t) = \sqrt{\tfrac{\rho + \mathcal{D}(t)}{2}} e^{i \arg (\alpha_1(t_1+\varepsilon)) + i \int_{t_1+ \varepsilon}^t N_1(\sigma) d\sigma}
\]
for any $0<\varepsilon<t_2-t_1$ and $t \in I_1$. By letting $\varepsilon \downarrow 0$ for each fixed $t$, one deduces from the second equality of the \eqref{E:mainpf1-3} and \eqref{E:mainpf1-4} that
\[
	\alpha_1(t) = \sqrt{\tfrac{\rho + \mathcal{D}(t)}{2}} 
	(i \tfrac{\alpha_2(t_1)}{|\alpha_2(t_1)|})
	e^{ i \int_{t_1}^t N_1(\sigma) d\sigma}
	= (-1)^1 \sqrt{\tfrac{\rho + \mathcal{D}(t)}{2}} e^{i \theta(0) + i \int_0^t N_1(\sigma) d\sigma}.
\]
Repeating this argument, one sees that the sign factor $(-1)^{k_1(t)}$ appears.
Note that the finiteness
of $k_1(t)$ follows from the fact that all zero point of $\alpha_1$ is isolated.
Thus, we obtain the desired expression of the solution on $I_{\max}$.
\end{proof}

\section{Integration of quadratic system for quadratic quantities}\label{S:qqq}
In this section, we consider the integration of \eqref{E:qqq}.
In light of Theorem \ref{T:main}, the matter is reduced to find an explicit formulas for $\rho$, $\mathcal{D}$,
$\mathcal{R}$, and $\mathcal{I}$ given in \eqref{D:rhoDRI}.
Recall that $\rho$ is conserved since \eqref{E:nullcond} holds true.
Hence, we mainly consider the other three quantities $\mathcal{D}$,
$\mathcal{R}$, and $\mathcal{I}$.

\subsection{Quadratic system for quadratic quantities}
Let us first confirm that the triplet $(\mathcal{D},\mathcal{R},\mathcal{I})$ solves \eqref{E:qqq}.

\begin{proposition}\label{P:qqq}
If $(\alpha_1,\alpha_2) $ be a solution to the system of \eqref{E:ODE} then
the triplet of the quadratic quantities
$(\mathcal{D},\mathcal{R},\mathcal{I})$ given by \eqref{D:rhoDRI} solves \eqref{E:qqq}.
\end{proposition}
This is an immediate consequence of the following lemma.
\begin{lemma}[\cite{MSU2}]\label{L:quadODE}
Let $(\alpha_1,\alpha_2) $ be a solution to the system of \eqref{E:CVODE}.
For any $(a,b,c) \in \R^3$, one has
\[
	\frac{d}{dt} \ltrans{\begin{bmatrix} |\alpha_1|^2 \\ \mathcal{R} \\ |\alpha_2|^2  \end{bmatrix}}
	\begin{bmatrix} a \\ b \\ c \end{bmatrix}
	= \mathcal{I} \ltrans{ \begin{bmatrix} |\alpha_1|^2 \\ \mathcal{R} \\ |\alpha_2|^2  \end{bmatrix} }
	C \begin{bmatrix} a \\ b \\ c \end{bmatrix}.
\]
Further,
\[
\begin{aligned}
\mathcal{I}' = \tfrac12\ltrans{ \begin{bmatrix} |\alpha_1|^2 \\ \mathcal{R} \\ |\alpha_2|^2  \end{bmatrix} }
B
\begin{bmatrix} |\alpha_1|^2 \\ \mathcal{R} \\ |\alpha_2|^2  \end{bmatrix} ,
\end{aligned}
\]
where
\begin{equation}\label{eq:B_Cform}
	B =B(C):= \begin{bmatrix}
	4 c_{13}
	&
	-c_{12}+2c_{23} 
	&
	2c_{11}+2c_{33}  
	\\
	-c_{12}+2c_{23} 
	& 
	-2 c_{22} 
	&
	-c_{32} + 2c_{21} 
	\\
	2c_{11}+2c_{33} 
	& 
	-c_{32} + 2c_{21} 
	&
	4 c_{31} 
	\end{bmatrix}.
\end{equation}
\end{lemma}

\subsection{A standard strategy for the integration}
We introduce one roadmap to integrate \eqref{E:qqq}. 

Let us introduce a new variable $s=s(t)$ by $s:= \int_0^t 2 \mathcal{I}(\tau)d\tau$.
Then, the first line of \eqref{E:qqq} is simplified as
\begin{equation}\label{E:qqq2}
\frac{d}{ds}\begin{bmatrix}\mathcal{D}\\\mathcal{R}\end{bmatrix} = \Omega
\begin{bmatrix}\mathcal{D}\\\mathcal{R}\end{bmatrix} +  \rho  \begin{bmatrix} p_5 \\ - p_4 \end{bmatrix},
\end{equation}
where
\[
	\Omega := 
\begin{bmatrix}
p_1 & p_2-p_3 \\ - p_2 - p_3 & p_1 
\end{bmatrix}.
\]
Thus, by solving the above ODE, one obtain the following:
\begin{equation}\label{E:DRform}
		\begin{bmatrix}\mathcal{D}(t)\\\mathcal{R}(t)\end{bmatrix}
		= e^{s \Omega } \begin{bmatrix}\mathcal{D}(0)\\\mathcal{R}(0)\end{bmatrix}
		+  \rho  \int_0^s e^{(s-\sigma)\Omega} 
		 \begin{bmatrix} p_5 \\ - p_4 \end{bmatrix} d\sigma,
\end{equation}
Note that
\[
	e^{s \Omega } = 
e^{p_1s}\begin{bmatrix}
	\cosh (s\sqrt{p_3^2-p_2^2}) & \frac{p_2-p_3}{\sqrt{p_3^2-p_2^2}}\sinh (s\sqrt{p_3^2-p_2^2}) \\
	\frac{-p_2-p_3}{\sqrt{p_3^2-p_2^2}}\sinh (s\sqrt{p_3^2-p_2^2}) & \cosh (s\sqrt{p_3^2-p_2^2})
	\end{bmatrix}
\]
if $p_3^2-p_2^2>0$,
\[
e^{s \Omega } = 
e^{p_1s}\begin{bmatrix}
	1 & (p_2-p_3)s \\
	-(p_2+p_3)s & 1
	\end{bmatrix}
\]
if $p_3^2-p_2^2=0$, and
\[
e^{s \Omega } = 
e^{p_1s}\begin{bmatrix}
	\cos (s\sqrt{p_2^2-p_3^2}) & \frac{p_2-p_3}{\sqrt{p_2^2-p_3^2}}\sin (s\sqrt{p_2^2-p_3^2}) \\
	\frac{-p_2-p_3}{\sqrt{p_2^2-p_3^2}}\sin (s\sqrt{p_2^2-p_3^2}) & \cos (s\sqrt{p_2^2-p_3^2})
	\end{bmatrix}
\]
if $p_3^2-p_2^2<0$.

Now we are in the position to obtain the explicit form of the quadratic quantities.
In light of \eqref{E:qqs}, we see that
$\mathcal{I}^2 = \rho^2 - \mathcal{D}^2 - \mathcal{R}^2$.
Hence, on each connected component of $\{\mathcal{I} \neq0\}$, one has
$\mathcal{I} = \sigma \mathcal{Q}$
with a suitable $\sigma \in \{\pm1\}$ and
\begin{equation}\label{D:calQ}
	\mathcal{Q}(s):= \sqrt{\rho^2 - \mathcal{D}(t)^2 - \mathcal{R}(t)^2}.
\end{equation}
In particular, if $\mathcal{I}(0) \neq0$ then, recalling that $\mathcal{I} = \frac12\frac{ds}{dt}$ and $s(0)=0$, we obtain 
\begin{equation}\label{E:finds}
	\int_{0}^{s(t)} \frac{d\tau}{\mathcal{Q}(\tau)} = 2 t \sign(\mathcal{I}(0))
\end{equation}
at least around $t=0$.
If the left hand side is explicitly integrable then
we obtain an explicit form of the function $t\mapsto s(t)=\int_0^t 2\mathcal{I}(\tau) d\tau$.

Thus, once we obtain an explicit formula of $s(t)$, the explicit formulas for $\mathcal{D}$, $\mathcal{R}$, and $\mathcal{I}$. are given by  \eqref{E:qqq2} and
the formula
\begin{equation}\label{E:qqq3}
\mathcal{I}(t) =\tfrac12 s'(t),
\end{equation}
respectively.

\subsection{Study of the specific systems}

Let us obtain explicit solutions to \eqref{E:qqq} in some cases

\subsubsection{Case 4, 12, and 13}
In these three cases, the system \eqref{E:qqq} is (reduced to) a linear system.
The ODE system \eqref{E:qqq_4} is a linear system.
The system \eqref{E:qqq_12} is reduced to a linear one.
Indeed, since $\mathcal{D}$ is constant, one sees that the the equations
for $\mathcal{R}$ and $\mathcal{I}$ is a linear ODE system.
The system \eqref{E:qqq_13} is handled similarly.

\subsubsection{Cases 1, 2, and 5}
Let us begin with the case $p_3=p_4=0$ and $(p_1,p_2) \neq (0,0)$.
This gives a unified treatment for the Cases 1,2, and 5.
Recall that the ODE is of the form \eqref{E:qqq_5}.

We follow the standard strategy presented in the previous section.
In this case, one sees from \eqref{E:DRform} that
\[
	\begin{bmatrix}\mathcal{D}(t)\\\mathcal{R}(t)\end{bmatrix}
	= e^{p_1s}\begin{bmatrix}
	\cos (p_2 s) & \sin (p_2 s) \\
	-\sin (p_2 s) & \cos (p_2 s)
	\end{bmatrix}
	\begin{bmatrix}\mathcal{D}(0)\\\mathcal{R}(0)\end{bmatrix}.
\]
Then, $\mathcal{Q}(s)$ given in \eqref{D:calQ} takes the form
\[
	\mathcal{Q}(s)= \sqrt{\rho^2- (\mathcal{D}(0)^2+\mathcal{R}(0)^2)e^{2p_1s}}.
\]
By solving \eqref{E:finds}, one obtains
\[
	s(t) = 
	\begin{cases}
	- \tfrac1{p_1} 
	\log 
	\tfrac{\cosh (2p_1 \rho t + \frac12 \log \frac{\rho-\mathcal{I}(0)}{\rho+\mathcal{I}(0)})   }  {\cosh ( \frac12 \log \frac{\rho-\mathcal{I}(0)}{\rho+\mathcal{I}(0)})} &
p_1 \neq0,\\
	2t \mathcal{I}(0) &
p_1=0.
\end{cases}
\]
Hence, we have the explicit formula of the solution by \eqref{E:qqq2} and \eqref{E:qqq3}.

\subsubsection{Cases 3, 7, and 8}
To handle  Cases 3 and 7, it suffices to prove Lemma \ref{L:ellipticODE1}.
One can check that the functions in the statement of the lemma actually solve \eqref{E:ellipticODE1},
in light of \eqref{E:Jacobi_d}. Hence, the lemma follows by the standard uniqueness property of the ODE system.
Alternatively, one can integrate directly by mimicking the argument in the standard strategy.

Let us next prove Lemma \ref{L:ellipticODE2} to give a proof of results in Case 8.
\begin{proof}[Proof of Lemma \ref{L:ellipticODE2}]
We only consider the case $(f_0,g_0),(f_0,h_0)\neq(0,0)$, in which case $P>0$ and $R_{fg}>0$.
The relation $f=-h'$ is nothing but the third equation of \eqref{E:ellipticODE2}.
Further, since 
\[
	(	2 g - h^2 )' = 0
\]
follows from the second and the third equation of \eqref{L:ellipticODE2}, we obtain the relation
\[
	g(t) = \tfrac{h(t)^2 -h_0^2 }{2} + g_0.
\]

Let us obtain the explicit form of the function $h(t)$.
We remark that if a triplet $(F, G, H)$ solves \eqref{L:ellipticODE1}, i.e., if they satisfy
\[
	F' = GH, \quad G' = -FH, \quad H'= - FG
\]
then $f(t)= 2 F(t) G(t)$, $g(t)=G(t)^2 - H(t)^2$, and $h(t)=2H(t)$ becomes a solution to \eqref{L:ellipticODE2}.
Set $H_0=h_0/2$ and $(F_0,G_0)=  (R_{fg} \sin \theta_0, R_{fg} \cos \theta_0)$, where $\theta_0 \in [0,\pi)$ is given by
the relation $(f_0,g_0)= (R_{fg}^2 \sin 2\theta_0,R_{fg}^2 \cos 2\theta_0)$.
Then, we have $(2F_0G_0, G_0^2 - F_0^2, 2H_0)=(f_0,g_0,h_0)$.

We apply Lemma \ref{L:ellipticODE1}.
Note that $ (F_0^2 + G_0^2)^{1/2}= R_{fg} $
and 
\[
	(F_0^2 + H_0^2)^{1/2}= (R_{fg}^2 \sin^2  \theta_0 + (h_0/2)^2 ) ^{1/2} = P.
\]
Hence, if $P > R_{fg}$ then we have
\[
	H(t) =(\sign H_0 ) P \dn \( (\sign H_0) P t + t_0 , \tfrac{R_{fg}^2}{P^2} \),
\]
where $t_0$ is given by $(\sn( t_0, R_{fg}^2/P^2), \cn (t_0, R_{fg}^2/P^2)) = (F_0/R_{fg},G_0/R_{fg})$.
If $P=R_{fg}$ then we have
\[
	H(t) = (\sign H_0)  P \sech ( P t + t_0),
\]
where $t_0= (\sign F_0G_0H_0) \cosh^{-1} (P/|H_0|)$.
If $P < R_{fg} $ then we have
\[
	H(t) = \rho \cn \( (\sign G_0) \omega t + \tilde{t}_0, \tfrac{P^2}{R_{fg}^2} \)
	= \rho \cn \( \omega t + (\sign G_0)\tilde{t}_0 , \tfrac{P^2}{R_{fg}^2} \).
\]
It is not hard to verify that this $t_0$ is also characterized as in the statement.
\end{proof}

\subsubsection{Case 6}
We follow the standard strategy.
First, \eqref{E:DRform} reads as
\[
	\mathcal{D}(t) = e^{p_1 s} \mathcal{D}(0),
	\quad \mathcal{R}(t) = e^{p_1 s} (\mathcal{R}(0) - \tfrac{p_4}{p_1}\rho) + \tfrac{p_4}{p_1}\rho.
\]
Then, $\mathcal{Q}(s)$ given in \eqref{D:calQ} takes the form
\[
	\mathcal{Q}(s)= \sqrt{-C_1 e^{2p_1s} +  C_2 e^{p_1s} + C_3},
\]
where $C_1 =   \mathcal{D}(0)^2 + (\mathcal{R}(0)-(p_4/p_1)\rho)^2 \ge 0$, $C_2=-2\rho(p_4/p_1) (\mathcal{R}(0)-(p_4/p_1)\rho)$, and
\[
	C_3 = (\mathcal{I}(0))^2 + C_1 - C_2 = (1-(p_4/p_1)^2) \rho^2.
\]
We remark that  $C_3 \ge 0$ (resp. $C_3 \le0$)  if and only if $p_1\ge p_4$ (resp, $p_1 \le p_4$).

The case $C_1=0$ corresponds to a stationary solution.
Note that it occurs only when $p_4 \le p_1$ since $\mathcal{R}(0) \le \rho$. 
We see that the stationary solution is $(0,\frac{p_4}{p_1}\rho , \pm \sqrt{1-(\frac{p_4}{p_1})^2} \rho)$.
Hence, we suppose $C_1>0$ in what follows.

We integrate \eqref{E:finds}.
By applying the change of variable $x= e^{p_1 \tau}$, the right hand becomes
\begin{align*}
	 \tfrac{1}{p_1} \int^{e^{p_1s}}_1 \tfrac{dx}{x\sqrt{- C_1 x^2 + C_2 x +C_3}}  
	=
	\begin{cases} 
	\left.\frac1{p_1C_3^{1/2}} \log \abs{\frac{ x}{C_2 x + 2C_3+2 \sqrt{C_3(-C_1 x^2 + C_2 x + C_3)}} } \right|^{x=e^{p_1s(t)}}_{x=1} & p_1 > p_4,\\
	\left. - \frac2{p_1C_2} \sqrt{\frac{C_2- C_1x}{x} } \right|^{x=e^{p_1s(t)}}_{x=1} & p_1=p_4,\\
	\left.\frac1{p_1|C_3|^{1/2}} \arcsin \frac{ C_2x+2C_3}{x  \sqrt{C_2^2+4C_1C_3} } \right|^{x=e^{p_1s(t)}}_{x=1} & p_1 < p_4.
	\end{cases}
\end{align*}
Note that $C_2^2+4C_1C_3 >0$ and that $C_2\ge C_1>0$ if $C_3\le 0$ ($\Leftrightarrow p_1 \ge p_4$).

Let us consider the case $p_1>p_4$. 
Hence, \eqref{E:finds} reads as
\[
	e^{p_1 s(t)} = \tfrac{2C_3}{ (C_2^2+4C_1C_3)^{1/2} \cosh (2p_1 C_3^{1/2} t - \tau_0) - C_2  },
\]
where 
\[
	\tau_0 =(\sign \mathcal{I}(0)) \log \tfrac{C_2+2C_3 + 2 \sqrt{C_3}|\mathcal{I}(0)|}{(C_2^2+4C_1C_3)^{1/2}}
	=(\sign \mathcal{I}(0)) \cosh^{-1} ((2C_3 + C_2) (C_2^2+4C_1C_3)^{-1/2}).
\]
Thus, by means of \eqref{E:qqq2} and \eqref{E:qqq3}, one has
\[
	\mathcal{D}(t) = \tfrac{2C_3\mathcal{D}(0)}{ (C_2^2+4C_1C_3)^{1/2} \cosh (2p_1 C_3^{1/2} t - \tau_0) - C_2 },
\]
\[
	\mathcal{R}(t) = \tfrac{2C_3 (\mathcal{R} (0) - (p_4/p_1)\rho) }{ (C_2^2+4C_1C_3)^{1/2}\cosh (2p_1 C_3^{1/2} t - \tau_0) - C_2  } + \tfrac{p_4}{p_1} \rho,
\]
and
\[
	\mathcal{I}(t) = \tfrac{ - C_3^{1/2}(C_2^2+4C_1C_3)^{1/2} \sinh (2p_1 C_3^{1/2} t - \tau_0) }{ (C_2^2+4C_1C_3)^{1/2} \cosh (2p_1 C_3^{1/2} t - \tau_0) - C_2 }.
\]

Next consider the case $p_1=p_4$. 
It follows that $C_3=0$. Hence,
we obtain the solution
\[
	e^{p_1 s(t)} = \tfrac{C_2}{ (C_2p_1t-\mathcal{I}(0))^2 + C_2 - (\mathcal{I}(0))^2} 
\]
from \eqref{E:finds},
where $C_2 = 2 \rho(\rho - \mathcal{R}(0)) \ge 0$.
Hence, we see from \eqref{E:qqq2} and \eqref{E:qqq3} that
and
\[
	\mathcal{D}(t) = \tfrac{C_2\mathcal{D}(0)}{(C_2p_1t-\mathcal{I}(0))^2 + C_2 - (\mathcal{I}(0))^2},
\]
\[
	\mathcal{R}(t)= \tfrac{C_2 (\mathcal{R} (0) - \rho)}{(C_2p_1t-\mathcal{I}(0))^2 + C_2 - (\mathcal{I}(0))^2}
	+ \rho,
\]
and
\[
	\mathcal{I}(t) = - \tfrac{C_2(p_1C_2t- \mathcal{I}(0))}{(C_2p_1t-\mathcal{I}(0))^2 + C_2 - (\mathcal{I}(0))^2}.
\]

When $p_1<p_4$, it follows from \eqref{E:finds} that
\[
	e^{p_1 s(t)} = \tfrac{2|C_3|}{ C_2 - \sqrt{C_2^2 + 4C_1C_3} \cos (2 p_1 |C_3|^{1/2} t +  \tau_0)} ,
\]
where
\[
	\tau_0 = (\sign \mathcal{I}(0))  \arccos \tfrac{C_2 + 2C_3}{ \sqrt{C_2^2 + 4C_1C_3} }.
\]
Hence, we use \eqref{E:qqq2} and \eqref{E:qqq3} to obtain
\[
	\mathcal{D}(t) = \tfrac{2|C_3|\mathcal{D}(0)}{ C_2 - \sqrt{C_2^2 + 4C_1C_3} \cos (2 p_1 |C_3|^{1/2} t +  \tau_0)},
\]
\[
	\mathcal{R}(t)= \tfrac{2|C_3| (\mathcal{R} (0) - (p_4/p_1)\rho) }{ C_2 - \sqrt{C_2^2 + 4C_1C_3} \cos (2 p_1 |C_3|^{1/2} t +  \tau_0)}	+ \tfrac{p_4}{p_1} \rho,
\]
and
\[
	\mathcal{I}(t) = 
	\tfrac{|C_3|^{1/2} \sqrt{C_2^2 + 4C_1C_3} \sin (2 t_1 |C_3|^{1/2} t +  \tau_0) }{ C_2 - \sqrt{C_2^2 + 4C_1C_3} \cos (2 t_1 |C_3|^{1/2} t +  \tau_0)}.
\]

\subsubsection{Cases 9 and 10}

To handle these two cases,
let us prove Lemma \ref{L:ellipticODE3}. 
We integrate \eqref{E:ellipticODE3} by using an argument similar to the standard strategy.
Although the orbit is deduced simply, it needs some computation to obtain its parametrization.

\begin{proof}[Proof of Lemma \ref{L:ellipticODE3}]
Note that if $(f,g,h)(t)$ is a solution to \eqref{E:ellipticODE3} then so is $(f,-g,-h)(t)$.
Hence, we may suppose that $h_0\ge0$ without loss of generality.
Note that $f(t)^2+g(t)^2$ and $h(t)^2 - (f(t) + \eta)^2$ are conserved.
We denote these values $R_0$ and $K_0$, respectively.
Then, the orbit of the solution  is a subset of the intersection of $\Upsilon:=\{(x,y,z) \in \R^3 \ |\ x^2 + y^2 = R_0^2 \}$
and $\Sigma:=\{(x,y,z) \in \R^3 \ |\ z^2 - (x+\eta)^2 = K_0 \}$.

We first note that any point in the $x$-axis
$\{(x,0,0) \ |\ f \in \R \}$, $z$-axis $\{(0,0,z) \ |\ h_0 \in \R \}$, or the line $\{ (-\eta,y,0) \ |\ y \in \R\}$ is a stationary solution.
Hence, we consider the other solutions in the sequel. In particular, we may suppose $R_0>0$.

Since $R_0=f^2 + g^2$ is a conserved quantity, we write
\[
	f = R_0 \cos \tau, \quad g = R_0 \sin \tau.
\]
Then, we have $\tau' = h$ from the first and the second equation of \eqref{E:ellipticODE3}.
Since $(f,g,h) \in \Sigma$, one sees that
\[
	(\tau')^2 =  K_0 + (R_0 \cos \tau + \eta)^2.
\]
Using the fact that $\tau'(0) = h_0 \ge0$, we have
\begin{equation}\label{E:fghepf1}
	\int^\tau	\tfrac{ dy }{ \sqrt{ K_0 + (R_0 \cos y + \eta)^2 } } = t + C
\end{equation}
with an integral constant $C$, at least for small time.

\subsection*{Step 1 (the subcase $K_0<0$)}

Let us begin with the case $K_0<0$. Notice that $\Sigma = \Sigma_+ \cup \Sigma_-$,
where
\[
	\Sigma_\pm := \Sigma \cap \{(x,y,z) \in \R^3 \ |\ \pm x > -\eta\} =
	 \{ ( -\eta \pm \sqrt{|K_0|} \cosh a ,y, \sqrt{|K_0|} \sinh a) \ |\ a,y \in \R  \}.
\]
Note that $\Upsilon \cap \tilde{\Sigma}_+$ and $\Upsilon \cap \tilde{\Sigma}_-$ are closed curves.
As mentioned above, we only consider the solution in $\Upsilon \cap \tilde{\Sigma}_+$ since the other is obtained by the symmetry around $x$-axis.

Since $\Sigma_+$ and $\Sigma_-$ are disjoint, if the solution belongs to one of them at some time then it does to the same one for all time.
By the change $t= \cos y$,
the left hand side of \eqref{E:fghepf1} becomes
\begin{align*}
	-\tfrac{(\sign g_0)}{R_0 }\int^{ \cos \tau } 
	\tfrac{  dt }{ \sqrt{ (1-t)(t-\gamma_+ )(t-\gamma_-)(t+1) } } ,
\end{align*}
where $$\gamma_\pm = \tfrac{-\eta \pm \sqrt{-K_0}}{R_0}$$ are the two roots of $R_0^2 x^2 + 2R_0 \eta x + \eta^2+ K_0= 0$.
This is the well-known elliptic integrals.
We remark that the geometric meaning of $\gamma_\pm$ is as follows:
The line $\{ (R_0 \gamma_+ , y, 0) \ |\ y \in \R \}$ (resp.\ $\{ (R_0 \gamma_- , y, 0) \ |\ y \in \R \}$) corresponds to
 the intersection of the $xy$-plane $\{ (x,y,0) \ |\ x,y \in \R  \}$ and  $\Sigma_+$ (resp.\ $\Sigma_-$).
From this, one sees that $R_0 \gamma_+ \le R_0$, i.e., $\gamma_+ \le 1$.
Further, if $\gamma_+ = 1$ then the solution is the fixed point $(R_0,0,0)$.
Hence, we may suppose that $\gamma_+<1$ in the sequel.

If $-1 < \gamma_- < \gamma_+ < 1$, i.e., if $R_0 > \eta+\sqrt{-K_0}$ then $\cos \tau \in [-1,\gamma_-]\cup[\gamma_+,1]$ holds.
We remark that $\cos \tau \in [-1,\gamma_-]$ (resp.\ $\cos \tau \in [\gamma_+,1]$) corresponds to the case  $(f,g,h) \in \Upsilon \cap \Sigma_-$
(resp.\ $(f,g,h) \in \Upsilon \cap \Sigma_+$).
In the both cases, the solution is periodic in time.
If $\cos \tau \in [-1,\gamma_-]$ then, thanks to \cite{BFBook}*{252.00}, \eqref{E:fghepf1} reads as
\[
 	t+C =  - \tfrac{(\sign g_0)}{R_0 } \tfrac{2}{ \sqrt{ (1-\gamma_-)(1+\gamma_+) }}
 	\sn^{-1} \( \sqrt{\tfrac{(1-\gamma_-)(\cos \tau+1)}{(\gamma_-+1)(1-\cos \tau)}}, \tfrac{(1-\gamma_+)(\gamma_-+1)}{(1-\gamma_-)(\gamma_++1)} \).
\]
Therefore,
\[
		\tfrac{(1-\gamma_-)(\cos \tau+1)}{(\gamma_-+1)(1-\cos \tau)}
		=  \sn^2 \( \tfrac{R_0 \sqrt{ (1-\gamma_-)(1+\gamma_+) }}{2} t+t_0 , \tfrac{(1-\gamma_+)(\gamma_-+1)}{(1-\gamma_-)(\gamma_++1)}\).
\]
Hence,
\[
	f 
	=R_0 \tfrac{ -(R_0 + \eta + \sqrt{-K_0}) +(R_0 - \eta - \sqrt{-K_0}) \sn^2\( \theta t+t_0 , m_0 \) }{(R_0 + \eta + \sqrt{-K_0}) +(R_0 - \eta - \sqrt{-K_0}) \sn^2\( \theta t+t_0 , m_0 \)},
\]
\[
	g =  R_0
	\tfrac{ 2 \sqrt{R_0^2 - (\eta + \sqrt{-K_0})^2} \sn\( \theta t+t_0 , m_0 \) }{{(R_0 + \eta + \sqrt{-K_0}) +(R_0 - \eta - \sqrt{-K_0}) \sn^2\( \theta t+t_0 , m_0 \)}},
\]
and
\begin{align*}
	h={}& 
	\tfrac{  (R_0 + \eta + \sqrt{-K_0})\sqrt{(R_0-\eta)^2 + K_0}  }
	{{(R_0 + \eta + \sqrt{-K_0}) +(R_0 - \eta - \sqrt{-K_0}) \sn^2\( \theta t+t_0 , m_0 \)}} 
	\cn\( \theta t+t_0 , m_0 \)
	 \dn\( \theta t+t_0 , m_0 \)
	,
\end{align*}
where
\[
	\theta = \tfrac{ \sqrt{ (R_0+\sqrt{-K_0})^2 - \eta^2 }}{2}  ,\quad m_0 = \tfrac{(R_0-\sqrt{-K_0})^2 - \eta^2}{(R_0+\sqrt{-K_0})^2 - \eta^2}.
\]
Similarly, if $\cos \tau \in [\gamma_+,1]$ then one sees from \cite{BFBook}*{256.00} that
\[
 	t+C =  - \tfrac{(\sign g_0)}{R_0 } \tfrac{2}{ \sqrt{ (1-\gamma_-)(1+\gamma_+) }}
 	\sn^{-1} \( \sqrt{\tfrac{(1-\gamma_-)(\cos \tau-\gamma_+)}{(1-\gamma_+)(\cos \tau-\gamma_-)}}, {\tfrac{(1-\gamma_+)(\gamma_-+1)}{(1-\gamma_-)(\gamma_++1)}} \).
\]
Therefore,
\[
		\tfrac{(1-\gamma_-)(\cos \tau-\gamma_+)}{(1-\gamma_+)(\cos \tau-\gamma_-)}
		=  \sn^2 \( \tfrac{R_0 \sqrt{ (1-\gamma_-)(1+\gamma_+) }}{2} t+t_0 ,{\tfrac{(1-\gamma_+)(\gamma_-+1)}{(1-\gamma_-)(\gamma_++1)}}\).
\]
One then sees that
\[
	f=-\eta  + \tfrac{ \sqrt{-K_0}( (R_0 + \eta + \sqrt{-K_0})+(R_0 + \eta -\sqrt{-K_0}) \sn^2\( \theta t+t_0 ,m_0 \) }{(R_0 + \eta + \sqrt{-K_0}) -(R_0 + \eta -\sqrt{-K_0}) \sn^2\( \theta t+t_0 , m_0 \)},
\]
and
\begin{align*}
	g ={}& - \tfrac{ R_0 (R_0 + \eta + \sqrt{-K_0}) \sqrt{R_0^2 -(\eta - \sqrt{-K_0})^2} }{(R_0 + \eta + \sqrt{-K_0}) -(R_0 + \eta -\sqrt{-K_0}) \sn^2\( \theta t+t_0 , m_0 \)}
	\cn\( \theta t+t_0 , m_0 \)
	\dn\( \theta t+t_0 , m_0 \)
\end{align*}
and
\[
	h= \tfrac{ 2 \sqrt{-K_0} \sqrt{(R_0 + \eta)^2 + K_0} \sn\( \theta t+t_0 , m_0 \)}
	{(R_0 + \eta + \sqrt{-K_0}) -(R_0 + \eta -\sqrt{-K_0}) \sn^2\( \theta t+t_0 , m_0 \)},
\]
where
\[
	\theta = \tfrac{ \sqrt{ (R_0+\sqrt{-K_0})^2 - \eta^2 }}{2} , \quad
	m_0 =  \tfrac{(R_0-\sqrt{-K_0})^2 - \eta^2}{(R_0+\sqrt{-K_0})^2 - \eta^2} .
\]

Let us proceed to the case $ \gamma_-<-1 < \gamma_+ < 1$, i.e., the case $| \eta - R_0 |< \sqrt{-K_0}$. 
In this case, $\Upsilon \cap \Sigma_- = \emptyset$.
The solution is a closed curve in $\Upsilon \cap \Sigma_+$.
We use \cite{BFBook}*{256.00} to obtain
\[
 	t+C =  - \tfrac{(\sign g_0)}{R_0 } \tfrac{2}{ \sqrt{ 2(\gamma_+ -\gamma_-) }}
 	\sn^{-1} \( \sqrt{\tfrac{2(\cos\tau-\gamma_+)}{(1-\gamma_+)(\cos \tau+1)}},  {\tfrac{(-1-\gamma_-)(1-\gamma_+)}{2(\gamma_+-\gamma_-)}} \).
\]
This gives us
\[
	\tfrac{2(\cos\tau-\gamma_+)}{(1-\gamma_+)(\cos \tau+1)}
	=  \sn^2 \( \tfrac{R_0 \sqrt{ 2(\gamma_+-\gamma_-) }}{2} t+t_0 , {\tfrac{(-1-\gamma_-)(1-\gamma_+)}{2(\gamma_+-\gamma_-)}}\).
\]
Therefore,
\[
	f =-R_0  + \tfrac{ 2 R_0 (R_0 - \eta + \sqrt{-K_0}) }{2R_0 -(R_0 + \eta - \sqrt{-K_0}) \sn^2\( \theta t+t_0 , m_0 \)},
\]
\[
g= \tfrac{ 2 R_0 \sqrt{R_0^2 - (\eta - \sqrt{-K_0})^2} \cn\( \theta t+t_0 ,  m_0 \)}{2R_0 -(R_0 + \eta - \sqrt{-K_0}) \sn^2\( \theta t+t_0 ,  m_0\)},
\]
and
\begin{align*}
	h ={}& -\tfrac{ 2 R_0^{1/2} (-K_0)^{1/4}\sqrt{R_0^2 - (\eta - \sqrt{-K_0})^2} }{2R_0 -(R_0 + \eta - \sqrt{-K_0}) \sn^2\( \theta t+t_0 , m_0 \)} \sn\( \theta t+t_0 , m_0 \) 
	 \dn\( \theta t+t_0 , m_0 \),
\end{align*}
where
\[
	\theta = R_0^{1/2} (-K_0)^{1/4}, \quad 
	m_0 =  \tfrac{{\eta^2 - (R_0-\sqrt{-K_0})^2}}{4 R_0(-K_0)^{1/2}} .
\]

If $ \gamma_- < \gamma_+ <-1<1$, i.e., if $R_0 < \eta - \sqrt{-K_0}$ then
we see that $\Upsilon \cap \Sigma_-= \emptyset$. Further, $\Upsilon \cap \Sigma_+$ consists of two disjoint closed curves. One of the curves lies in $\{z>0\}$ and the other in $\{z<0\}$.
Again, by means of \cite{BFBook}*{256.00}, we have
\[
 	t+C =  - \tfrac{(\sign g_0)}{R_0 } \tfrac{2}{ \sqrt{ (1-\gamma_+)(-1-\gamma_-) }}
 	\sn^{-1} \( \sqrt{\tfrac{(1-\gamma_+)(\cos \tau+1)}{2(\cos \tau-\gamma_+)}},  {\tfrac{2(\gamma_+-\gamma_-)}{(1-\gamma_+)(-\gamma_--1)}} \).
\]
Therefore,
\[
		\tfrac{(1-\gamma_+)(\cos \tau+1)}{2(\cos\tau-\gamma_+)}
		=  \sn^2 \( \tfrac{R_0 \sqrt{ (1-\gamma_+)(-1-\gamma_-) }}{2} t+t_0 ,{\tfrac{2(\gamma_+-\gamma_-)}{(-1-\gamma_-)(1-\gamma_+)}}\).
\]
From this, one obtains
\[
	f =-R_0 + \tfrac{  2R_0(-R_0+\eta- \sqrt{-K_0}) \sn^2\( \theta t+t_0 , m_0 \)}{(R_0+\eta - \sqrt{-K_0}) -2R_0 \sn^2\( \theta t+t_0 , m_0 \)},
\]
\begin{align*}
	g={}& \tfrac{ - 2R_0 \sqrt{(\eta- \sqrt{-K_0})^2 -R_0^2} }{(R_0+\eta - \sqrt{-K_0}) -2R_0 \sn^2\( \theta t+t_0 , m_0 \)}  \sn\( \theta t+t_0 , m_0 \)  \cn \( \theta t+t_0 , m_0 \) ,
\end{align*}
and
\[
	h = \tfrac{  (R_0 + \eta- \sqrt{-K_0})\sqrt{(R_0-\eta)^2 + K_0} \dn\( \theta t+t_0 , m_0 \)}{(R_0+\eta - \sqrt{-K_0}) -2R_0 \sn^2\( \theta t+t_0 ,m_0 \)},
\]
where
\[
	\theta = \tfrac{ \sqrt{ \eta^2 - (R_0 - \sqrt{-K_0})^2 }}{2}  , \quad
	m_0  =  \tfrac{4 R_0 \sqrt{-K_0}}{\eta^2 - (R_0 - \sqrt{-K_0})^2}  .
\]

Let us consider the threshold cases.
If $\gamma_- = -1 <\gamma_+<1$, i.e., if $R_0 = \eta + \sqrt{-K_0}$ then $\Upsilon \cap \Sigma_- = \{(-R_0,0,0)\} $ holds
and $\Upsilon \cap \Sigma_+  $ is a closed curve.
Let us consider the case $(f,g,h)\in \Upsilon \cap \Sigma_+$. This implies $\cos \tau \in [\gamma_+,1]$.
We have
\begin{align*}
	t+C&{}=-\tfrac{(\sign g_0)}{R_0 }\int^{ \cos \tau } 
	\tfrac{  dt }{(t+1) \sqrt{ (1-t)(t-\gamma_+ ) } } \\
	&{}= -\tfrac{(\sign g_0) }{R_0 \sqrt{2(1+\gamma_+)}} \arcsin \( \tfrac{3+\gamma_+}{1-\gamma_+} - \tfrac{4(1+\gamma_+) }{(1-\gamma_+)(\cos \tau +1)} \) .
\end{align*}
This implies that
\[
	\tfrac{3+\gamma_+}{1-\gamma_+} - \tfrac{4(1+\gamma_+) }{(1-\gamma_+)(\cos \tau +1)}
	=  \sin \( R_0 \sqrt{2(1+\gamma_+)} t + t_0\) 
\]
and so that
\[
	f= R_0 -  \tfrac{   2 R_0\eta \sin^2  ( t \sqrt{ R_0 (R_0-\eta)}+t_0) }{ R_0 -\eta  \cos^2 ( t \sqrt{ R_0 (R_0-\eta)}+t_0)  } .
\]
Similarly,
\[
	g= \tfrac{ 2R_0 \sqrt{\eta (R_0 - \eta )}  \sin  ( t \sqrt{ R_0 (R_0-\eta)}+t_0) }{ R_0 - \eta   \cos^2 ( t \sqrt{ R_0 (R_0-\eta)}+t_0)  }, \quad
	h =   \tfrac{ 2(R_0 - \eta )\sqrt{R_0\eta}  \cos  ( t \sqrt{ R_0 (R_0-\eta)}+t_0) }{ R_0  - \eta  \cos^2 (t  \sqrt{ R_0 (R_0-\eta)}+t_0) }.
\]

Let us consider the case $\gamma_- < \gamma_+=-1 <1$, i.e., $R_0 = \eta - \sqrt{-K_0}$.
In this case $\Upsilon$ and $\Sigma_+$ are tangent at $(-R_0,0,0)=(-\eta +\sqrt{-K_0},0,0)$.
One sees that the orbit of the solution is an open curve of which ends are the stationary point.
Let us obtain an explicit formula of the solution. 
We have
\begin{align*}
	t+C&{}=-\tfrac{1}{R_0 }\int^{ \cos \tau } 
	\tfrac{  dt }{(t+1) \sqrt{ (1-t)(t-\gamma_- ) } } \\
	&{}= -\tfrac{1 }{R_0 \sqrt{-2(1+\gamma_-)}} \log \abs{ 3+\gamma_- + \tfrac{-4(1+\gamma_-)- 2 \sqrt{-2(1+\gamma_-)(1-\cos \tau)(\cos \tau - \gamma_-)} }{\cos \tau +1} } 
\end{align*}
Hence,
\[
	h = -f'/g= \tfrac{ R_0\sqrt{2(\gamma_-^2-1)} \cosh (R_0 \sqrt{\frac{|\gamma_-|-1}2} t + t_0) }{(1+|\gamma_-|) \sinh^2 (R_0 \sqrt{\frac{|\gamma_-|-1}2} t + t_0) +|\gamma_-|-1}.
\]
This gives us
\[
	f= -R_0 + \tfrac{ 2 R_0 (\eta-R_0) }{\eta \cosh^2 ( t \sqrt{R_0(\eta-R_0)}  + t_0) -R_0}
\]
and
\[
	g=  \tfrac{2 R_0  \sqrt{\eta (\eta-R_0)} \sinh (t \sqrt{R_0 (\eta-R_0)}  + t_0) }{\eta \cosh^2 (t \sqrt{R_0 (\eta-R_0)}  + t_0) - R_0}, \quad
	h =  \tfrac{ R_0\sqrt{2\eta (\eta-R_0)} \cosh (t \sqrt{R_0 (\eta-R_0)}  + t_0) }{\eta \cosh^2 (t \sqrt{R_0 (\eta-R_0)}  + t_0) - R_0}.
\]
This completes the proof of the case $K_0<0$.

\subsection*{Step 2 (the subcase $K_0=0$)}
Next, consider the case $K_0=0$, i.e., $h_0^2= (f_0+\eta)^2$. 
In this case, one sees that $\Sigma$ is a union of two planes given by $x+\eta = z$ and $x+\eta = -z$.
$\Upsilon \cap \Sigma$ consists of two ellipses.
If $R_0 < \eta$ then the ellipses do not intersect each other. Hence, the orbit of the solution is either one of the ellipses, and the solution is periodic in time.
If $R_0= \eta$ then the ellipses intersect at one point $(-\eta,0,0)$. 
The orbit of the solution is one of the ellipses except for the point. The solution tends to the point as $t\to\pm \infty$.
If $R_0 > \eta$ then the ellipses intersect at two points $(-\eta,\pm \sqrt{R_0^2-\eta^2},0)$.
The solution tends to one of the points as $t\to-\I$ and the other as $t\to\I$.
The direction is easily obtained by the system.

Let us obtain the explicit formula of the solution.
If $f_0+\eta>0$ then
\[
	t +C = \int^\tau \tfrac{dy}{R_0 \cos y + \eta}  
	= 
	\begin{cases}
	\frac2{\sqrt{\eta^2-R_0^2}} \arctan \( \sqrt{\frac{\eta-R_0}{\eta+R_0}} \tan \frac{\tau}2\) &  R_0 < \eta ,\\
	\frac1R_0 \tan \frac{\tau}2 & R_0=\eta,\\
	\frac1{\sqrt{R_0^2-\eta^2}} \log \abs{ \frac{\sqrt{R_0^2-\eta^2} \sin \tau + \eta \cos \tau + R_0 }{ R_0 \cos \tau + \eta} } &  R_0 > \eta.
	\end{cases}
\]
Hence,
\[
	f+ \eta = h =  \tfrac{\eta^2 - R_0^2 }{ \eta - R_0 \cos \( \sqrt{\eta^2-R_0^2} t+ t_0 \) },\quad
	g= \tfrac{R_0 \sqrt{\eta^2 - R_0^2 } \sin \( \sqrt{\eta^2-R_0^2} t+ t_0 \) }{ \eta - R_0 \cos \( \sqrt{\eta^2-R_0^2} t+ t_0 \) }
\]
if $R_0 < \eta$,
\[
	f + \eta = h = \tfrac{ 2\eta }{ 1+ (\eta t + t_0)^2}, \quad
	g = \tfrac{2(\eta t + t_0) }{ 1+ (\eta t + t_0)^2}
\]
if $R_0 = \eta$, and
\[
	f+\eta = h = \tfrac{R_0^2-\eta^2}{ R_0 \cosh (t \sqrt{R_0^2-\eta^2} + t_0) - \eta }, \quad
	g = \tfrac{ R_0 \sqrt{R_0^2 -\eta^2} \sinh (t \sqrt{R_0^2-\eta^2} + t_0) }{ R_0 \cosh (t \sqrt{R_0^2-\eta^2} + t_0) - \eta }
\]
if $R_0 > \eta$.
On the other hand, if $f_0+ \eta < 0$ then we have $ R_0 \ge -f_0 > \eta$. Hence,
\[
	- t +C = \int^\tau \tfrac{dy}{R_0 \cos y + \eta}  
	= 
	\tfrac1{\sqrt{R_0^2-\eta^2}} \log \abs{ \tfrac{\sqrt{R_0^2-\eta^2} \sin \tau + \eta \cos \tau + R_0 }{ R_0 \cos \tau + \eta} }. 
\]
This gives us
\[
	f+ \eta = - h = - \tfrac{\eta}{R_0 \cosh (t \sqrt{R_0^2-\eta^2} + t_0)} ,\quad
	g = - \tfrac{R_0 \sqrt{R_0^2-\eta^2} \sinh (t \sqrt{R_0^2-\eta^2} + t_0)  }{R_0 \cosh (t \sqrt{R_0^2-\eta^2} + t_0)}
\]
as above.

\subsection*{Step 3 (the subcase $K_0>0$)}
Finally, we consider the case $K_0>0$. 
In this case, one has $\Sigma = \tilde{\Sigma}_+ \cup \tilde{\Sigma}_-$, where
\[
	\tilde{\Sigma}_\pm 
	:= \Sigma \cap \{(x,y,z) \in \R^3 \ |\ \pm z > 0 \} =
	 \{ ( -\eta + \sqrt{K_0} \sinh a , y, \pm \sqrt{K_0} \cosh a) \ |\ a,y \in \R  \}.
\]
Hence, $\Upsilon \cap \Sigma$ consists of two closed curves. One belongs to $\tilde{\Sigma}_+$
and the other to $\tilde{\Sigma}_-$. The solutions are periodic in time.

Let us obtain the explicit formula of the solutions.
As mentioned above, we only consider a solution in $\tilde{\Sigma}_+$.
That for the other solution is obtained by the symmetry with respect the $x$-axis.
By applying the change $t= \cos y$ to the integral in \eqref{E:fghepf1}, we obtain
\[
- \tfrac{\sign g_0}{R_0}
	\int^{\cos \tau}  \tfrac{ dt}{ \sqrt{(1-t^2)( t^2 + 2  (\eta/R_0) t + (K_0 + \eta^2)/R_0^2 )} } = t+C.
\]
We denote $b=\eta/R_0>0$ and $m_0=\sqrt{K_0+\eta^2}/R_0>0$.
Let $\xi \in (-1,0)$ be the larger root of $b\xi^2 + (m_0^2+1)\xi + b=0$. Namely,
\[
	\xi = \tfrac{ -(m_0^2+1)+ \sqrt{(m_0^2+1)^2-4b^2} }{2b} . 
\]
Note that the other root is $1/\xi \in (-\infty,-1)$.
Let us introduce a new variable $y$ by $t=(\xi^{-1} y + \xi)/(y+1)$. Then,
\[
	\int^{\cos \tau}  \tfrac{ dt}{ \sqrt{(1-t^2)( t^2 + 2 b t + m_0^2 )} }
	= A\int^{ \frac{\xi - \cos \tau }{\cos \tau -1/\xi} } \tfrac{ dy }{ \sqrt{(\xi^2-y^2)(y^2+\nu^2)} },
\]
where
\[
	A =  \tfrac{ \xi \sqrt{1-\xi^2} }{ \sqrt{ m_0^2 \xi^2 + 2  b  \xi +1 } },\quad
	\nu = 
	\tfrac{-\xi\sqrt{\xi^2 + 2b \xi + m_0^2}}{ \sqrt{m_0^2 \xi^2 + 2  b  \xi +1}}.
\]
It follows from \cite{BFBook}*{213.00}
that the integral in the right hand side is written in terms of the inverse of the elliptic functions.
We have
\[
	f = R_0 \tfrac{ \xi +  \cn \(\theta t + t_0 , m_0 \) }{ 1 + \xi \cn \(\theta t + t_0, m_0  \)}, \quad
	g 
	= R_0 \tfrac{  \sqrt{1-\xi^2} \sn \(\theta t + t_0 , m_0 \)}{ 1 + \xi  \cn \(\theta t + t_0 , m_0 \) },
\]
and
\[
	h =   \theta \tfrac{  \sqrt{1-\xi^2} \dn \(\theta t + t_0 , m_0 \)}{ 1 + \xi   \cn \(\theta t + t_0 , m_0 \) }
\]
for suitable $t_0 \in \R$, where
\[
	\theta =  ((R_0+\eta)^2 + K_0)^{\frac14} ((R_0-\eta)^2 + K_0)^{\frac14}
\]
and
\[
	m_0 =  \tfrac{ \theta^2+ R_0^2  - K_0 - \eta^2  }{ 2\theta^2 }.
\]
Note that $\xi$ is written as
\[
	\xi =- \tfrac{ 2 \eta R_0 } { K_0 + \eta^2+ R_0^2 + \sqrt{(K_0 + (\eta+ R_0)^2)(K_0 + (\eta- R_0)^2)  } }
	= - \tfrac{ 2 \eta R_0 } { K_0 + \eta^2+ R_0^2 + \theta^2 }.
\]
By regarding $-\xi$ as $\xi$, we obtain the formula in the statement.
\end{proof}



\subsubsection{Case 11}

Next we consider the case $p_2=p_4=p_5=0$.
We follow the standard argument.
In this case, one sees from \eqref{E:DRform} that
\begin{align*}
	\mathcal{D} (t) &= \tfrac{\mathcal{D}(0)+\mathcal{R}(0)}{2} e^{(p_1-p_3)s} + \tfrac{\mathcal{D}(0)-\mathcal{R}(0)}{2}
	e^{(p_1+p_3)s}, \\
	\mathcal{R} (t) &= \tfrac{\mathcal{D}(0)+\mathcal{R}(0)}{2} e^{(p_1-p_3)s}- \tfrac{\mathcal{D}(0)-\mathcal{R}(0)}{2}
	e^{(p_1+p_3)s}.
\end{align*}
Then, $\mathcal{Q}(s)$ given in \eqref{D:calQ} takes the form
\[
	\mathcal{Q}(s)= \sqrt{\rho^2- c_+ e^{2(p_1+p_3)s}  -c_- e^{2(p_1-p_3)s}}
\]
where $c_\pm =\tfrac12(\mathcal{D}(0)\mp \mathcal{R}(0))^2$. 
Note that $0 \le c_\pm $ and $c_++c_-= \rho^2-\mathcal{I}(0)^2 \le \rho^2$.
Then,
\begin{equation}\label{E:Jacobi_criteria}
	  \int_0^{s(t)}  \frac{d \tau }{ \sqrt{\rho^2- c_+ e^{2(p_1+p_3)\tau}  -c_- e^{2(p_1-p_3)\tau} }} =2 t\sign \mathcal{I}(0)
\end{equation}

When $p_1=p_3$, one introduces $w=e^{4p_1 \tau}$ to obtain
\[
	  \int_0^{s(t)}  \frac{d \tau }{ \sqrt{\rho^2- c_+ e^{2(p_1+p_3)\tau}  -c_- e^{2(p_1-p_3)\tau} }}
	  = \tfrac{1}{4p_1}  \int^{e^{4p_1 s(t)}}_1  \frac{d w }{ w\sqrt{(\rho^2  -c_-) - c_+ w }}.
\]
Then,
\[
	s(t) = \tfrac1{4p_1}\log \tfrac{\rho^2- c_-}{c_+} - \tfrac1{2p_1} \log \cosh (4 p_1 \sqrt{\rho^2 - c_-} t + \tfrac12 \log \tfrac{ \sqrt{\rho^2-c_-} - \mathcal{I}(0) }{ \sqrt{\rho^2-c_-} + \mathcal{I}(0) })
\]
and hence
\[
	\mathcal{I}(t)= - \sqrt{\rho^2 - c_-} \tanh (4 p_1 \sqrt{\rho^2 - c_-} t + \tfrac12 \log \tfrac{ \sqrt{\rho^2-c_-} - \mathcal{I}(0) }{ \sqrt{\rho^2-c_-} + \mathcal{I}(0) })
\]
and
\begin{align*}
	\mathcal{D} (t) &= \tfrac{\mathcal{D}(0)+\mathcal{R}(0)}{2}+ \tfrac{\mathcal{D}(0)-\mathcal{R}(0)}{2}\sqrt{\tfrac{\rho^2 - c_- }{c_+}} \sech (4 p_1 \sqrt{\rho^2 - c_-} t + \tfrac12 \log \tfrac{ \sqrt{\rho^2-c_-} - \mathcal{I}(0) }{ \sqrt{\rho^2-c_-} + \mathcal{I}(0) })),  \\
	\mathcal{R} (t) &= \tfrac{\mathcal{D}(0)+\mathcal{R}(0)}{2}- \tfrac{\mathcal{D}(0)-\mathcal{R}(0)}{2}\sqrt{\tfrac{\rho^2 - c_- }{c_+}} \sech (4 p_1 \sqrt{\rho^2 - c_-} t + \tfrac12 \log \tfrac{ \sqrt{\rho^2-c_-} - \mathcal{I}(0) }{ \sqrt{\rho^2-c_-} + \mathcal{I}(0) })) .
\end{align*}

When $p_1=3p_3$, we introduce $e^{4p_3 \tau} = w$ to get
\[
	  \int_0^{s(t)}  \frac{d \tau }{ \sqrt{\rho^2- c_+ e^{2(p_1+p_3)\tau}  -c_- e^{2(p_1-p_3)\tau} }}
	  = \tfrac{1}{4p_3}  \int^{e^{4p_3 s(t)}}_1  \frac{d w }{ w\sqrt{\rho^2  -c_-w - c_+ w^2 }}.
\]
One has
\[
	e^{4p_3 s(t)} = 
	\tfrac{
	4 \rho^2
	}{
	{\sqrt{ 8\rho^2(\mathcal{D}(0)-\mathcal{R}(0))^2+ (\mathcal{D}(0)+\mathcal{R}(0))^4 }}\cosh (8p_3 \rho \tau + \tau_0) + {(\mathcal{D}(0)+\mathcal{R}(0))^2}
	},
\]
which yields the result.

When $p_1=\tfrac13 p_3$, we introduce $ w = e^{-2(p_1-p_3)\tau} = e^{\frac43 p_3 \tau}$ to get
\[
	  \int_0^{s(t)}  \frac{d \tau }{ \sqrt{\rho^2- c_+ e^{2(p_1+p_3)\tau}  -c_- e^{2(p_1-p_3)\tau} }}
	  = \tfrac{3}{4p_3}  \int^{e^{\frac43 p_3 s(t)}}_1  \frac{d w }{ \sqrt{w(- c_+ w^3+\rho^2w  -c_- ) }}.
\]
We put $P(w)=- c_+ w^3+\rho^2w  -c_- $. 
If $c_+>0$ then the equation $P(w)=0$ has three real solutions.
Let $\alpha \le \beta \le \gamma$ be the three roots.

First note that $c_+=c_-=0$ implies that $(\mathcal{D}(0),\mathcal{R}(0),\mathcal{I}(0))=(0,0,\pm \rho)$.
Hence, we consider the other case.
If $c_+=0$ and $c_- > 0$ then 
one has
\[
	e^{-\frac23 p_3 s(t)} = \tfrac{\rho}{\sqrt{c_-} \cosh (\frac43 p_3 \rho \tau + \tau_0)},
\]
where $\tau_0 = (\sign \mathcal{I}(0) )\cosh^{-1} \frac{\rho}{ \sqrt{c_-}}$.
The explicit formula of the solution immediately follows.
If $c_+>0$ and $c_-=0$ then the primitive is given as follows:
\[
	\int^x  \tfrac{d w }{ w\sqrt{\rho^2 - c_+ w^2 }}
	= \tfrac1\rho \log|\tfrac{x}{\rho + \sqrt{\rho^2 - c_+ x^2}}|.
\]
Hence,
\[
	e^{\frac43 p_3 s(t)} = \tfrac{\rho}{\sqrt{c_+} \cosh (\frac83 p_3 \rho \tau + \tau_0)},
\]
where $\tau_0 = -(\sign \mathcal{I}(0) )\cosh^{-1} \frac{\rho}{ \sqrt{c_+}}$.
The explicit formula of the solution immediately follows.

Let us consider the case $c_+,c_->0$ and $\mathcal{I}(0)=0$.
The latter relation reads as $\rho^2 = c_+ + c_-$.
The subcase $2c_+ = c_-$ corresponds to the constant solution.
Let us consider the case $2c_+>c_-$. In this case, the three roots satisfies $\alpha<0<\beta<1=\gamma$.
Hence, as long as $\beta<x<1$, one has
\begin{align*}
	\int^x_1  \frac{d w }{ \sqrt{w(- c_+ w^3+\rho^2w  -c_- ) }}
	={}&\tfrac1{\sqrt{c_+}}	\int^x_1  \frac{d w }{ \sqrt{(w-\alpha)w(w-\beta)(1-w) }} \\
	={}&\tfrac1{\sqrt{c_+}} \tfrac2{\sqrt{\beta-\alpha}} \(\sn^{-1} \(  \sqrt{\tfrac{ x-\beta }{ (1-\beta)x }}, 
	{\tfrac{ (1-\beta)(-\alpha) }{ \beta-\alpha }}
	\) - K \(  
	{\tfrac{ (1-\beta)(-\alpha) }{ \beta-\alpha }}
	\)\),
\end{align*}
where $K(m)$ is the complete elliptic integral of the first kind (See \cite{BFBook}*{256.00}).
Hence one obtains
\[
	e^{\frac43 p_3 s(t)} =  \tfrac{ \beta }{ 1- (1-\beta) \cd^2(\frac43 p_3 \sqrt{c_+(\beta-\alpha)} t, \frac{ -\alpha(1-\beta)}{ \beta-\alpha }  ) }.
\]
If $2c_+<c_-$ then
 the three roots satisfies $\alpha<0<\beta=1<\gamma$.
Hence, as long as $1<x<\gamma$, one sees from \cite{BFBook}*{256.00} that
\begin{align*}
	\int^x_1  \frac{d w }{ \sqrt{w(- c_+ w^3+\rho^2w  -c_- ) }}
	={}&\tfrac1{\sqrt{c_+}}	\int^x_1  \frac{d w }{ \sqrt{(w-\alpha)w(w-1)(\gamma-w) }} \\
	={}&\tfrac1{\sqrt{c_+}} \tfrac2{\sqrt{\gamma(1-\alpha)}} \sn^{-1} \(  \sqrt{\tfrac{ \gamma (x-1) }{ (\gamma-1)x }}, 
	{\tfrac{ (\gamma-1)(-\alpha) }{ \gamma(1-\alpha) }}
	\),
\end{align*}
yielding
\[
	e^{\frac43 p_3 s(t)} =  \tfrac{ \gamma }{ \gamma -  (\gamma-1) \sn^2(\frac43 p_3 \sqrt{c_+\gamma(1-\alpha)} \tau, \frac{ (\gamma-1)(-\alpha) }{ \gamma(1-\alpha) } ) }.
\]

\begin{remark}\label{R:Jacobi_criteria}
One sees from the proof that if the left hand side of \eqref{E:Jacobi_criteria} is explicitly integrable for any possible choice of $(\rho,\mathcal{D}(0),\mathcal{R}(0))$, then we obtain
a formula for the solution. 
To obtain a sufficient condition, let us consider the primitive
\[
	\mathfrak{I}=\int \frac{d\tau}{ \sqrt{a - b e^{k_1 \tau} - c e^{k_2 \tau}  }},
\]
where $b,c >0 $, $a\ge b+ c$ and  $k_1 >\max(0, |k_2|)$.
Note that this is easily integrable if $bc=0$. 
Let us exclude this degenerate case and 
consider the general (nontrivial) combination of $a,b,c$. We claim that $\mathfrak{I}$ is described by the elementary functions and the Jacobi elliptic functions if
\[
	\tfrac{k_1+k_2}{k_1-k_2} \in \{ \tfrac13, 1, \tfrac53,  2, \tfrac73, 3,4, 5, 7, 9, 11 \}.
\]
We remark that $k_1 = 2(p_1+p_3)$ and $k_2=2(p_1-p_3)$ imply $p_1/p_3 =\frac{k_1+k_2}{k_1-k_2}$.
Let us prove the claim.
By the change of variable $y=e^{k_0 \tau}$ for some $k_0\neq0$, one has
\[
	\mathfrak{I} = \tfrac{1}{k_0}\int \frac{dy}{ \sqrt{a y^{2} - b y^{k_1/k_0} - c y^{{k_2}/{k_0} } }}.
\]
On the other hand, it is known that the following integral is written in terms of the elementary functions and the Jacobi elliptic functions:
\[
	\int \frac{y^m dy}{\sqrt{P(y)}} ,
\]
where $m \in \Z$ and $P$ is a quartic polynomial. Hence, we see that if there exists $\ell \in 2\Z$ such that
\[
	2, \tfrac{k_1}{k_0}, \tfrac{k_2}{k_0} \in \{ \ell, \ell +1 , \ell +2, \ell + 3, \ell + 4 \}
\]
then $\mathfrak{I}$ is written in terms of the elementary functions and the Jacobi elliptic functions.
Recalling that $k_1 >\max(0, |k_2|)$, the possible choices of $(\ell ,k_1/k_0,k_2/k_0)$ are as follows:
\begin{align*}
	& (-2,1, 0),\,  (-2,2, -1),\, (-2,2, 0),\, (-2,2, 1), \\
	&(0,1, 0),\, (0,2, 0),\, (0,2, 1),\, (0,3, 0),\, (0,3, 1),\, (0,3,2),\, (0,4, 0), \,(0,4, 1),\, (0,4, 2),\, (0,4, 3),\\
	& (2,3, 2),\, (2,4, 2),\, (2,4,3),\, (2,5, 2), \, (2,5 3),\, (2,5, 4),\, (2,6, 2),\, (2,6, 3),\, (2,6, 4),\, (2,6,5).
\end{align*}
Hence, $\frac{k_1+k_2}{k_1-k_2} \in \{ \frac13, 1, \frac53,  2, \frac73, 3,4, 5, 7, 9, 11 \}$.
The claim is proven.
Thus, one can obtain an explicit formula of solutions to \eqref{E:qqq_11} if $p_1/p_3 \in \{ \frac13, 1, \frac53,  2, \frac73, 3,4, 5, 7, 9, 11 \}$.
However, this involves roots of cubic or quartic equations and hence the description of the solution would be complicated.
\end{remark}

\begin{remark}
A phase portrait analysis for \eqref{E:NLS11} shows that the nonlinear synchronization occurs if $p_1>p_3$. It can be verified, for instance, from the fact that $\mathcal{I}$ is strictly monotone decreasing for all non-equilibrium solutions.
On the other hand, if $p_1<p_3$ then there exist six fixed points.
None of them are asymptotically stable.
\end{remark}

\subsubsection{Cases 14 and 15}

We consider the case $p_3^2=p_1^2+p_2^2 $, $p_1>0$ and $p_2\neq0$. Note that $p_3+p_2 > 0$.
We introduce $\Theta \in (0,\pi/2)$ by the relation
$\tan \Theta = \frac{p_1}{p_3+p_2}
= (\tfrac{p_3-p_2}{p_3+p_2})^{1/2}=\frac{p_3-p_2}{p_1}$.
Let $p_4 \ge 0$ and $p_5 = p_4 \tan \Theta $.
This notation gives us a unified treatment of Cases 14 and 15.
Indeed, the cases $p_4=0$ and $p_4>0$ correspond to Cases 14 an 15, respectively.

Let us introduce
\[
	X(t) = \tfrac{1}{2 \sin \Theta} \mathcal{D}(t) + \tfrac{1}{2 \cos \Theta} \mathcal{R}(t) + 
	\tfrac{p_2p_4}{2p_1p_3 \cos \Theta} \rho
\]
and
\[
	Y(t) = -\tfrac{1}{2 \sin \Theta} \mathcal{D}(t) + \tfrac{1}{2 \cos \Theta} \mathcal{R}(t) - 
	\tfrac{p_4}{2p_1 \cos \Theta} \rho.
\]
Then, one sees from \eqref{E:qqq2} that
$\tfrac{d}{ds} X = 0$ and $\tfrac{d}{ds} Y = 2p_1 Y$.
Hence,
\[
	X(t) = X(0) ,\quad Y(t)= Y(0)e^{2p_1 s}.
\]
These yield
\begin{align*}
	\mathcal{D}^2 + \mathcal{R}^2 ={}&
	\ltrans{\begin{bmatrix} X - \tfrac{p_2p_4}{2p_1p_3 \cos \Theta} \rho \\ Y + \tfrac{p_4}{2p_1 \cos \Theta} \rho \end{bmatrix}}
	\ltrans{\begin{bmatrix}
	\frac{1}{2\sin \Theta} & 	\frac{1}{2\cos \Theta} \\
	-\frac{1}{2\sin \Theta} & 	\frac{1}{2\cos \Theta} 
	\end{bmatrix}}^{-1}
	\begin{bmatrix}
	\frac{1}{2\sin \Theta} & 	\frac{1}{2\cos \Theta} \\
	-\frac{1}{2\sin \Theta} & 	\frac{1}{2\cos \Theta} 
	\end{bmatrix}^{-1}
	\begin{bmatrix} X - \tfrac{p_2p_4}{2p_1p_3 \cos \Theta} \rho \\ Y + \tfrac{p_4}{2p_1\cos \Theta} \rho \end{bmatrix} \\
	={}&
 (X - \tfrac{p_2p_4}{2p_1p_3 \cos \Theta} \rho )^2+2 \cos 2\Theta (X - \tfrac{p_2p_4}{2p_1p_3 \cos \Theta} \rho)(Y + \tfrac{p_4}{2p_1\cos \Theta} \rho) + (Y + \tfrac{p_4}{2p_1\cos \Theta} \rho)^2.
\end{align*}
We see that $\cos 2\Theta 
= \frac{1- \tan^2 \Theta}{1+\tan^2 \Theta}= \frac{p_2}{p_3}$. 
Hence,
\[
	\mathcal{Q}(s)= \sqrt{ - Y(0)^2e^{4p_1s} - 2\tfrac{p_2}{p_3}(X(0) + \tfrac{p_1p_4}{2p_2p_3\cos \Theta} \rho) Y(0)e^{2p_1s} + (1-\tfrac{p_4^2}{4p_3^2\cos^2 \Theta})\rho^2-X(0)^2 },
\]
and hence
\[
	2 t \sign \mathcal{I}(0) = \int_0^s \tfrac{d\tau}{ Q(\tau)}=\tfrac1{2p_1}\int_1^{e^{2p_1s}} \tfrac{dx}{ x \sqrt{ - Y(0)^2 x^2 - 2\frac{p_2}{p_3}(X(0) + \frac{p_1p_4}{2p_2p_3\cos \Theta} \rho) Y(0)x + (1- \frac{p_4^2}{4p_3^2\cos^2\Theta})\rho^2-X(0)^2  }} 
\]
for small time.
The right hand side is integrable.
The relation 
\[
	- Y(0)^2  - 2\tfrac{p_2}{p_3}(X(0) + \tfrac{p_1p_4}{2p_2p_3\cos \Theta} \rho) Y(0) + (1- \tfrac{p_4^2}{4p_3^2\cos^2\Theta})\rho^2-X(0)^2 = \mathcal{I}(0)^2 
\]
is useful.
We let
\[
	r= \sqrt{ \left\lvert(1-\tfrac{p_4^2}{4p_3^2 \cos^2 \Theta})\rho^2-X(0)^2\right\rvert}.
\]

If $(1-\frac{p_4^2}{4p_3^2\cos^2 \Theta})\rho^2-X(0)^2>0$ then one has
\[
	e^{2p_1s} = \tfrac{4r^2}{ ((r - \mathcal{I}(0))^2 +Y(0)^2) e^{4p_1r t }  + ((r + \mathcal{I}(0))^2 +Y(0)^2) e^{-4p_1r t} -2\mathcal{I}(0)^2 - 2Y(0)^2 + 2r^2}.
\]
This yields the desired formulas of the solution.

If $\tfrac{4p_3^2-p_4^2}{4p_3^2}\rho^2-X(0)^2=0$ then one obtains
\[
	e^{2p_1s} = \tfrac{\mathcal{I}(0)^2 + Y(0)^2 }{
	(2p_1 t (\mathcal{I}(0)^2 + Y(0)^2)- \mathcal{I}(0))^2 + Y(0)^2}.
\]
This yields the desired formulas.

If $\tfrac{4p_3^2-p_4^2}{4p_3^2}\rho^2-X(0)^2<0$ then one obtains
\[
	e^{2p_1s} = \tfrac{ 2r^2 }{
	 (\mathcal{I}(0)^2 + Y(0)^2+r^2)- 2r \mathcal{I}(0) \sin (4p_1 r t) - ( \mathcal{I}(0))^2 + Y(0)^2-r^2) \cos (4p_1 r t)}.
\]
Note that $\tfrac{4p_3^2-p_4^2}{4p_3^2}\rho^2-X(0)^2<0$ implies $Y(0)\neq0$ and hence the denominator is positive for all time.
This yields the formulas.

\section{Proof of Theorem \ref{P:NS}}\label{S:NS}

\begin{proof}
By \eqref{E:main_asymptotics}, one has
\[
	t^{\frac12} \| ( \gamma_1 u_1 + \gamma_2 u_2 )(t,2t \cdot) \|_{L^\infty (\Omega)} 
	= 2^{-\frac12} 
	\| ( \gamma_1 A_1^+ + \gamma_2 A_2^+ )(\tfrac12 \log t) \|_{L^\infty (\Omega)} + O(t^{-\frac14+\delta})
\]
as $t\to \infty$ for any $\Omega \subset \R$.
Hence, it suffices to establish
\[
	\| ( \gamma_1 A_1^+ + \gamma_2 A_2^+ )(\tau) \|_{L^\infty (\Omega)} \to 0
\]
as $\tau\to\infty$ for the target sets.

Introduce the polar coordinate 
\begin{equation}\label{E:nspf0}
	p_\infty = (\cos \varphi_1, \sin \varphi_1 \cos \varphi_2, \sin \varphi_1 \sin \varphi_2)
\end{equation}
with
$\varphi_1\in [0,\pi]$ and $\varphi_2 \in [0,2\pi)$.
Define
\[
	\gamma_1 = (1-\cos \varphi_1)^{\frac12}, \quad \gamma_2 =-(1+\cos \varphi_1)^{\frac12}e^{-i\varphi_2}.
\]

Pick $(\alpha_1,\alpha_2) \in \C^2$ so that
\[
	(|\alpha_1|^2-|\alpha_2|^2, 2\Re \overline{\alpha_1}\alpha_2,2\Im \overline{\alpha_1}\alpha_2) \in \{k (S^2\setminus\mathfrak{P})\in \R^3; k \ge 0\} 
\]
and let
$(A_1(\tau),A_2(\tau))$ be a solution to \eqref{E:ODE} with $(A_1(0),A_2(0))=(\alpha_1,\alpha_2).$
We claim that
\begin{equation}\label{E:nspf1}
	|\gamma_1 A_1 (\tau)+\gamma_2 A_2(\tau)| \to 0
\end{equation}
as $\tau \to \infty$.
Since it is trivial when $(\alpha_1,\alpha_2)=(0,0)$, we consider the other case.
Let
$\rho=|\alpha_1|^2 + |\alpha_2|^2>0$,
\[
		\mathcal{D}(\tau)=|A_1(\tau)|^2 - |A_2(\tau)|^2,\quad \mathcal{R}(\tau)= 2\Re( \overline{A_1(\tau)}A_2(\tau)),\quad \mathcal{I}(\tau)= 2\Im (\overline{A_1(\tau)}A_2(\tau)).
\]
Since the nonlinear synchronization occurs by assumption and since $(\mathcal{D}(0),\mathcal{R}(0),\mathcal{I}(0))\in S^2_\rho \setminus(\rho\mathfrak{P})$, one has
\[
 (\mathcal{D}(\tau),\mathcal{R}(\tau),\mathcal{I}(\tau)) \to \rho
	p_\infty \in S_\rho^2
\]
as $\tau\to\infty$. 
One verifies from \eqref{E:nspf0} that the convergence reads as
\[
	\rho - \cos \varphi_1 \mathcal{D}(\tau) - \sin \varphi_2 \cos \varphi_2 \mathcal{R}(\tau) - \sin \varphi_2 \sin \varphi_2 \mathcal{I}(\tau) \to 0
\]
as $\tau\to\infty$. 
Note that the following identity holds:
\begin{align*}
	&\rho - \cos \varphi_1 \mathcal{D}(\tau) - \sin \varphi_1 \cos \varphi_2 \mathcal{R}(\tau) - \sin \varphi_1 \sin \varphi_2 \mathcal{I}(\tau) \\
	&=(1-\cos \varphi_1)|A_1(\tau)|^2 +(1+\cos \varphi_1)|A_2(\tau)|^2 \\
	&\quad - 2 \sin \varphi_1( \cos \varphi_2 \Re ( \overline{A_1(\tau)}A_2(\tau))+ \sin \varphi_2 \Im ( \overline{A_1(\tau)}A_2(\tau)) ) \\
	&= |(1-\cos \varphi_1)^{\frac12} A_1(\tau)-(1+\cos \varphi_1)^{\frac12}e^{-i\varphi_2}A_2(\tau)|^2.
\end{align*}
Hence, \eqref{E:nspf1} holds true.

We next claim that
for any $\eps>0$ and
for any closed set 
$\mathfrak{E} \subset  S^2 \setminus \mathfrak{P}$, 
there exists $T=T(\mathfrak{E},\eps)\ge0$ such that if $t\ge T$ then
\begin{equation}\label{E:nspf3}
	\sup_{(\mathcal{D}(0),\mathcal{R}(0),\mathcal{I}(0))\in \mathfrak{E}}
	|(\mathcal{D}(t),\mathcal{R}(t),\mathcal{I}(t))-p_\infty| \le \eps,
\end{equation}
where $|\cdot|$ is the standard Euclidean norm on $\R^3$.
Let us prove the claim.
Suppose it fails.
Then, there exists $\eps_0>0$ and a closed set $\mathfrak{E}_0 \subset S^2 \setminus \mathfrak{P}$
such that there exist $\{t_n\}_n \subset [0,\infty) $, $t_n\to \infty$ as $n\to\infty$,
and $\{q_n\}_n\subset\mathfrak{E}_0$
such that 
the solution 
$(\mathcal{D}_n(t),\mathcal{R}_n(t),\mathcal{I}_n(t))$ to \eqref{E:qqq} given by the initial condition
$(\mathcal{D}_n(0),\mathcal{R}_n(0),\mathcal{I}_n(0))=q_n$
satisfies
\begin{equation}\label{E:nspf4}
	|(\mathcal{D}_n(t_n),\mathcal{R}_n(t_n),\mathcal{I}_n(t_n))-p_\infty| \ge \eps_0
\end{equation}
for all $n$.
Let $U \subset S^2$ be an open neighborhood of
$p_\infty \in S^2$ such that $p \in U$ implies
$|p-p_\infty|<\eps_0$.
Then, for this $U$, one can find an open neighborhood $V\subset S^2$ of $p_\infty \in S^2$
such that the second property of the asymptotic stability holds.
We choose $\delta_0>0$
so that $p\in S^2$ and $|p-p_\infty|<2\delta_0$ imply $p\in V$.
Since $\mathfrak{E}_0$ is  compact, one can find a subsequence of $n$, which is again denoted by $n$, and a point $\tilde{q} \in \mathfrak{E}_0$ such that $q_n \to \tilde{q}$ as $n\to\infty$.
Since $\tilde{q} \in \mathfrak{E}_0 \subset S^2 \setminus \mathfrak{P}$, one has $\omega(\tilde{q}) = \{p_\infty\}$ by the first property of the asymptotic stability.
Hence, there exists $T\ge0$ such that
\[
	|(\mathcal{D}_\infty(t),\mathcal{R}_\infty(t),\mathcal{I}_\infty(t))-p_\infty| < \delta_0
\]
for all $t\ge T$,
where
$(\mathcal{D}_\infty(t),\mathcal{R}_\infty(t),\mathcal{I}_\infty(t))$ is the solution 
 to \eqref{E:qqq} given by the initial condition
$(\mathcal{D}_\infty(0),\mathcal{R}_\infty(0),\mathcal{I}_\infty(0))=\tilde{q}$.
On the other hand, since the nonlinearity of \eqref{E:qqq} is locally Lipschitz continuous,  the solutions \eqref{E:qqq} depend continuously on the data. Hence, one can find $\tilde{\delta}>0$ such that
if $p\in S^2$ satisfies
$|p-\tilde{q}|\le \tilde{\delta}$
then
\[
	|(\mathcal{D}(T),\mathcal{R}(T),\mathcal{I}(T))-(\mathcal{D}_\infty(T),\mathcal{R}_\infty(T),\mathcal{I}_\infty(T))| < \delta_0.
\]
Hence,
there exists $N$ such that if $n\ge N$ then
$|q_n-\tilde{q}| \le \delta$ and consequently
\begin{align*}
	|(\mathcal{D}_n(T),\mathcal{R}_n(T),\mathcal{I}_n(T))-p_\infty| 
	&<
	|(\mathcal{D}_n(T),\mathcal{R}_n(T),\mathcal{I}_n(T))-(\mathcal{D}_\infty(T),\mathcal{R}_\infty(T),\mathcal{I}_\infty(T))| \\
	&\quad+|(\mathcal{D}_\infty(T),\mathcal{R}_\infty(T),\mathcal{I}_\infty(T))-p_\infty| \\
	&< 2\delta_0.
\end{align*}
By the choice of $\delta_0$, this implies that
	$(\mathcal{D}_n(T),\mathcal{R}_n(T),\mathcal{I}_n(T))\in V$
for all $n\ge N$.
Thus, by definitions of $V$ and $U$, one has
\[
	\sup_{t\ge T} |(\mathcal{D}_n(t),\mathcal{R}_n(t),\mathcal{I}_n(t)) - p_\infty| < \eps_0
\]
for all $n\ge N$. This contradicts with \eqref{E:nspf4} since $t_n\to \infty$ as $n\to \infty$.
We have proved the claim \eqref{E:nspf3}.

Let us complete the proof. Pick a closed set $\mathfrak{E} \subset  S^2 \setminus \mathfrak{P}$
and $\eps>0$.
Recall that $|A_1^+(\tau;\xi)|^2 + |A_2^+(\tau;\xi)|^2$ is independent of $\tau$ and hence equals to $\rho(\xi):=|\alpha_1^+(\xi)|^2 + |\alpha_2^+(\xi)|^2$.
We see that
\[
	\sup_{\tau \ge 0}\| ( \gamma_1 A_1^+ + \gamma_2 A_2^+ )(\tau) \|_{L^\infty (\Omega(\mathfrak{E})\cap\{\rho(\xi)\le \eps^2\})} \lesssim \eps.
\]
Pick an arbitrary point $\xi \in \Omega(\mathfrak{E})\cap \{\rho(\xi) \ge \eps^2\}$.
Define $T=T(\mathfrak{E},\eps^2/\norm{\rho}_{L^\infty})\ge0$ so that \eqref{E:nspf3} holds. Since
\[
	\rho(\xi)^{-1} (|\alpha_1^+(\xi)|^2-|\alpha_2^+(\xi)|^2, 2\Re \overline{\alpha_1^+(\xi)}\alpha_2^+(\xi),2\Im \overline{\alpha_1^+(\xi)}\alpha_2^+(\xi)) \in \mathfrak{E},
\]
one sees from the scaling \eqref{E:ODEscale} that
\[
	\sup_{t\ge \rho(\xi)^{-1}T}  |(\mathcal{D}(t),\mathcal{R}(t),\mathcal{I}(t))-\rho(\xi)p_\infty| \le \rho(\xi) \tfrac{\eps^2}{\norm{\rho}_{L^\infty}} \le \eps^2.
\]
Arguing as in the proof of \eqref{E:nspf1}, one sees that this gives us
\[
	\sup_{t\ge \eps^{-2}T} |( \gamma_1 A_1^+ + \gamma_2 A_2^+ )(\tau,\xi)|^2
	\lesssim \eps^2.
\]
Taking supremium with respect to $\xi\in \Omega(\mathfrak{E})\cap \{\rho(\xi) \ge \eps^2\}$,
one obtains
\[
	\sup_{\tau \ge \eps^{-2} T}\| ( \gamma_1 A_1^+ + \gamma_2 A_2^+ )(\tau) \|_{L^\infty (\Omega(\mathfrak{E})\cap\{\rho(\xi)\ge \eps^2\})} \lesssim \eps.
\]
Thus, we reach to the estimate
\[
	\sup_{\tau \ge \eps^{-2} T}\| ( \gamma_1 A_1^+ + \gamma_2 A_2^+ )(\tau) \|_{L^\infty (\Omega(\mathfrak{E}))} \lesssim \eps.
\]
This completes the proof.
\end{proof}

\appendix

\section{Quick review on Jacobi elliptic functions}\label{S:elliptic}

We collect the definition and the basic facts about the Jacobi elliptic functions.
See e.g.\ \cite{BFBook} for more detail.

\subsection{definition and basic relations}
Let us collect definition and basic properties of the Jacobi elliptic functions.
Let $0 \le m \le 1$ be a parameter. 
The Jacobi elliptic functions $\sn u=\sn (u,m)$, $\cn u =\cn (u,m)$, and $\dn u = \dn (u,m)$, which are smooth bounded functions defined on $\R$, and the amplitude function $\am(u,m)$
are given as follows:
When $m<1$, for given $u \in \R $, define $\phi\in \R$ by the relation
\begin{equation}\label{E:Jacdef}
	u = \int_0^\phi \tfrac{d\theta}{\sqrt{1-m \sin^2 \theta}}.
\end{equation}
Note that the integrand is positive and continuous and hence $\phi$ is uniquely determined.
Then, $\sn (u,m)$, $\cn (u,m)$, and $\dn (u,m)$ are defined as
\[
	\sn (u,m) = \sin \phi, \quad
	\cn (u,m) = \cos \phi, \quad
	\dn (u,m) = \sqrt{1-m\sin^2 \phi}, 
\]
respectively. We also let
\[
	\am (u,m)  = \phi.
\]
When $m=1$, we define
\[
	\sn (u,1) = \tanh u, \quad
	\cn (u,1) = \sech u, \quad
	\dn (u,1) = \sech u
\]
for $u\in \R$, respectively.
We let 
\[
	\am(u,1) = \arctan (\sinh u),
\]
which is known as the Gudermannian function.
As is well known, the Jacobi elliptic functions are generalization of trigonometric functions. Indeed, one has
\[
	\sn (u,0) = \sin u, \quad
	\cn (u,0) = \cos u, \quad
	\dn (u,0) = 1.
\]
We further define 
\[
\cd (u,m)=\tfrac{\cn (u,m)}{\dn (u,m)}, \quad \sd (u,m)=\tfrac{\sn (u,m)}{\dn (u,m)},\quad \nd (u,m)=\tfrac1{\dn (u,m)}.
\]

The relations
\[
	\sn^2 (u,m) + \cn^2 (u,m) =1, \quad
	\dn^2 (u,m) + m\sn^2 (u,m) =1
\]
immediately follow by definition.
For any $0 \le m \le 1$, $\sn (u,m)$ is an odd function in $u$ and $\cn (u,m)$ and $\dn (u,m)$ are even functions in $u$.

When $m<1$, $\sn (u,m)$, $\cn (u,m)$, and $\dn (u,m)$ are periodic in $u$. The period of $\sn (u,m)$ and
$\cn (u,m)$ is $4K(m)$ and the period of $\dn (u,m)$ is $2K(m)$, where $K(m)$ is the complete elliptic integral of the first kind, i.e.,
\[
	K(m) := \int_0^{\pi/2} \tfrac{d\theta}{\sqrt{1-m \sin^2 \theta}}.
\]
As for the translation by the quarter or the half of the period,
we have 
\begin{align*}
	\sn (x+K(m),m) ={}& \cd(x,m), &
	\cn (x+K(m),m) ={}& \sqrt{1-m}\sd(x,m), \\
	\sn (x+2K(m),m) ={}& -\sn(x,m), &
	\cn (x+2K(m),m) ={}& -\cn(x,m),
\end{align*}
and
\[
	\dn (x+K(m),m) = \sqrt{1-m}\nd(x,m).
\]
See, e.g., \cite{BFBook}*{122.03}.
We remark that $K(m)$ is increasing in $m$ and that $K(0)=\pi/2$ and $\lim_{m \uparrow 1} K(m) = \I$.



\subsection{On derivatives}
The following identities are known:
\begin{equation}\label{E:Jacobi_d}
\begin{aligned}
	\frac{d}{du} \sn(u,m) {}&= \cn(u,m) \dn(u,m),\\
	\frac{d}{du} \cn(u,m) {}&= -\sn(u,m) \dn(u,m),\\
	\frac{d}{du} \dn(u,m) {}&= -m \sn(u,m) \cn(u,m).
\end{aligned}
\end{equation}
See, e.g., \cite{BFBook}*{731.01,731.02,731.03}.
Further, $\am(u,m)$ is the definite integral of $\dn(u,m)$ from $0$, i.e.,
\[
	\am (u,m) = \int_0^u \dn (v,m) dv.
\]
In particular,
\[
	\frac{d}{du} \am (u,m) = \dn (u,m).
\]



\subsection*{Acknowledgements} 
The author expresses sincere thanks to Jason Murphy for his valuable comments on the preliminary version of the manuscript.
The author was supported by JSPS KAKENHI Grant Numbers JP21H00991 and JP21H00993.

\end{document}